\documentclass[12pt]{amsart}
\usepackage{a4wide}

\newtheorem{theorem}{Theorem}[section]
\newtheorem{lemma}[theorem]{Lemma}

\newtheorem{corollary}[theorem]{Corollary}

\theoremstyle{definition}
\newtheorem{definition}[theorem]{Definition}
\newtheorem{example}[theorem]{Example}

\theoremstyle{remark}
\newtheorem{remark}[theorem]{Remark}

\numberwithin{equation}{section}
\usepackage{epsfig}
\usepackage{color} 
\usepackage{amssymb, amsmath}
\usepackage{mathrsfs}
\usepackage{ulem}

\usepackage{float}
\usepackage{graphicx}

\usepackage{tikz}
\usepackage{tikz-cd}
\usetikzlibrary{snakes}
\usetikzlibrary{intersections, calc}
\usepackage{epsfig}
\usepackage[all]{xy}
\usepackage{epstopdf}
\usepackage{framed}
\usetikzlibrary{decorations.pathreplacing}

\newcommand{\Z}{{\mathbb{Z}}}

\newcommand{\N}{{\mathbb{N}}}

\newcommand{\algt}{\mathfrak{t}}

\begin{document}

\title[GKM graph locally modeled by $T^{n}\times S^{1}$-action on $T^{*}\mathbb{C}^{n}$]{GKM graph locally modeled by $T^{n}\times S^{1}$-action on $T^{*}\mathbb{C}^{n}$ and its graph equivariant cohomology}

\author{Shintar\^o KUROKI}
\address{Okayama University of Science, 
1-1 Ridai-cho Kita-ku Okayama-shi Okayama 700-0005, JAPAN}
\email{kuroki@xmath.ous.ac.jp}
\thanks{The 1st author was supported by JSPS KAKENHI Grant Number 17K14196 and 21K03262.}

\author{Vikraman UMA}
\address{Department of Mathematics, Indian Institute of Technology, Madras, Chennai 600036, INDIA
}
\email{vuma@iitm.ac.in}



\begin{abstract}
We introduce a class of labeled graphs (with legs) which contains two classes of GKM graphs of $4n$-dimensional manifolds with $T^{n}\times S^{1}$-actions, i.e., GKM graphs of  
the toric hyperK${\rm\ddot{a}}$hler manifolds and of the cotangent bundles of toric manifolds.
Under some conditions, 
the graph equivariant cohomology ring of such a labeled graph is computed.
We also give a module basis of the graph equivariant cohomology by using a shelling structure of such a labeled graph, and study their multiplicative structure.
\end{abstract}

\maketitle

\section{Introduction}

A {\it GKM graph} is a labeled graph defined by the special but wide class of manifolds with torus actions, called {\it GKM manifolds}.
From the torus action on a GKM manifold, a GKM graph is defined by its zero and one dimensional orbits together with the labels on edges defined by the tangential representations around fixed points. 
Goresky-Kottwicz-MacPherson in \cite{GKM} show that if a GKM manifold satisfies a certain condition, called {\it equivariant formality}, then its equivariant cohomology is isomorphic to an algebra defined from its GKM graph.
We call this algebra a {\it graph equivariant cohomology} in this paper.  
Motivated by the work of Goresky-Kottwicz-MacPherson, Guillemin-Zara in \cite{GZ} introduce the abstract GKM graph without considering any GKM manifolds, and they translate some geometric properties of GKM manifolds into combinatorial ones of GKM graphs. 
After the works of Guillemin-Zara,  
a GKM graph can be regarded as a combinatorial approximation of space with torus action, and  
it has been studied by some mathematicians, e.g. see \cite{MMP, GHZ, GSZ, FIM, FY, Ku19, DKS}.
In this paper, we introduce a certain class of GKM graphs with legs
and attempt to unify two slightly different classes of manifolds from the GKM theoretical point of view, i.e., 
{\it toric hyperK\"ahler manifolds} and {\it cotangent bundles of toric manifolds}, 
where a {\it leg} is a half-line whose boundary corresponds to the initial vertex. 
We briefly recall toric hyperK\"ahler manifolds and introduce the motivation of the present paper. 
We shall then state our
  main results and outline the organization of this paper.

A {\it toric hyperK\"ahler variety} is defined by the
hyperK${\rm\ddot{a}}$hler quotient of a torus action on the cotangent
bundle $T^{*}\mathbb{C}^{m}$.  This space is introduced by Goto and
Bielwasky-Dancer in \cite{G, BD} as the hyperK${\rm\ddot{a}}$hler
analogue of the symplectic toric manifolds.  The non-singular toric
hyperK${\rm\ddot{a}}$hler varieties are $4n$-dimensional non-compact
manifolds with $T^{n}$-action.  They are completely determined by some
class of hyperplane arrangements in $\mathbb{R}^{n}$ (see \cite{BD})
like symplectic toric manifolds are completely determined by Delzant
polytopes in $\mathbb{R}^{n}$ (see \cite{D}).  The equivariant
topology and geometry of toric hyperK${\rm\ddot{a}}$hler manifolds are
studied by some mathematicians, e.g. \cite{Ko99, Ko00, Ko03, HP, P,
  Ku11}.  In particular, Harada-Proudfoot show that every toric
hyperK${\rm\ddot{a}}$hler manifold admits the residual $S^{1}$-action
and the equivariant cohomology of a toric hyperK${\rm\ddot{a}}$hler
manifold with $T^{n}\times S^{1}$-action is determined by the
half-space arrangements in $\mathbb{R}^{n}$.  They also show that the
toric hyperK${\rm\ddot{a}}$hler manifolds with
$T^{n}\times S^{1}$-actions satisfy the {\it GKM condition}, i.e., its
zero and one dimensional orbits have the structure of a graph.  Note
that the $T^{n}$-action on a toric hyperK${\rm\ddot{a}}$hler manifold
does not satisfy the GKM condition.  Therefore, we can define the
labeled graph (with legs) from toric hyperK${\rm\ddot{a}}$hler
manifolds with $T^{n}\times S^{1}$-actions.  The GKM graph of a toric
hyperK${\rm\ddot{a}}$hler manifold is obtained from the
one-dimensional intersections of hyperplanes like the GKM graph of a
symplectic toric manifold is obtained from the one-skeleton of a
moment-polytope.  By definition, the tangential representations of
$T^{n}\times S^{1}$-actions on the fixed points are isomorphic (up to
automorphism on $T^{n}\times S^{1}$) to the standard $T^{n}$-action on
$T^{*}\mathbb{C}^{n}$ together with the scalar multiplication of
$S^{1}$ on the fiber, i.e.,
\begin{align}
\label{standard-action}
(t_{1},\ldots, t_{n}, r)\cdot (z_{1},\ldots, z_{n}, w_{1},\ldots, w_{n})\mapsto 
(t_{1}z_{1},\ldots, t_{n}z_{n}, rt_{1}^{-1}w_{1},\ldots, rt_{n}^{-1}w_{n}),
\end{align}
where 
\begin{align*}
(t_{1},\ldots, t_{n})\in T^{n},\quad r\in S^{1}, \quad 
(z_{1},\ldots, z_{n},w_{1},\ldots, w_{n})\in T^{*}\mathbb{C}^{n}(\simeq \mathbb{C}^{n}\times \mathbb{C}^{n}).
\end{align*}
We call the action defined by \eqref{standard-action} is the {\it standard $T^{n}\times S^{1}$-action on $T^{*}\mathbb{C}^{n}$}.

On the other hand, the cotangent bundle $T^{*}M$ of a $2n$-dimensional
toric manifold $M$ also has the $T^{n}\times S^{1}$-action.
More precisely, because $T$ acts on $M$ smoothly, each element $t\in T$ induces the diffeomorphism $t:M\to M$, say $t:p\mapsto t\cdot p$ for $p\in M$. By taking its differential $dt:TM \to TM$, we have the lift of the $T$-action on the tangent bundle $TM$. Note that $(dt)_{p}:T_{p}M \to T_{t\cdot p}M$ is the linear isomorphism. Because the cotangent bundle $T^{*}M$ is defined by the bundle over $M$ whose fibres are $T^{*}_{p}M:={\rm Hom}(T_{p}M, \mathbb{R})$. Therefore, for an element $f\in T^{*}_{p}M$, we can define the $t\in T$ action by $t\cdot f:=f\circ (dt)_{p}^{-1}:T_{t\cdot p}M\to \mathbb{R}$. Together with the scalar multiplication by $S^{1}$ on each fibre $T_{p}^{*}M$, we have the $T^{n}\times S^{1}$-action on $T^{*}M$.
It follows from the definition that this also
satisfies the GKM conditions and the tangential representation around
every fixed point in $T^{*}M$ is isomorphic to the standard
$T^{n}\times S^{1}$-action on $T^{*}\mathbb{C}^{n}$.

Note that $T^{*}M$ of a toric manifold $M$ is not a toric
hyperK${\rm\ddot{a}}$hler manifold except in the case when $M$ is a product of some projective spaces, see \cite{BD}.
So the cotangent bundles of toric manifolds and the toric hyperK${\rm\ddot{a}}$hler manifolds are different classes of manifolds.
However, 
it is known that their equivariant cohomologies are quite similar.
The equivariant cohomology $H_{T^{n}}^{*}(M)$ of a symplectic toric manifold $M$ with $T^{n}$-action is isomorphic to the Stanley-Reisner ring of the moment-polytope (see e.g. \cite[Lemma 7.4.34]{BP} for more general class of manifolds with $T^{n}$-actions).
Because there is an equivariant deformation retract from $T^{*}M$ to $M$, we see that 
 $H_{T^{n}}^{*}(T^{*}M)$ is also isomorphic to the Stanley-Reisner ring of a polytope. 
On the other hand, by Konno's theorem \cite{Ko99}, the equivariant cohomology of a toric hyperK${\rm\ddot{a}}$hler manifold with $T^{n}$-action is isomorphic to the Stanley-Reisner ring of hyperplane arrangements. 
Therefore, these distinct classes of manifolds have similar equivariant cohomology ring structures.
So it may be natural to ask whether we can unify these classes of manifolds. 
One answer is that there exists an embedding from a symplectic toric manifold $M$ to a toric hyperK${\rm\ddot{a}}$hler manifold (see \cite{BD, HP}).
In this paper, we answer this question from a different direction, namely, we unify the equivariant cohomologies of these classes by using GKM graphs.

To achieve that, we introduce the class of GKM graphs whose axial functions around vertices are modeled by the standard $T^{n}\times S^{1}$-action on $T^{*}\mathbb{C}^{n}$, 
called a {\it GKM graph locally modeled by $T^{n}\times S^{1}$-action on $T^{*}\mathbb{C}^{n}$} or {\it $T^{*}\mathbb{C}^{n}$-modeled GKM graph} for short in Definition~\ref{def_modeled_graph}. 
This GKM graph behaves like the hyperplane arrangements but does not always come from the hyperplane arrangements, see Section~\ref{sect:3}.
We study the graph equivariant cohomology of $T^{*}\mathbb{C}^{n}$-modeled GKM graphs. 
The first main theorem of this paper is as follows (the technical notions will be introduced in Section~\ref{sect:3} and Section~\ref{sect:4}):
\begin{theorem}[Theorem~\ref{main-theorem1}]
Let $\mathcal{G}$ be a $2n$-valent $T^{*}\mathbb{C}^{n}$-modeled GKM graph and $\mathbf{L}=\{L_{1},\ \cdots,\ L_{m}\}$ be the set of all
hyperplanes in $\mathcal{G}$.
Assume that $\mathcal{G}$ satisfies the following two assumptions:
\begin{enumerate}
\item For each $L \in \mathbf{L}$, there exist the unique pair of the halfspace $H$ and its opposite side $\overline{H}$ such that $H \cap \overline{H}=L$;
\item For every subset $\mathbf{L}'\subset \mathbf{L}$, its
  intersection $\displaystyle\bigcap_{L\in\mathbf{L}'} L$ is empty or connected.
\end{enumerate}
Then the following ring isomorphism holds: 
\begin{align*}
H^{*}(\mathcal{G})\simeq \Z[\mathcal{G}].
\end{align*}
\end{theorem}
To prove this, we introduce an {\it $x$-forgetful graph
  $\mathcal{\widetilde{G}}$} from a $T^{*}\mathbb{C}^{n}$-modeled GKM
graph $\mathcal{G}$ in Section~\ref{sect:5}.  An $x$-forgetful graph
$\mathcal{\widetilde{G}}$ is a labeled graph but not a GKM graph, and
it may be regarded as the combinatorial counterpart of the
$T^{n}$-actions on toric hyperK${\rm\ddot{a}}$hler manifolds and
cotangent bundles over toric manifolds.  We define its graph
equivariant cohomology $H^{*}(\widetilde{\mathcal{G}})$, and prove its
ring structure in Theorem~\ref{main-theorem1-1}.  In
Section~\ref{sect:6}, we give a proof of Theorem~\ref{main-theorem1}
by using Theorem~\ref{main-theorem1-1}.  In Section~\ref{sect:7}, we
also study the $H^{*}(BT^{n})$-module structure of
$H^{*}(\widetilde{\mathcal{G}})$.  As the second main result of this
paper, in Theorem~\ref{module-generators}, we exhibit an 
$H^{*}(BT^{n})$-module basis of $H^{*}(\widetilde{\mathcal{G}})$ by
using the shellablility of a simplicial complex $\Delta_{\mathbf{L}}$
associated to $\mathbf{L}$. Dividing $H^{*}(\widetilde{\mathcal{G}})$
by $H^{>0}(BT^{n})$, we also introduce
$H^{*}_{ord}(\widetilde{\mathcal{G}})$, which corresponds to the
ordinary cohomology of the usual equivariant cohomology.  Then, we
show that the $H^{*}(BT^{n})$-module basis of
$H^{*}(\widetilde{\mathcal{G}})$ induces a $\mathbb{Z}$-module basis
for $H^{*}_{ord}(\widetilde{\mathcal{G}})$. Finally, in the case when
$\widetilde{\mathcal{G}}$ corresponds to the line arrangements in
$\mathbb{R}^{2}$ (which corresponds geometrically to the
$8$-dimensional toric hyperK${\rm\ddot{a}}$hler manifolds), we
describe the structure constants of
$H^{*}_{ord}(\widetilde{\mathcal{G}})$ with respect to this basis.


\section{GKM graph locally modeled by $T^{n}\times S^{1}$-action on $T^{*}\mathbb{C}^{n}$}

This section aims to define a 
GKM graph with legs and its graph equivariant cohomology.
In particular, we introduce {\it 
GKM graphs locally modeled by $T^{n}\times S^{1}$-action on $T^{*}\mathbb{C}^{n}$} as the special class of  GKM graphs with legs.


\subsection{Notations}
\label{sect:2.1}
We first prepare some notations.
In this paper $\Gamma$ is a connected graph which possibly has legs,
where a \textit{leg} means an outgoing half-line from one vertex (see the left graph in the Figure \ref{1st_example}).
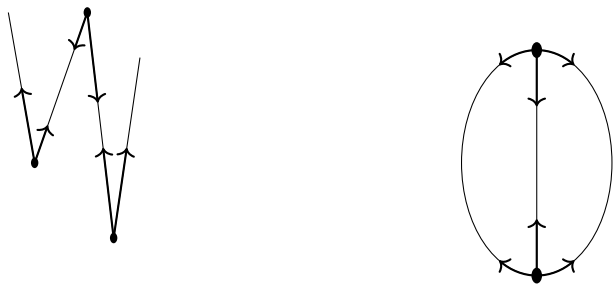
\begin{figure}[h]
\begin{tikzpicture}
\begin{scope}[xscale=0.7, yscale=1.0]
\fill (0,2)  coordinate (p) circle (2pt); 
\fill (-1,0)  coordinate (q) circle (2pt); 
\fill (0.5,-1)  coordinate (r) circle (2pt); 
\fill (-1.5,2) coordinate (s);
\fill (1,1.4) coordinate (t);

\draw (p)--(q);
\draw (p)--(r);
\draw (q)--(s);
\draw (r)--(t);

\draw [->, thick] (p)--(-0.25,1.5);
\draw[->, thick] (q)--(-0.75,0.5);

\draw [->, thick] (p)--(0.2,0.8);
\draw [->, thick] (r)--(0.3,0.2);

\draw[->, thick] (q)--(-1.25,1);

\draw[->, thick] (r)--(0.75,0.2);


\end{scope}

\begin{scope}[xshift=170, xscale=1.0, yscale=1.5]
\fill (0,-1)  coordinate (p) circle (2pt); 
\fill (0,1)  coordinate (q) circle (2pt); 

\draw (0,-1)--(0,1);
\draw (0,0) circle [radius=1];

\draw [->, thick] (0,-1) arc (270:240:1);
\draw [->, thick] (0,-1)--(0,-0.5);
\draw [->, thick] (0,-1) arc (270:300:1);

\draw [->, thick] (0,1) arc (90:120:1);
\draw [->, thick] (0,1)--(0,0.5);
\draw [->, thick] (0,1) arc (90:60:1);
\end{scope}
\end{tikzpicture}
\caption{These are examples of regular graphs with legs and
  orientations.  The left $2$-valent graph has two legs, on the other
  hand the right $3$-valent graph has no legs. Note that all edges
  have two orientations and all legs have only one orientation.}
\label{1st_example}
\end{figure}

We define a graph with legs more precisely. 
Let $\mathcal{V}$ be a set of vertices, $E$ be a set of edges  and $Leg$ be a set of legs in $\Gamma$, 
and $\mathcal{E}=E\cup Leg$. 
The graph $\Gamma$ is denoted by 
\begin{eqnarray*}
\Gamma=(\mathcal{V},\ \mathcal{E}).
\end{eqnarray*}
In this paper, we assume that $\mathcal{V}$ and $\mathcal{E}$ are finite sets. 
We also assume that $\Gamma$ is an oriented graph.
For $\epsilon\in E$, 
we denote by $i(\epsilon)$ and $t(\epsilon)$ the initial vertex and the terminal vertex of $\epsilon$, respectively.
We denote the opposite directed edge of $\epsilon$ as $\overline{\epsilon}$, i.e., 
$i(\overline{\epsilon})=t(\epsilon)$ and $t(\overline{\epsilon})=i(\epsilon)$.
For $\ell\in Leg$, there is no terminal vertex but there exists an initial vertex $i(\ell)$.
Note that the leg in $\Gamma$ can be characterized by the element $\epsilon$ in $\mathcal{E}$ such that there is no $\overline{\epsilon}$.
For a vertex $p\in \mathcal{V}$, we put the set of all outgoing edges and legs from $p\in \mathcal{V}$ by
\begin{eqnarray*}
\mathcal{E}_{p}=\{\epsilon \in \mathcal{E}\ |\ i(\epsilon)=p\}.
\end{eqnarray*}
Assume that $|\mathcal{E}_{p}|=m$ for all $p\in \mathcal{V}$,
where the symbol $|X|$ represents the cardinality of the finite set $X$. 
We call such a graph a {\it (regular) $m$-valent graph}, see Figure~\ref{1st_example}. 

Let $\Gamma=(\mathcal{V},\mathcal{E})$ be a graph with legs.
We denote a subgraph of $\Gamma$ by $G=(\mathcal{V}^{G},\ \mathcal{E}^{G})$, that is, $G$ satisfies
$\mathcal{V}^{G}\subset \mathcal{V}$ and $\mathcal{E}^{G}\subset \mathcal{E}$.
We use the following symbols.
\begin{itemize}
\item $\mathcal{E}^{G}_{p}$: the set of all outgoing edges and legs in $\mathcal{E}^{G}$ from $p\in \mathcal{V}^{G}$.
\item $E^{G}\subset \mathcal{E}^{G}$ (resp.~$E^{G}_{p}\subset \mathcal{E}^{G}_{p}$): the set of all edges  (resp.~out going from $p$) in $G$, i.e., if $\epsilon\in E^{G}$, then the both $i(\epsilon),t(\epsilon)\in \mathcal{V}^{G}$.
\item $Leg^{G}\subset \mathcal{E}^{G}$ (resp.~$(Leg)_{p}^{G}\subset \mathcal{E}^{G}_{p}$): the set of all legs  (resp.~out going from $p$) in $G$.
\end{itemize}
Note the following remark about legs in $G$. 
\begin{remark}
Because $\Gamma$ is an oriented graph, 
we may consider the subgraph $G=(\mathcal{V}^{G},\ \mathcal{E}^{G})$ of $\Gamma=(\mathcal{V},\mathcal{E})$ such that there exists a leg $\epsilon\in Leg^{G}$ in $G$ which is an edge $\epsilon\in \mathcal{E}$ in $\Gamma$.
In other words, $t(\epsilon)\not\in \mathcal{V}^{G}$ but $t(\epsilon)\in \mathcal{V}$; or 
$\overline{\epsilon}\not\in \mathcal{E}^{G}$ but $\overline{\epsilon}\in \mathcal{E}$. 
\end{remark}


\subsection{GKM graph with legs and its graph equivariant cohomology}
\label{sect:2.2}
In this section, 
we shall define a \textit{GKM graph} ({\it (possibly) with legs}) and its \textit{graph equivariant cohomology}.

Let $\Gamma=(\mathcal{V},\ \mathcal{E})$ be an $m$-valent graph.
We first prepare the following notations.
Let 
$\nabla=\{\nabla_{\epsilon}\ |\ \epsilon\in E\}$ be a collection of bijective maps
\begin{eqnarray*}
\nabla_{\epsilon}:\mathcal{E}_{i(\epsilon)}\to \mathcal{E}_{t(\epsilon)}
\end{eqnarray*}
for all edges $\epsilon\in E$.
A {\it connection} on $\Gamma$ is a set $\nabla=\{\nabla_{\epsilon}|\ \epsilon\in E\}$ which satisfies the following two conditions:
\begin{itemize}
\item $\nabla_{\bar{\epsilon}}=\nabla_{\epsilon}^{-1}$;
\item $\nabla_{\epsilon}(\epsilon)=\bar{\epsilon}$.
\end{itemize}
We can easily check that an $m$-valent graph $\Gamma$ admits different
$((m-1)!)^{g}$ connections, where $g$ is the number of (unoriented)
edges $E$.  

Let $T^{n}$ be an $n$-dimensional torus.  In particular,
we often denote a $1$-dimensional torus by $S^{1}$.  If we do not
emphasis the dimension of $T^{n}$, then we denote it by $T$.  Let
$\algt$ be a Lie algebra of $T$, $\algt_{\Z}$ be the lattice of
$\algt$ and $\algt^{*}$ (resp.\ $\algt_{\Z}^{*}$) be the dual of
$\algt$ (resp.\ $\algt_{\Z}$).  The symbol ${\rm Hom}(T,S^{1})$
represents a set of all homomorphisms from the torus $T$ to $S^{1}$.
It is well-known that ${\rm Hom}(T^{n},S^{1})\simeq \mathbb{Z}^{n}$.
Moreover, it may be regarded as $\algt^{*}_{\Z}$ and
$H^{1}(T)\simeq H^{2}(BT)$, where $BT$ is the classifying space of
$T$.  In this paper, if we omit the coefficient of the cohomology,
then it means the cohomology with integer coefficients.  Therefore, we
have the identification
\begin{align*}
{\rm Hom}(T,S^{1})\simeq \algt^{*}_{\Z}\simeq H^{2}(BT).
\end{align*}
Define an \textit{axial function} by the function 
\begin{eqnarray*}
\alpha:\mathcal{E}\longrightarrow H^{2}(BT)
\end{eqnarray*}
such that it satisfies the following three conditions:
\begin{itemize}
\item $\alpha(\bar{\epsilon})=\pm \alpha(\epsilon)$ for all edges $\epsilon\in E$;
\item
  $\alpha(\mathcal{E}_{p})=\{\alpha(\epsilon)\ |\
  \epsilon\in\mathcal{E}_{p}\}$ are \textit{pairwise linearly
    independent} for all $p\in \mathcal{V}$, that is, for every two
  distinct elements $\epsilon_{1},\ \epsilon_{2}\in \mathcal{E}_{p}$,
  $\alpha(\epsilon_{1}),\ \alpha(\epsilon_{2})$ are linearly
  independent in $H^{2}(BT)$;
\item there is a connection $\nabla$ which satisfies the following 
\textit{congruence relation} for all edges $\epsilon\in E$:
\begin{align*}
\alpha(\epsilon')-\alpha(\nabla_{\epsilon}(\epsilon'))\equiv 0 ({\rm mod}\ \alpha(\epsilon))
\end{align*}
for all $\epsilon' \in \mathcal{E}_{i(\epsilon)}$.
\end{itemize}


\begin{definition}[GKM graph with legs]
Let $\mathcal{G}=(\Gamma,\ \alpha,\ \nabla)$ be a collection of an $m$-valent graph  $\Gamma=(\mathcal{V},\ \mathcal{E})$,
where  
the map  
\begin{eqnarray*}
\alpha:\mathcal{E}\longrightarrow H^{2}(BT^{n}),
\end{eqnarray*}
is an axial function ($n\le m$), and $\nabla$ is a connection on
$\Gamma$.  We call $\mathcal{G}=(\Gamma,\ \alpha,\ \nabla)$ a
\textit{GKM graph (with legs)}.
\end{definition}

\begin{remark}
\label{uniqueness}
Suppose that $\mathcal{E}_{p}$ satisfies the $3$-linearly independent
condition for all $p\in \mathcal{V}$, i.e., for every distinct three
elements $\epsilon_{1},\epsilon_{2},\epsilon_{3}\in \mathcal{E}_{p}$
the axial function
$\alpha(\epsilon_{1}),\alpha(\epsilon_{2}),\alpha(\epsilon_{3})$ are
linearly independent.  Then, by the similar proof for the cases of GKM
graphs without legs in \cite{GZ}, if there exists a connection
$\nabla$, then the connection $\nabla$ is unique.  In particular, if
the GKM graph $\mathcal{G}$ satisfies the $3$-linearly independent
condition, then for any two edges (or legs) $\epsilon, \epsilon'$ in $E_{p}$ we can
determine the $2$-valent GKM subgraph which contains $\epsilon,\ \epsilon'$.
Hence, we often omit the connection $\nabla$ in the GKM graph $\mathcal{G}$.
Namely, we often denote the GKM graph by $\mathcal{G}=(\Gamma,\alpha)$ if the connection is obviously determined by the context.
\end{remark}

Due to the theory of toric hyperK{$\ddot{\rm a}$}hler varieties (see \cite{BD, HP, Ko99}), 
the tangential representation on each fixed point is isomorphic to the $T^{n}$-action on $T^{*}\mathbb{C}^{n}$ ($\simeq \mathbb{H}^{n}$, i.e., the $n$-dimensional quaternionic space) which is defined by the strandard $T^{n}$-action on $\mathbb{C}^{n}$, i.e.,
\begin{align*}
(t_{1},\ldots, t_{n})\cdot (z_{1},\ldots, z_{n}, w_{1},\ldots, w_{n}):=
(t_{1}z_{1},\ldots, t_{n}z_{n}, t_{1}^{-1}w_{1},\ldots, t_{n}^{-1}w_{n}),
\end{align*}
where $(t_{1},\ldots, t_{n})\in T^{n}$, $z=(z_{1},\ldots, z_{n})\in\mathbb{C}^{n}$ and $(w_{1},\ldots, w_{n})\in T^{*}_{z}(\mathbb{C}^{n})$.
On the other hand, Harada-Proudfoot \cite{HP} found that 
there exists the residual $S^{1}$-action on the toric hyperK{$\ddot{\rm a}$}hler varieties and this action fits into the GKM theory. 
In this case, the tangential representation on each fixed point may be regarded as the $T^{n}\times S^{1}$-action on $T^{*}\mathbb{C}^{n}$, i.e., 
\begin{align*}
(t_{1},\ldots, t_{n},r)\cdot (z_{1},\ldots, z_{n}, w_{1},\ldots, w_{n}):=
(t_{1}z_{1},\ldots, t_{n}z_{n}, rt_{1}^{-1}w_{1},\ldots, rt_{n}^{-1}w_{n}),
\end{align*}
where $r\in S^{1}$.
Therefore, the toric hyperK{$\ddot{\rm a}$}hler manifolds with
$T^{n}\times S^{1}$-actions induce the GKM graphs with legs whose
axial functions around every
vertex are isomorphic to 
\begin{align*}
\{e_{1}^{*},\ldots, e_{n}^{*}, -e_{1}^{*}+x,\ldots, -e_{n}^{*}+x\},
\end{align*} 
where $e_{1}^{*},\ldots, e_{n}^{*}$ is a generator of a rank $n$ subspace in $H^{2}(BT^{n}\times BS^{1})\simeq \mathbb{Z}^{n}\oplus \mathbb{Z}$ and $x$ is a generator of $H^{2}(BS^{1})\simeq \mathbb{Z}$.
By defining this abstractly, we introduce the following notion:


\begin{definition}[GKM graph locally modeled by $T^{n}\times S^{1}$-action on $T^{*}\mathbb{C}^{n}$]
\label{def_modeled_graph}
Let $\mathcal{G}=(\Gamma,\ \alpha,\ \nabla)$ be a 
$2n$-valent GKM graph with legs with an axial function 
\begin{align*}
\alpha:\mathcal{E}\longrightarrow H^{2}(BT^{n}\times BS^{1})\simeq \algt^{*}_{\Z}\oplus \Z x,
\end{align*}
where $x$ is a generator of the dual of the Lie algebra of $S^{1}$.
We call $\mathcal{G}=(\Gamma,\ \alpha,\ \nabla)$ a \textit{GKM graph modeled by the $T^{n}\times S^{1}$-action on $T^{*}\mathbb{C}^{n}$} (or simply a {\it $T^{*}\mathbb{C}^{n}$-modeled GKM graph}) if it satisfies the following conditions for all $p\in\mathcal{V}$:
\begin{enumerate}
\item We can divide 
$\mathcal{E}_{p}$ into $\{\epsilon_{1}^{+},\ \cdots,\ \epsilon_{n}^{+},\ \epsilon_{1}^{-},\ \cdots,\ \epsilon_{n}^{-}\}$
such that 
\begin{align*}
\alpha(\epsilon_{j}^{+})+\alpha(\epsilon_{j}^{-})=x
\end{align*}
for all $j=1,\ \cdots,\ n$;
\item The set $\{\alpha(\epsilon_{j}^{+}),\ x\ |\ j=1,\cdots ,n\}$
spans $\algt^{*}_{\Z}\oplus \Z x$, i.e., 
\begin{align*}
\langle \alpha(\epsilon_{1}^{+}),\ \cdots,\ \alpha(\epsilon_{n}^{+}),\ x\rangle=\algt^{*}_{\Z}\oplus \Z x.
\end{align*}
\end{enumerate}
We call $\{ \epsilon_{j}^{+},\ \epsilon_{j}^{-}\}$ such that 
$\alpha(\epsilon_{j}^{+})+\alpha(\epsilon_{j}^{-})=x$ a \textit{$1$-dimensional pair} in $\mathcal{E}_{p}$.
Furthermore, we call an element $x$ a {\it residual basis}.
\end{definition}

Figure~\ref{Figure2} shows some examples of $T^{*}\mathbb{C}^{n}$-modeled GKM graphs.
\begin{figure}[h]
\begin{tikzpicture}
\begin{scope}[xscale=1.0, yscale=1.0]
\fill (-1,1) coordinate (r) circle (2pt);
\fill (2,-2) coordinate (q) circle (2pt); 
\fill (-1,-2) coordinate (p) circle (2pt); 

\draw (-2,2)--(3,-3);
\draw (-1,2)--(-1,-3);
\draw (-2,-2)--(3,-2);

\draw[->, thick] (p)--(-0.5,-2);
\node[below] at (-0.2,-2) {$e_{2}^{*}$};
\draw[->, thick] (p)--(-1,-1.5);
\node[left] at (-1,-1.2) {$e_{1}^{*}$};
\draw[->, thick] (p)--(-1.5,-2);
\node[above] at (-2,-2) {$x-e_{2}^{*}$};
\draw[->, thick] (p)--(-1,-2.5);
\node[left] at (-1,-2.8) {$x-e_{1}^{*}$};

\draw[->, thick] (q)--(1.5,-2);
\node[below] at (1.2,-2) {$-e_{2}^{*}$};
\draw[->, thick] (q)--(2.5,-2);
\node[above] at (2.5,-2) {$e_{2}^{*}+x$};
\draw[->, thick] (q)--(2.5,-2.5);
\node[below] at (2.5,-2.5) {$e_{2}^{*}-e_{1}^{*}+x$};
\draw[->, thick] (q)--(1.5,-1.5);
\node[left] at (1.5,-1.5) {$e_{1}^{*}-e_{2}^{*}$};

\draw[->, thick] (r)--(-1,0.5);
\node[left] at (-1,0.3) {$-e_{1}^{*}$};
\draw[->, thick] (r)--(-1,1.5);
\node[right] at (-1,1.5) {$e_{1}^{*}+x$};
\draw[->, thick] (r)--(-1.5,1.5);
\node[left] at (-1.6,1.5) {$e_{1}^{*}-e_{2}^{*}+x$};
\draw[->, thick] (r)--(-0.5,0.5);
\node[right] at (-0.5,0.5) {$e_{2}^{*}-e_{1}^{*}$};

\end{scope}

\begin{scope}[xshift=200, xscale=1.0, yscale=1.0]
\fill (2,0) coordinate (p) circle (2pt);
\fill (2,2) coordinate (q) circle (2pt);
\fill (0,2) coordinate (r) circle (2pt);
\fill (-2,2) coordinate (s) circle (2pt);
\fill (-2,0) coordinate (t) circle (2pt);
\fill (-2,-2) coordinate (u) circle (2pt);
\fill (0,-2) coordinate (v) circle (2pt);
\fill (2,-2) coordinate (w) circle (2pt);

\draw (-3,2)--(3,2);
\draw (-3,0)--(-1,0);
\draw (1,0)--(3,0);
\draw (-3,-2)--(3,-2);

\draw (-2,3)--(-2,-3);
\draw (0,3)--(0,1);
\draw (0,-3)--(0,-1);
\draw (2,3)--(2,-3);

\draw[->, thick] (p)--(2.5,0);
\node[below] at (2.9,0) {$e_{2}^{*}+2x$};
\draw[->, thick] (p)--(1.5,0);
\node[above] at (1.1,0) {$-e_{2}^{*}-x$};
\draw[->, thick] (p)--(2,0.5);
\node[right] at (2,0.7) {$e_{1}^{*}$};
\draw[->, thick] (p)--(2,-0.5);
\node[left] at (2,-0.7) {$x-e_{1}^{*}$};

\draw[->, thick] (q)--(2.5,2);
\node[below] at (2.9,2) {$e_{2}^{*}+2x$};

\draw[->, thick] (r)--(0,2.5);
\node[left] at (0,2.7) {$e_{1}^{*}+x$};
\draw[->, thick] (r)--(0,1.5);
\node[left] at (0,1.3) {$-e_{1}^{*}$};


\draw[->, thick] (t)--(-1.5,0);
\node[below] at (-1.1,0) {$e_{2}^{*}$};
\draw[->, thick] (t)--(-2.5,0);
\node[above] at (-2.9,0) {$x-e_{2}^{*}$};
\draw[->, thick] (t)--(-2,0.5);
\node[right] at (-2,0.5) {$e_{1}^{*}$};
\draw[->, thick] (t)--(-2,-0.5);
\node[left] at (-2,-0.7) {$x-e_{1}^{*}$};

\draw[->, thick] (v)--(0,-1.5);
\node[left] at (0,-1.3) {$e_{1}^{*}-x$};
\draw[->, thick] (v)--(0,-2.5);
\node[left] at (0,-2.7) {$2x-e_{1}^{*}$};

\draw[->, thick] (w)--(2.5,-2);
\node[below] at (2.9,-2) {$e_{2}^{*}+2x$};

\end{scope}
\end{tikzpicture}
\caption{$T^{*}\mathbb{C}^{2}$-modeled GKM graphs, where $\langle {\rm e_{1}^{*},\ e_{2}^{*}} \rangle\simeq (\algt^{2})^{*}_{\Z}$. In the left and the right figures, we assume $\alpha(\epsilon)=-\alpha(\overline{\epsilon})$ and omit some axial functions which are automatically determined by the definition.}
\label{Figure2}
\end{figure}
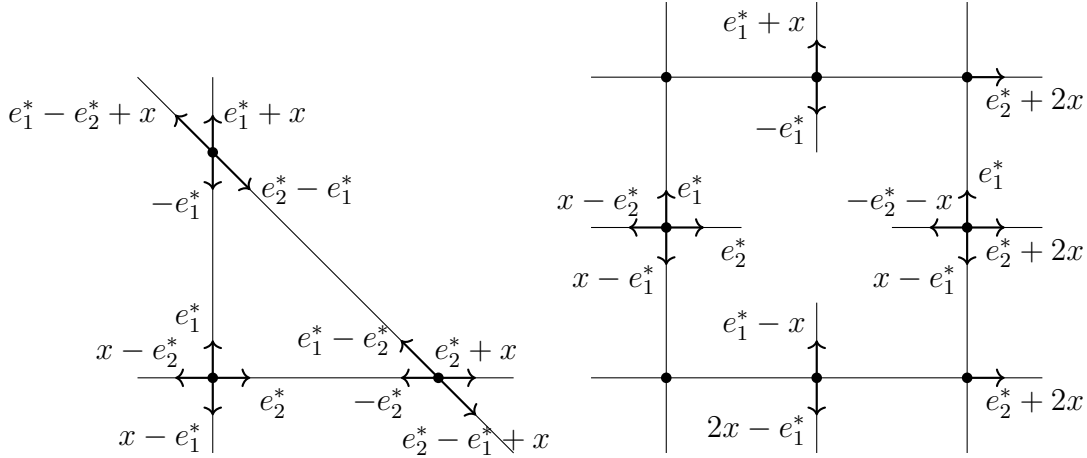

\begin{remark}
  Note that the axial function on $T^{*}\mathbb{C}^{n}$-modeled GKM
  graphs satisfies the $3$-linearly independent condition for all
  vertices.  Therefore, the connection on a
  $T^{*}\mathbb{C}^{n}$-modeled GKM graph $\mathcal{G}$ is uniquely
  determined and we may denote it by \begin{align*}
    \mathcal{G}=(\Gamma,\alpha).
\end{align*}
\end{remark}

We also have the following lemma for the $1$-dimensional pair $\{ \epsilon^{+},\ \epsilon^{-}\}$.
\begin{lemma}
\label{lem_about_pair}
Let $\{\epsilon^{+},\ \epsilon^{-}\}$ be a $1$-dimensional pair in $\mathcal{E}_{i(\epsilon')}$ for some edge $\epsilon'\in E$.
Then $\{\nabla_{\epsilon'}(\epsilon^{+}),\ \nabla_{\epsilon'}(\epsilon^{-})\}$ is also a $1$-dimensional pair in $\mathcal{E}_{t(\epsilon')}$.
\end{lemma}
\begin{proof}
We first divide the edges and legs $\mathcal{E}_{i(\epsilon')}$ by the $1$-dimensional pairs as follows:
\begin{align*}
\mathcal{E}_{i(\epsilon')}=\{\epsilon^{+}_{1},\ \epsilon^{-}_{1}\}\cup \cdots \cup \{\epsilon^{+}_{n},\ \epsilon^{-}_{n}\}
\end{align*}
Because the axial function $\alpha$ satisfies the congruence relation, there are integers $k^{+}_{j}$ and $k^{-}_{j}$, $j=1,\ldots, n$, such that 
\begin{align*}
& \alpha(\nabla_{\epsilon'}(\epsilon_{j}^{+}))-\alpha(\epsilon_{j}^{+})=k^{+}_{j}\alpha(\epsilon'), \\
& \alpha(\nabla_{\epsilon'}(\epsilon_{j}^{-}))-\alpha(\epsilon_{j}^{-})=k^{-}_{j}\alpha(\epsilon').
\end{align*}
Since $\{\epsilon_{j}^{+},\ \epsilon_{j}^{-}\}$ is a $1$-dimensional pair, we also have 
\begin{align}
\label{equation-1-dim_pair}
& (\alpha(\nabla_{\epsilon'}(\epsilon_{j}^{+}))-\alpha(\epsilon_{j}^{+}))+
(\alpha(\nabla_{\epsilon'}(\epsilon_{j}^{-}))-\alpha(\epsilon_{j}^{-})) \\
= &
\alpha(\nabla_{\epsilon'}(\epsilon_{j}^{+}))+\alpha(\nabla_{\epsilon'}(\epsilon_{j}^{-}))-x  \nonumber \\
= & (k_{j}^{+}+k_{j}^{-})\alpha(\epsilon'). \nonumber
\end{align}
In order to show the statement, it is enough to show that 
the following equation holds:
\begin{align*}
k_{j}^{-}=-k_{j}^{+}.
\end{align*}
Suppose on the contrary that $k_{j}^{-}\not=-k_{j}^{+}$.
Then, by \eqref{equation-1-dim_pair},
we have 
\begin{align*}
\alpha(\nabla_{\epsilon'}(\epsilon_{j}^{-}))= -\alpha(\nabla_{\epsilon'}(\epsilon_{j}^{+}))+x+(k_{j}^{+}+k_{j}^{-})\alpha(\epsilon')\not=-\alpha(\nabla_{\epsilon'}(\epsilon_{j}^{+}))+x.
\end{align*}
This implies that $\{\nabla_{\epsilon'}(\epsilon_{j}^{+}),\ \nabla_{\epsilon'}(\epsilon_{j}^{-})\}$ is not a $1$-dimensional pair.
Therefore, there is another element 
$\epsilon(\not=\nabla_{\epsilon'}(\epsilon_{j}^{-}))$ in $\mathcal{E}_{t(\epsilon')}$
such that 
\begin{align*}
\alpha(\epsilon)=-\alpha(\nabla_{\epsilon'}(\epsilon_{j}^{+}))+x.
\end{align*}
This gives that  
\begin{align*}
\alpha(\nabla_{\epsilon'}(\epsilon_{j}^{-}))&=\alpha(\epsilon)+(k_{j}^{+}+k_{j}^{-})\alpha(\epsilon').
\end{align*}
However, since
$\{\nabla_{\epsilon'}(\epsilon_{j}^{-}),\epsilon,\overline{\epsilon'}\}\subset
\mathcal{E}_{t(\epsilon')}$, this is a contradiction to the fact that
$T^{*}\mathbb{C}^{n}$-modeled GKM graph is always $3$-linearly
independent.  Hence, we must have $k_{j}^{-}=-k_{j}^{+}$.  This establishes
the statement.
\end{proof}

Finally, in this section, we also define the notion of {\it graph equivariant cohomology}.
Let $\mathcal{G}=(\Gamma,\ \alpha,\ \nabla)$ be a GKM graph (with 
legs) such that $\alpha:\mathcal{E}\to H^{2}(BT)$.
With the definition similar to that of the GKM graph without legs, the graph equivariant cohomology is defined as follows.
\begin{definition}[graph equivariant cohomology]\label{gec}
The following ring is called a {\it graph equivariant cohomology} of  $\mathcal{G}$:
\begin{eqnarray*}
H^{*}(\mathcal{G})=\{\varphi:\mathcal{V}\to H^{*}(BT)\ |\ \varphi(i(\epsilon))-\varphi(t(\epsilon))\equiv 0\ ({\rm mod}\ \alpha(\epsilon))\},
\end{eqnarray*}
We call the relation $\varphi(i(\epsilon))-\varphi(t(\epsilon))\equiv 0\ ({\rm mod}\ \alpha(\epsilon))$ a {\it congruence relation} of $\varphi$ on an edge $\epsilon\in E$.  
\end{definition}

\section{Some notions of $T^{*}\mathbb{C}^{n}$-modeled GKM graphs}
\label{sect:3}

Let $\mathcal{G}=(\Gamma,\alpha,\nabla)$ be a $2n$-valent
$T^{*}\mathbb{C}^{n}$-modeled GKM graph such that
$\alpha:\mathcal{E}\to H^{2}(BT^{n})\oplus \mathbb{Z}x$, where $x$ is
a residual basis.  The goal of this paper is to compute
$H^{*}(\mathcal{G})$ for a certain class of
$T^{*}\mathbb{C}^{n}$-modeled GKM graphs.  To do that we prepare some
notions and properties of $T^{*}\mathbb{C}^{n}$-modeled GKM graphs.
We first introduce the following notion.

We define a {\it GKM subgraph} $\mathcal{H}=(H,\alpha^{H},\nabla^{H})$ by a $k$-valent subgraph $H$ of $\Gamma$ such that the axial function is defined by 
\begin{align*}
\alpha^{H}:=\alpha|_{\mathcal{E}^{H}},
\end{align*}
and the connection is defined by 
\begin{align*}
\nabla^{H}:=\{\nabla_{\epsilon}|_{\mathcal{E}^{H}}\ |\ \epsilon\in E^{H}\}.
\end{align*}
We also denote $\nabla^{H}_{\epsilon}:=\nabla_{\epsilon}|_{\mathcal{E}^{H}}$ for an edge $\epsilon\in E^{H}$.

\subsection{Hyperplane}
\label{sect:3.1}

In this section, we introduce the notion of a hyperplane in $\mathcal{G}$ and show a key property Lemma~\ref{key-properties-hyperplane} which will be used to show the main theorem of this paper.

\begin{definition}[hyperplane]
\label{hyperplane}
Let $\mathcal{G}$ be a $T^{*}\mathbb{C}^{n}$-modeled GKM graph.
Assume that a GKM subgraph $\mathcal{L}=(L,\alpha^{L},\nabla^{L})$ of
$\mathcal{G}$ is a $(2n-2)$-valent subgraph of $\Gamma$ and it is a
$T^{*}\mathbb{C}^{n-1}$-modeled GKM graph with the residual basis $x$,
i.e., there are $1$-dimensional pairs
$\{\epsilon_{j}^{+},\epsilon_{j}^{-}\}\subset \mathcal{E}_{p}^{L}$,
$j=1,\ldots, n-1$, on each vertex such that
\begin{align*}
\alpha^{L}(\epsilon_{j}^{+})+\alpha^{L}(\epsilon_{j}^{-})=x.
\end{align*}
Such a GKM subgraph $\mathcal{L}$ is said to be a {\it hyperplane} if $L$ is a maximal $(2n-2)$-valent connected subgraph in $\Gamma$, i.e., if $L'$ is a $(2n-2)$-valent connected subgraph in $\Gamma$ such that $L\subset L'$ then $L=L'$.
\end{definition}

\begin{example}
The following two figures show an example and a non-example of hyperplanes. 
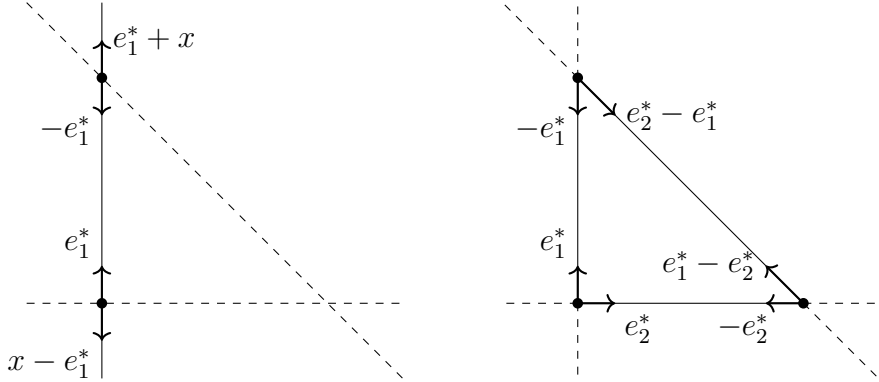
\begin{figure}[h]
\begin{tikzpicture}
\begin{scope}[xscale=1.0, yscale=1.0]
\fill (-1,1) coordinate (r) circle (2pt);
\fill (-1,-2) coordinate (p) circle (2pt); 

\draw[dashed] (-2,2)--(3,-3);
\draw (-1,2)--(-1,-3);
\draw[dashed] (-2,-2)--(3,-2);

\draw[->, thick] (p)--(-1,-1.5);
\node[left] at (-1,-1.2) {$e_{1}^{*}$};
\draw[->, thick] (p)--(-1,-2.5);
\node[left] at (-1,-2.8) {$x-e_{1}^{*}$};

\draw[->, thick] (r)--(-1,0.5);
\node[left] at (-1,0.3) {$-e_{1}^{*}$};
\draw[->, thick] (r)--(-1,1.5);
\node[right] at (-1,1.5) {$e_{1}^{*}+x$};
\end{scope}

\begin{scope}[xshift=180, xscale=1.0, yscale=1.0]
\fill (-1,1) coordinate (r) circle (2pt);
\fill (2,-2) coordinate (q) circle (2pt); 
\fill (-1,-2) coordinate (p) circle (2pt); 

\draw[dashed] (p)--(-2,-2);
\draw[dashed] (p)--(-1,-3);
\draw (p)--(q);
\draw[dashed] (q)--(3,-2);
\draw[dashed] (q)--(3,-3);
\draw (q)--(r);
\draw[dashed] (r)--(-2,2);
\draw[dashed] (r)--(-1,2);
\draw (r)--(p);

\draw[->, thick] (p)--(-0.5,-2);
\node[below] at (-0.2,-2) {$e_{2}^{*}$};
\draw[->, thick] (p)--(-1,-1.5);
\node[left] at (-1,-1.2) {$e_{1}^{*}$};

\draw[->, thick] (q)--(1.5,-2);
\node[below] at (1.2,-2) {$-e_{2}^{*}$};
\draw[->, thick] (q)--(1.5,-1.5);
\node[left] at (1.5,-1.5) {$e_{1}^{*}-e_{2}^{*}$};

\draw[->, thick] (r)--(-1,0.5);
\node[left] at (-1,0.3) {$-e_{1}^{*}$};
\draw[->, thick] (r)--(-0.5,0.5);
\node[right] at (-0.5,0.5) {$e_{2}^{*}-e_{1}^{*}$};
\end{scope}
\end{tikzpicture}
\caption{The left figure shows a hyperplane of the left GKM graph in
  Figure~\ref{Figure2}. The right figure is a $2$-valent GKM subgraph
  of the left GKM graph in Figure~\ref{Figure2} but it is not a
  hyperplane.}
\label{Figure3}
\end{figure}

\end{example}

For the hyperplane $\mathcal{L}$ we have the following property (this
may be regarded as the analogue of the property of facets in a torus
graph \cite{MMP}).
\begin{lemma}
\label{key-properties-hyperplane}
Let $\mathcal{G}=(\Gamma,\ \alpha,\ \nabla)$ be a $2n$-valent $T^{*}\mathbb{C}^{n}$-modeled GKM graph.
Take a vertex $p\in \mathcal{V}$.
Then, for every $1$-dimensional pair $\{\epsilon^{+},\epsilon^{-}\}\subset \mathcal{E}_{p}$, there exists a unique hyperplane $\mathcal{L}=(L,\alpha^{L},\nabla^{L})$
such that $\mathcal{E}_{p}^{L}=\mathcal{E}_{p}\setminus \{\epsilon^{+},\epsilon^{-}\}$.
\end{lemma}
\begin{proof}
  We first prove the existence of the hyperplane which satisfies the
  statement.  Put
  $\mathcal{E}_{p}^{L}=\mathcal{E}_{p}\setminus
  \{\epsilon^{+},\epsilon^{-}\}$.  Then we can write
  $\mathcal{E}_{p}^{L}=\{\epsilon_{1}^{+},\ \cdots,\
  \epsilon_{n-1}^{+},\ \epsilon_{1}^{-},\ \cdots,\
  \epsilon_{n-1}^{-}\}$ the $(n-1)$ $1$-dimensional pairs in
  $\mathcal{E}_{p}$ which are different from
  $\{\epsilon^{+},\epsilon^{-}\}$.
Let   
\begin{align*}
R:=\langle \alpha(\epsilon_{1}^{+}),\ \cdots,\
  \alpha(\epsilon_{n-1}^{+}),\ x \rangle.
\end{align*}
By
  Definition~\ref{def_modeled_graph} (2) we may assume that
  $\langle \alpha(\epsilon_{1}^{+}),\ldots, \alpha(\epsilon_{n-1}^{+})
  \rangle \subset \algt^{*}_{\mathbb{Z}}\oplus \mathbb{Z}x$ is a submodule of rank
  $(n-1)$. Thus we have

\begin{align*}
R\simeq H^{2}(BT^{n-1})\oplus \Z x.
\end{align*}

Take an element $\epsilon\in \mathcal{E}_{p}^{L}$ which becomes an edge in $\Gamma$, i.e., $i(\epsilon)=p$ and there exists $t(\epsilon)\in \mathcal{V}$.
In other words, $\epsilon\in \mathcal{E}^{L}_{p}\cap E_{p}$ (Note that this will be $E_{p}^{L}$ of $L$).
By Lemma \ref{lem_about_pair}, 
the subset $\nabla_{\epsilon}(\mathcal{E}_{p}^{L})$ in $\mathcal{E}_{t(\epsilon)}$ consists of exactly $(n-1)$ $1$-dimensional pairs.
Moreover, 
because $\alpha$ satisfies the congruence relation on the edge $\epsilon \in \mathcal{E}^{L}_{p}\cap E_{p}$,
$\alpha(\nabla_{\epsilon}(\mathcal{E}_{p}^{L}))$ and $x$ span the same subspace $R$ as above.
This property holds for all edges $\epsilon\in \mathcal{E}^{L}_{p}\cap E_{p}$.
Hence, we can define the following $(2n-2)$-valent subgraph 
in $\Gamma=(\mathcal{V},\mathcal{E})$:
\begin{eqnarray*}
L_{1}=(\mathcal{V}^{L_{1}},\mathcal{E}^{L_{1}})
\end{eqnarray*}
such that  
\begin{align*}
& \mathcal{V}^{L_{1}}:=\{p,\ t(\epsilon)\ |\ \epsilon\in \mathcal{E}^{L}_{p}\cap E_{p}\}; \\
& \mathcal{E}^{L_{1}}:=\bigcup_{\epsilon\in \mathcal{E}^{L}_{p}\cap E_{p}}\nabla_{\epsilon}(\mathcal{E}_{p}^{L})\cup\mathcal{E}_{p}^{L}.
\end{align*}
If we restrict $\alpha$ and $\nabla$ onto $L_{1}$, then this becomes a
$(2n-2)$-valent $T^{*}\mathbb{C}^{n-1}$-modeled GKM subgraph, say
$\mathcal{L}_{1}$, in $\mathcal{G}$.  If $L_{1}$ is maximal, i.e., if
there is a $(2n-2)$-valent graph $L'$ such that $L_{1}\subset L'$ then
$L_{1}=L'$, then $\mathcal{L}_{1}$ is a hyperplane.  Assume that
$L_{1}$ is not maximal.  In this case, for every vertex
$q\in \mathcal{V}^{L_{1}}$ and every edge
$\epsilon\in \mathcal{E}_{q}^{L_{1}}\cap E_{q}$, we can apply the similar
method stated as above.  Then we can construct the $(2n-2)$-valent
$T^{*}\mathbb{C}^{n-1}$-modeled GKM subgraph $\mathcal{L}_{2}$ which
contains $\mathcal{L}_{1}$.  If $\mathcal{L}_{2}$ is maximal, then
this is a hyperplane which we want to have.  Otherwise, by repeating
similar arguments, we get the hyperplane $\mathcal{L}$ which contains
$\mathcal{E}_{p}^{L}$.

Suppose that 
there are two hyperplanes $\mathcal{L}=(L,\alpha^{L},\nabla^{L})$ and $\mathcal{L}'=(L',\alpha^{L'},\nabla^{L'})$ such that $\mathcal{E}_{p}^{L}=\mathcal{E}_{p}^{L'}$.
Because $E_{p}^{L}=E_{p}^{L'}$ and two connections are restricted from the connection $\nabla$ of $\mathcal{G}$, we see that the following two subgraphs are the same graph:
\begin{eqnarray*}
\bigcup_{\epsilon\in E_{p}^{L}}\nabla_{\epsilon}(\mathcal{E}_{p}^{L})\cup\mathcal{E}_{p}^{L}=\bigcup_{\epsilon\in E_{p}^{L'}}\nabla_{\epsilon}(\mathcal{E}_{p}^{L'})\cup\mathcal{E}_{p}^{L'}.
\end{eqnarray*} 
By iterating this construction along all edges in $L$ and $L'$, 
finally, we know that $L=L'$.
Therefore, such a hyperplane is unique. 
\end{proof}

\subsection{Pre-halfspace and its Thom class}
\label{sect:3.2}

In this section, we introduce a {\it pre-halfspace} and its {\it Thom class}.

Take a subgraph $H=(\mathcal{V}^{H},\mathcal{E}^{H})$ of $\Gamma$ such that 
$|\mathcal{E}^{H}_{p}|=2n-1$ or $2n$ for all $p\in \mathcal{V}^{H}$. We assume that there always exists a vertex $p\in \mathcal{V}^{H}$ with $|\mathcal{E}_{p}^{H}|=2n-1$.
Moreover, we assume that $H$ is \textit{closed} under the connection $\nabla$ of $\mathcal{G}=(\Gamma,\ \alpha,\ \nabla)$, 
that is, 
\begin{description}
\item[(C1)] $\nabla_{\epsilon}^{H}:=\nabla_{\epsilon}|_{\mathcal{E}_{i(\epsilon)}^{H}}:\mathcal{E}^{H}_{i(\epsilon)}\to \mathcal{E}^{H}_{t(\epsilon)}$ is 
bijective, if $|\mathcal{E}^{H}_{i(\epsilon)}|=|\mathcal{E}^{H}_{t(\epsilon)}|=2n-1$ or $2n$;
\item[(C2)] $\nabla_{\epsilon}^{H}:\mathcal{E}^{H}_{i(\epsilon)}\to \mathcal{E}^{H}_{t(\epsilon)}$ is 
injective, if $|\mathcal{E}^{H}_{i(\epsilon)}|=2n-1<|\mathcal{E}^{H}_{t(\epsilon)}|=2n$.
\end{description}
In addition, we also assume that 
if $|\mathcal{E}^{H}_{i(\epsilon)}|(=2n-1)<|\mathcal{E}^{H}_{t(\epsilon)}|(=2n)$ then $\nabla_{\epsilon}^{H}$ satisfies the following congruence relation
for $\{n^{H}(i(\epsilon))\}=\mathcal{E}^{\Gamma}_{i(\epsilon)}-\mathcal{E}^{H}_{i(\epsilon)}$ 
(we call such an $n^{H}(p)$ a \textit{normal edge} or a \textit{normal leg} of $H$ at $p$):
\begin{eqnarray}
\label{congrel_pre-halfsp}
\alpha(n^{H}(p))-x\equiv 0\ ({\rm mod}\ \alpha(\epsilon)).
\end{eqnarray}
Now we may define the pre-halfspace.

\begin{definition}[pre-halfspace]
\label{pre-halfspace}
Let $\mathcal{H}:=(H,\alpha^{H},\nabla^{H})$ be the triple as above, i.e.,
\begin{itemize}
\item $H=(\mathcal{V}^{H},\mathcal{E}^{H})$ is a subgraph $\Gamma$ such that 
$|\mathcal{E}^{H}_{p}|=2n-1$ or $2n$ for all $p\in \mathcal{V}^{H}$, where there exists a vertex $p\in \mathcal{V}^{H}$ with $|\mathcal{E}^{H}_{p}|=2n-1$;
\item $\alpha^{H}:=\alpha|_{\mathcal{E}^{H}}:\mathcal{E}^{H}\to H^{2}(BT^{n})\oplus\mathbb{Z} x$ is the axial function restricted onto $H$;
\item $\nabla^{H}:=\{\nabla_{\epsilon}^{H}\ |\ \epsilon\in E^{H}\}$ is the restrected connection of $\nabla$ onto $H$ which satisfies the conditions (C1), (C2), \eqref{congrel_pre-halfsp} as above,
\end{itemize}
 then we call $\mathcal{H}$ a {\it pre-halfspace} of $\mathcal{G}=(\Gamma,\ \alpha,\ \nabla)$.
\end{definition}

\begin{example}
\label{example-pre-halfspace}
The following left figure (Figure~\ref{fig-pre-halfspace}) shows
an example of pre-halfspace of the left graph in Figure~\ref{Figure2}.
On the other hand, the following right figure (Figure~\ref{fig-non-pre-halfspace}) is an abstract subgraph of the left graph in Figure~\ref{Figure2}.
However, this is not closed under the connection $\nabla$, because the congruence relation does not hold on the diagonal edge.
Therefore, this is not a pre-halfspace. 
\begin{figure}[htbp]
 \begin{minipage}{0.4\hsize}
\begin{tikzpicture}
\begin{scope}[xscale=0.8, yscale=0.8]
\fill (-1,1) coordinate (r) circle (2pt);
\fill (2,-2) coordinate (q) circle (2pt); 
\fill (-1,-2) coordinate (p) circle (2pt); 

\draw[dashed] (-2,2)--(r);
\draw (-1,1)--(3,-3);
\draw (-1,2)--(-1,-3);
\draw[dashed] (-2,-2)--(q);
\draw (-1,-2)--(3,-2);

\draw[->, thick] (p)--(-0.5,-2);
\node[below] at (-0.2,-2) {$e_{2}^{*}$};
\draw[->, thick] (p)--(-1,-1.5);
\node[left] at (-1,-1.2) {$e_{1}^{*}$};
\draw[->, thick] (p)--(-1,-2.5);
\node[left] at (-1,-2.8) {$x-e_{1}^{*}$};

\draw[->, thick] (q)--(1.5,-2);
\node[below] at (1.2,-2) {$-e_{2}^{*}$};
\draw[->, thick] (q)--(2.5,-2);
\node[above] at (2.8,-2) {$e_{2}^{*}+x$};
\draw[->, thick] (q)--(2.5,-2.5);
\node[right] at (2.5,-2.5) {$e_{2}^{*}-e_{1}^{*}+x$};
\draw[->, thick] (q)--(1.5,-1.5);
\node[left] at (1.5,-1.5) {$e_{1}^{*}-e_{2}^{*}$};

\draw[->, thick] (r)--(-1,0.5);
\node[left] at (-1,0.3) {$-e_{1}^{*}$};
\draw[->, thick] (r)--(-1,1.5);
\node[right] at (-1,1.5) {$e_{1}^{*}+x$};
\draw[->, thick] (r)--(-0.5,0.5);
\node[right] at (-0.5,0.5) {$e_{2}^{*}-e_{1}^{*}$};
\end{scope}
\end{tikzpicture}
\caption{}
\label{fig-pre-halfspace}
 \end{minipage}
 \begin{minipage}{0.4\hsize}
\begin{tikzpicture}
\begin{scope}[xscale=0.8, yscale=0.8]
\fill (-1,1) coordinate (r) circle (2pt);
\fill (2,-2) coordinate (q) circle (2pt); 
\fill (-1,-2) coordinate (p) circle (2pt); 

\draw[dashed] (-2,2)--(r);
\draw (-1,1)--(3,-3);
\draw (-1,2)--(-1,-3);
\draw[dashed] (-2,-2)--(q);
\draw (-1,-2)--(q);
\draw[dashed] (q)--(3,-2);

\draw[->, thick] (p)--(-0.5,-2);
\node[below] at (-0.2,-2) {$e_{2}^{*}$};
\draw[->, thick] (p)--(-1,-1.5);
\node[left] at (-1,-1.2) {$e_{1}^{*}$};
\draw[->, thick] (p)--(-1,-2.5);
\node[left] at (-1,-2.8) {$x-e_{1}^{*}$};

\draw[->, thick] (q)--(1.5,-2);
\node[below] at (1.2,-2) {$-e_{2}^{*}$};
\draw[->, thick] (q)--(2.5,-2.5);
\node[right] at (2.5,-2.5) {$e_{2}^{*}-e_{1}^{*}+x$};
\draw[->, thick] (q)--(1.5,-1.5);
\node[left] at (1.5,-1.5) {$e_{1}^{*}-e_{2}^{*}$};

\draw[->, thick] (r)--(-1,0.5);
\node[left] at (-1,0.3) {$-e_{1}^{*}$};
\draw[->, thick] (r)--(-1,1.5);
\node[right] at (-1,1.5) {$e_{1}^{*}+x$};
\draw[->, thick] (r)--(-0.5,0.5);
\node[right] at (-0.5,0.5) {$e_{2}^{*}-e_{1}^{*}$};
\end{scope}
\end{tikzpicture}
\caption{}
\label{fig-non-pre-halfspace}
 \end{minipage}
\end{figure}

\end{example}

Let $\mathcal{H}=(H,\ \alpha^{H},\ \nabla^{H})$ be a pre-halfspace of a $T^{*}\mathbb{C}^{n}$-modeled GKM graph $\mathcal{G}=(\Gamma,\ \alpha,\ \nabla)$.
We can define the notion of a \textit{Thom class} for the pre-halfspace (also see \cite{MMP}).
\begin{definition}[Thom class]
\label{Thom class}
A {\it Thom class} of $\mathcal{H}$ is defined by the map $\tau_{H}:\mathcal{V}\to H^{2}(BT^{n})\oplus\mathbb{Z}x$ such that 
\begin{eqnarray*}
\tau_{H}(p)=\left\{
\begin{array}{ll}
0 & {\rm if}\ p\not\in \mathcal{V}^{H} \\
x & {\rm if}\ |\mathcal{E}^{H}_{p}|=2n \\
\alpha(n^{H}(p)) & {\rm if}\ |\mathcal{E}^{H}_{p}|=2n-1,
\end{array} \right.
\end{eqnarray*}
\end{definition}
We also call a vertex $p\in \mathcal{V}$ with $\tau_{H}(p)=0$ (resp.\ $\tau_{H}(p)=x$, $\tau_{H}(p)=\alpha(n^{H}(p))$)
an {\it exterior} (resp.\ {\it interior}, {\it boundary}) {\it vertex} of $H$.

For the Thom class $\tau_{H}$ of a pre-halfspace $H$, we have the following lemma.
\begin{lemma}
\label{thom}
The Thom class $\tau_{H}$ of a pre-halfspace $\mathcal{H}$ is an element of $H^{*}(\mathcal{G})$.
\end{lemma}
\begin{proof}
Take an edge $\epsilon\in E$.
We claim that $\tau_{H}$ satisfies the congruence relation on $\epsilon$, that is, $\tau_{H}(i(\epsilon))-\tau_{H}(t(\epsilon))\equiv 0$ $({\rm mod}\ \alpha(\epsilon))$.

We first assume that $i(\epsilon),\ t(\epsilon)\not\in \mathcal{V}^{H}$ or $|\mathcal{E}_{i(\epsilon)}^{H}|=2n=|\mathcal{E}_{t(\epsilon)}^{H}|$. 
Namely, both of $i(\epsilon)$ and $t(\epsilon)$ are exterior vertices or interior vertices of $H$.
Then, by definition of the Thom class, 
\begin{align*}
\tau_{H}(i(\epsilon))-\tau_{H}(t(\epsilon))=0\equiv 0\ {\rm mod}\ \alpha(\epsilon).
\end{align*}
So the congruence relation holds for these cases. 

We next consider the other cases, i.e., the case when both of $i(\epsilon)$ and $t(\epsilon)$ are boundary vertices or the case when $i(\epsilon)$ is a boundary vertex but $t(\epsilon)$ is an exterior or an interior vertex.
Put $p=i(\epsilon)$.
Assume that $t(\epsilon)(=:q)$ is also a boundary vertex.
Because the pre-halfspace is closed by the connection $\nabla$, we have that $\nabla_{\epsilon}(n^{H}(p))=n^{H}(q)$. 
By the definition of Thom classes, we have that 
\begin{align*}
\tau_{H}(i(\epsilon))-\tau_{H}(t(\epsilon))
&=\alpha(n^{H}(p))-\alpha(n^{H}(q)) \\
&=\alpha(n^{H}(p))-\alpha(\nabla_{\epsilon}(n^{H}(p)))\equiv 0\ {\rm mod}\ \alpha(\epsilon).
\end{align*}
Assume that $t(\epsilon)$ is an exterior vertex, 
Then $\epsilon=n^{H}(p)$ and 
\begin{align*}
\tau_{H}(i(\epsilon))-\tau_{H}(t(\epsilon))
=\alpha(n^{H}(p))-0
\equiv 0\ {\rm mod}\ \alpha(\epsilon)=\alpha(n^{H}(p)).
\end{align*}
Assume that $t(\epsilon)(=:q)$ is an interior vertex. 
Namely, $\tau_{H}(q)=x$. 
In this case, 
we have 
\begin{align*}
\tau_{H}(p)-\tau_{H}(q)=\alpha(n^{H}(p))-x.
\end{align*}
It follows from \eqref{congrel_pre-halfsp} that 
\begin{align*}
\alpha(n^{H}(p))-x\equiv 0\ {\rm mod}\ \alpha(\epsilon).
\end{align*}
So the congruence relation also holds for these cases. 
Consequently, we have $\tau_{H}\in H^{*}(\mathcal{G})$.
\end{proof}

\begin{example}
\label{example-Thom-class}
The following figure (Figure~\ref{fig-Thom-class}) shows an example of the Thom class of the pre-halfspace in Figure~\ref{fig-pre-halfspace}.

\begin{figure}[h]
\begin{tikzpicture}
\begin{scope}[xscale=0.8, yscale=0.8]
\fill (-1,1) coordinate (r) circle (2pt);
\fill (2,-2) coordinate (q) circle (2pt); 
\fill (-1,-2) coordinate (p) circle (2pt); 

\draw[dashed] (-2,2)--(r);
\draw (-1,1)--(3,-3);
\draw (-1,2)--(-1,-3);
\draw[dashed] (-2,-2)--(q);
\draw (-1,-2)--(3,-2);

\draw[->, thick] (p)--(-1.5,-2);
\node[above] at (-2,-2) {$x-e_{2}^{*}$};

\node[below] at (q) {$x$};

\draw[->, thick] (r)--(-1.5,1.5);
\node[left] at (-1.6,1.5) {$e_{1}^{*}-e_{2}^{*}+x$};
\end{scope}
\end{tikzpicture}
\caption{}
\label{fig-Thom-class}
\end{figure}
\end{example}

\subsection{Opposite side of a pre-halfspace}
\label{sect:3.3}

Next we define the \textit{opposite side} of the pre-halfspace.
In order to define it, we need to prove Lemma \ref{existence_of_opposite-side}

In order to prove it, we define a \textit{boundary} of a pre-halfspace.
Let $H$ be a pre-halfspace.
By the definition of a pre-halfspace, there is a vertex $p\in \mathcal{V}^{H}$ such that $|\mathcal{E}_{p}^{H}|=2n-1$.
Then $\mathcal{E}_{p}^{H}$ has $(n-1)$ $1$-dimensional pairs, say $\{\epsilon_{1}^{+},\ \cdots,\ \epsilon_{n-1}^{+},\ \epsilon_{1}^{-},\ \cdots,\ \epsilon_{n-1}^{-}\}$.
Because of Lemma \ref{key-properties-hyperplane}, there exists a  unique hyperplane $\mathcal{L}=(L,\alpha^{L},\nabla^{L})$ in $\mathcal{G}$ such that
\begin{align*}
\alpha^{L}:=\alpha|_{\mathcal{E}^{L}},\quad 
\nabla^{L}:=\{\nabla_{\epsilon}|_{\mathcal{E}_{i(\epsilon)}^{L}}\ |\ \epsilon\in E^{L}\}.
\end{align*}
Moreover, $\mathcal{L}$ satisfies that 
\begin{align*}
p\in \mathcal{V}^{L}
\end{align*}
 and 
\begin{align*}
\mathcal{E}_{p}^{L}=\{\epsilon_{1}^{+},\ \cdots,\ \epsilon_{n-1}^{+},\ \epsilon_{1}^{-},\ \cdots,\ \epsilon_{n-1}^{-}\}.
\end{align*}
We call the union of all such hyperplanes $\mathcal{L}$ a
\textit{boundary} of the pre-halfspace $\mathcal{H}$, and we denote it
by
$\partial \mathcal{H}=(\partial H,\alpha^{\partial H},
\nabla^{\partial H})$.  Note that a boundary of $\mathcal{H}$ may not
be connected (see Figure~\ref{figure of opposite side2}).  To define the opposite side of $\mathcal{H}$, we need the
following lemma:
\begin{lemma}
\label{existence_of_opposite-side}
Let $\mathcal{H}=(H,\alpha^{H},\nabla^{H})$ be a pre-halfspace in
$\mathcal{G}=(\Gamma,\ \alpha,\ \nabla)$
and $x$ be a residual basis.
Then there is a unique pre-halfspace $\mathcal{I}=(I,\alpha^{I},\nabla^{I})$ such that 
\begin{itemize}
\item $H\cup I=\Gamma$;
\item $\tau_{H}+\tau_{I}=\chi \in H^{*}(\mathcal{G})$,
\end{itemize}
where $\chi$ is an element of $H^{*}(\mathcal{G})$ defined by $\chi(p)=x$ for all $p\in \mathcal{V}$.
\end{lemma}
\begin{proof}
Set $H=(\mathcal{V}^{H},\ \mathcal{E}^{H})$ and $\hat{I}=(\mathcal{V}^{\Gamma}-\mathcal{V}^{H},\ \mathcal{E}^{\Gamma}-\mathcal{E}^{H})$.
Define
\begin{eqnarray*}
I=\hat{I}\cup \partial H.
\end{eqnarray*}
Namely,
\begin{align*}
I=(\mathcal{V}^{I},\ \mathcal{E}^{I})=((\mathcal{V}^{\Gamma}-\mathcal{V}^{H})\cup \mathcal{V}^{\partial H},\ (\mathcal{E}^{\Gamma}-\mathcal{E}^{H})\cup \mathcal{E}^{\partial H}).
\end{align*}
Then we can easily see that $H\cup I=\Gamma$ and $H\cap I=\partial H$.

We next prove $I$ is a pre-halfspace.
Take $p\in \mathcal{V}^{I}$.
If $p\in \mathcal{V}^{\hat{I}}=\mathcal{V}^{\Gamma}-\mathcal{V}^{H}$, then $\mathcal{E}_{p}^{H}=\emptyset$; therefore, 
$\mathcal{E}_{p}^{I}=\mathcal{E}_{p}^{\Gamma}-\emptyset=\mathcal{E}_{p}^{\Gamma}$ and $|\mathcal{E}_{p}^{I}|=|\mathcal{E}_{p}^{\Gamma}|=2n$.
If $p\in \mathcal{V}^{\partial H}$, then $\mathcal{E}_{p}^{I}=\mathcal{E}_{p}^{\partial H}\cup \{n_{H}(p)\}$, that is, $|\mathcal{E}_{p}^{I}|=2n-1$.
Here $n_{H}(p)$ is a normal edge (leg) of $H$ on $p$.
Therefore, for an edge $\epsilon\in E^{I}$ and the restricted connection $\nabla_{\epsilon}^{I}:=\nabla_{\epsilon}|_{\mathcal{E}_{i(\epsilon)}^{I}}$, 
it follows from the definition of $\nabla$ on $\Gamma$ that we have   
\begin{itemize}
\item $\nabla_{\epsilon}^{I}:\mathcal{E}_{i(\epsilon)}^{I}\to \mathcal{E}_{t(\epsilon)}^{I}$ is bijective, if $|\mathcal{E}_{i(\epsilon)}^{I}|=|\mathcal{E}_{t(\epsilon)}^{I}|$,  
\item $\nabla_{\epsilon}^{I}:\mathcal{E}_{i(\epsilon)}^{I}\to \mathcal{E}_{t(\epsilon)}^{I}$ is injective, if $|\mathcal{E}_{i(\epsilon)}^{I}|=2n-1<|\mathcal{E}_{t(\epsilon)}^{I}|=2n$.
\end{itemize}
Take an edge $\epsilon\in E^{I}$ in $I$ such that $|\mathcal{E}_{i(\epsilon)}^{I}|=2n-1<|\mathcal{E}_{t(\epsilon)}^{I}|=2n$.
Put 
$i(\epsilon)=p\in \mathcal{V}^{\partial H}$ and $n_{H}(p)=\epsilon^{+}(=\epsilon)$. 
Then the normal edge (leg) of $I$ on $p$ can be taken as $\epsilon^{-}=n_{I}(p)$, 
where $\{\epsilon^{+},\ \epsilon^{-}\}=\mathcal{E}_{p}-\mathcal{E}^{\partial H}_{p}$ is a $1$-dimensional pair in $\mathcal{E}_{p}$.
So we have the following equation:
\begin{eqnarray}
\label{two-side_I_H}
\alpha(n_{I}(p))+\alpha(n_{H}(p))=\alpha(\epsilon^{-})+\alpha(\epsilon^{+})=(-\alpha(\epsilon^{+})+x)+\alpha(\epsilon^{+})=x.
\end{eqnarray}
Therefore, we have that 
\begin{align*}
\alpha(n_{I}(p))-x=&-\alpha(n_{H}(p)) \\
=&-\alpha(\epsilon)\equiv 0\ ({\rm mod}\ \alpha(n_{H}(p))=\alpha(\epsilon)).
\end{align*}
Consequently, we have $\mathcal{I}=(I,\alpha^{I},\nabla^{I})$ is a pre-halfspace such that $H\cup I=\Gamma$, where $\alpha^{I}:=\alpha|_{\mathcal{E}^{I}}$ and $\nabla^{I}:=\{\nabla_{\epsilon}|_{\mathcal{E}^{I}}\ |\ \epsilon\in E^{I}\}$.
Moreover, we have that $\tau_{H}+\tau_{I}=\chi$, because of the above equation and the definition of the Thom class of the pre-halfspace.

We finally claim the uniqueness of $I$.
By two conditions $H\cup I=\Gamma$, $\tau_{H}+\tau_{I}=\chi$ and the definition of the Thom class, we see $\partial I=\partial H$ and 
$I=(\mathcal{V}^{\Gamma}-\mathcal{V}^{H},\ \mathcal{E}^{\Gamma}-\mathcal{E}^{H})\cup \partial I$.
From Lemma \ref{key-properties-hyperplane}, the boundary $\partial H=\partial I$ is uniquely determined (though it may not be connected).
So we know the uniqueness of $I$.
\end{proof}

We call $\mathcal{I}$ in Lemma \ref{existence_of_opposite-side} an {\it opposite side} of $\mathcal{H}$ and denote it by $\overline{\mathcal{H}}=(\overline{H},\alpha^{\overline{H}},\nabla^{\overline{H}})$.
Note that 
\begin{eqnarray*}
H\cap \overline{H}=\partial H
\end{eqnarray*}
by the proof of Lemma \ref{existence_of_opposite-side}.


\subsection{Halfspace and Ring $\Z[\mathcal{G}]$}
\label{sect:3.4}
Under the above preparations, we may define the halfspace.

\begin{definition}[halfspace]
A pre-halfspace $\mathcal{H}$ is said to be a \textit{halfspace}, if $H$ is a connected subgraph and its opposite side is also connected.
\end{definition}

\begin{figure}[h]
\begin{tikzpicture}
\begin{scope}[xscale=0.8, yscale=0.8]
\fill (-1,1) coordinate (r) circle (2pt);
\fill (-1,-2) coordinate (p) circle (2pt); 
\fill (2,-2) coordinate (q) circle (2pt); 

\node[below] at (q) {$0$};

\draw (-1,2)--(-1,-3);
\draw (p)--(-2,-2);
\draw (r)--(-2,2);

\draw[->, thick] (p)--(-0.5,-2);
\node[right] at (-0.5,-2) {$e_{2}^{*}$};
\draw[->, thick] (r)--(-0.5,0.5);
\node[right] at (-0.5,0.5) {$e_{2}^{*}-e_{1}^{*}$};
\end{scope}

\begin{scope}[xshift=180, xscale=0.8, yscale=0.8]
\fill (-1,1) coordinate (r) circle (2pt);
\fill (2,-2) coordinate (q) circle (2pt); 
\fill (-1,-2) coordinate (p) circle (2pt); 

\draw (-1,1)--(3,-3);
\draw (-1,2)--(-1,-3);
\draw (-1,-2)--(3,-2);

\node[below] at (q) {$x$};

\draw[->, thick] (p)--(-1.5,-2);
\node[left] at (-1.5,-2) {$x-e_{2}^{*}$};
\draw[->, thick] (r)--(-1.5,1.5);
\node[left] at (-1.5,1.5) {$e_{1}^{*}-e_{2}^{*}+x$};
\end{scope}
\end{tikzpicture}
\caption{The above figures are a halfspace and its opposite side of the left GKM graph in Figure~\ref{Figure2}.
The labels on vertices mean the values of their Thom classes on vertices, where $0$ means the value of the Thom class $\tau_{H}$ on the exterior vertex of a halfspace $H$. We call such a vertex a {\it fake vertex} when we consider the Thom class $\tau_{H}$ of $H$.
Note that the boundary $\partial H=H\cap \overline{H}$ is connected.}
\label{figure of opposite side1}
\end{figure}
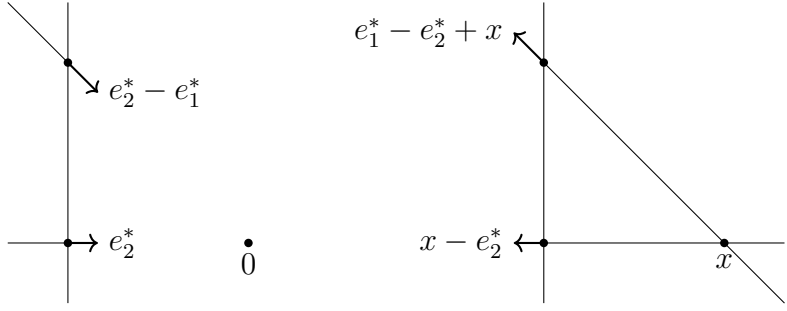

\begin{figure}[h]
\begin{tikzpicture}
\begin{scope}[xscale=0.8, yscale=0.8]
\fill (2,0) coordinate (p) circle (2pt);
\fill (2,2) coordinate (q) circle (2pt);
\fill (0,2) coordinate (r) circle (2pt);
\fill (-2,2) coordinate (s) circle (2pt);
\fill (-2,0) coordinate (t) circle (2pt);
\fill (-2,-2) coordinate (u) circle (2pt);
\fill (0,-2) coordinate (v) circle (2pt);
\fill (2,-2) coordinate (w) circle (2pt);

\draw (-3,2)--(0,2);
\draw (-3,0)--(-1,0);
\draw (-3,-2)--(0,-2);

\draw (-2,3)--(-2,-3);
\draw (0,3)--(0,1);
\draw (0,-3)--(0,-1);

\draw[->, thick] (r)--(0.5,2);
\node[below] at (0.8,2) {$e_{2}^{*}+x$};

\draw[->, thick] (v)--(0.5,-2);
\node[below] at (0.8,-2) {$e_{2}^{*}+x$};

\node[below] at (p) {$0$};
\node[below] at (q) {$0$};
\node[below] at (w) {$0$};

\node[below] at (-2.3,2) {$x$};
\node[below] at (-2.3,0) {$x$};
\node[below] at (-2.3,-2) {$x$};
\end{scope}

\begin{scope}[xshift=200, xscale=0.8, yscale=0.8]
\fill (2,0) coordinate (p) circle (2pt);
\fill (2,2) coordinate (q) circle (2pt);
\fill (0,2) coordinate (r) circle (2pt);
\fill (-2,2) coordinate (s) circle (2pt);
\fill (-2,0) coordinate (t) circle (2pt);
\fill (-2,-2) coordinate (u) circle (2pt);
\fill (0,-2) coordinate (v) circle (2pt);
\fill (2,-2) coordinate (w) circle (2pt);

\draw (0,2)--(3,2);
\draw (1,0)--(3,0);
\draw (0,-2)--(3,-2);

\draw (0,3)--(0,1);
\draw (0,-3)--(0,-1);
\draw (2,3)--(2,-3);

\draw[->, thick] (r)--(-0.5,2);
\node[below] at (-0.5,2) {$-e_{2}^{*}$};

\draw[->, thick] (v)--(-0.5,-2);
\node[below] at (-0.5,-2) {$-e_{2}^{*}$};

\node[below] at (2.3,0) {$x$};
\node[below] at (2.3,2) {$x$};
\node[below] at (2.3,-2) {$x$};

\node[below] at (s) {$0$};
\node[below] at (t) {$0$};
\node[below] at (u) {$0$};
\end{scope}
\end{tikzpicture}
\caption{The above figures are a halfspace and its opposite side of the right GKM graph in Figure~\ref{Figure2}, where $0$'s are the fake vertices.
Note that the boundary $\partial H=H\cap \overline{H}$ is not connected.}
\label{figure of opposite side2}
\end{figure}
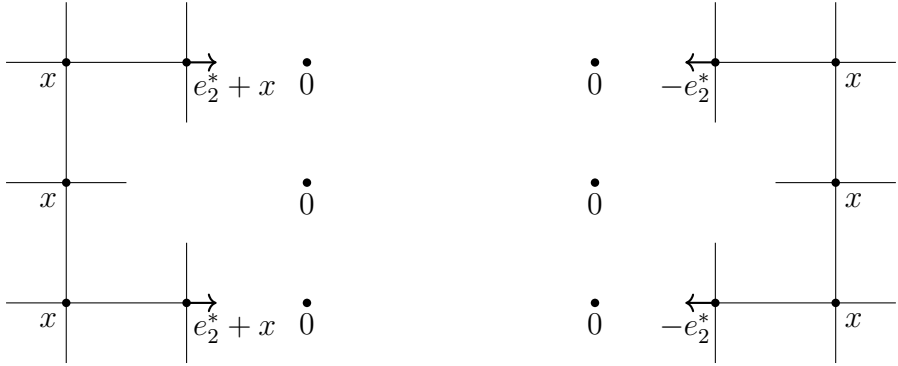

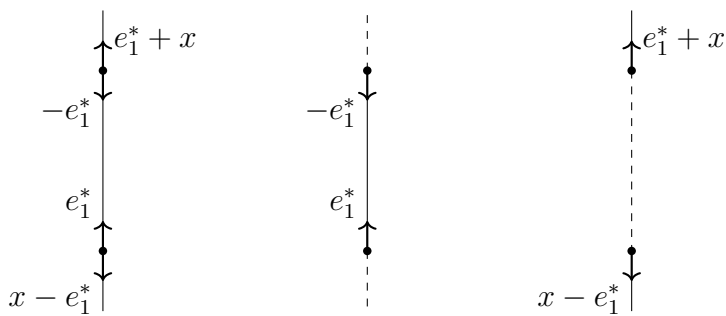
\begin{figure}[h]
\begin{tikzpicture}
\begin{scope}[xscale=0.8, yscale=0.8]
\fill (-1,1) coordinate (r) circle (2pt);
\fill (-1,-2) coordinate (p) circle (2pt); 

\draw (-1,2)--(-1,-3);

\draw[->, thick] (p)--(-1,-1.5);
\node[left] at (-1,-1.2) {$e_{1}^{*}$};
\draw[->, thick] (p)--(-1,-2.5);
\node[left] at (-1,-2.8) {$x-e_{1}^{*}$};

\draw[->, thick] (r)--(-1,0.5);
\node[left] at (-1,0.3) {$-e_{1}^{*}$};
\draw[->, thick] (r)--(-1,1.5);
\node[right] at (-1,1.5) {$e_{1}^{*}+x$};
\end{scope}

\begin{scope}[xshift=100, xscale=0.8, yscale=0.8]
\fill (-1,1) coordinate (r) circle (2pt);
\fill (-1,-2) coordinate (p) circle (2pt); 

\draw[dashed] (p)--(-1,-3);
\draw[dashed] (r)--(-1,2);
\draw (r)--(p);

\draw[->, thick] (p)--(-1,-1.5);
\node[left] at (-1,-1.2) {$e_{1}^{*}$};

\draw[->, thick] (r)--(-1,0.5);
\node[left] at (-1,0.3) {$-e_{1}^{*}$};
\end{scope}

\begin{scope}[xshift=200, xscale=0.8, yscale=0.8]
\fill (-1,1) coordinate (r) circle (2pt);
\fill (-1,-2) coordinate (p) circle (2pt); 

\draw (-1,2)--(r);
\draw (p)--(-1,-3);
\draw[dashed] (-1,1)--(-1,-2);

\draw[->, thick] (p)--(-1,-2.5);
\node[left] at (-1,-2.8) {$x-e_{1}^{*}$};

\draw[->, thick] (r)--(-1,1.5);
\node[right] at (-1,1.5) {$e_{1}^{*}+x$};
\end{scope}
\end{tikzpicture}
\caption{The middle graph (edge) $\mathcal{H}$ is a pre-halfspace of the left GKM graph. However, its opposite side is the right graph (two legs). This is not connected; therefore, $\mathcal{H}$ is not a halfspace.}
\label{figure of disconnected opposite side}
\end{figure}

Let $\mathbf{H}$ be the set of all halfspaces in $\mathcal{G}$.
Because the graph $\Gamma$ is finite and the opposite side of the halfspace is also a halfspace, we may write the set of all halfspaces by 
\begin{eqnarray*}
\mathbf{H}=\{H_{1},\ \cdots,\ H_{m},\ \overline{H_{1}},\ \cdots,\ \overline{H_{m}} \}.
\end{eqnarray*}
Put
\begin{eqnarray*}
\Z[X,\ \mathbf{H}]=\Z[X,\ H_{1},\ \cdots,\ H_{m},\ \overline{H_{1}},\ \cdots,\ \overline{H_{m}}]
\end{eqnarray*} 
where $\Z[X,\ H_{1},\ \cdots,\ H_{m},\ \overline{H_{1}},\ \cdots,\ \overline{H_{m}}]$ 
is a polynomial ring which is generated by $X$ and all elements in $\mathbf{H}$, and put
\begin{eqnarray*}
\mathcal{I}=\Big\langle H_{i}+\overline{H_{i}}-X,\ \prod_{H\in \mathbf{H}'} H\ \Big|\ i=1,\ \cdots,\ m,\ \mathbf{H}'\in \mathbf{I}(\mathbf{H}) \Big\rangle
\end{eqnarray*}
which is the ideal in $\Z[X,\ \mathbf{H}]$
generated by $H_{i}+\overline{H_{i}}-X$ $(i=1,\ \cdots,\ m)$ and the product 
\begin{align*}
\prod_{H\in \mathbf{H}'\in \mathbf{I}(\mathbf{H})} H,
\end{align*}
where 
\begin{align*}\displaystyle
\mathbf{I}(\mathbf{H})=\{\mathbf{H}'\subset \mathbf{H}\ |\ \bigcap_{H\in \mathbf{H}'}H=\emptyset\}.
\end{align*}
We define the following ring $\Z[\mathcal{G}]$:
\begin{eqnarray*}
\Z[\mathcal{G}]:=\Z[X,\ \mathbf{H}] \Big/\mathcal{I}.
\end{eqnarray*}
From the next section, we shall prove this ring $\mathbb{Z}[\mathcal{G}]$ is isomorphic to the graph equivariant cohomology ring  $H^{*}(\mathcal{G})$ under some conditions.


\section{Ring structure of the graph equivariant cohomology of a $T^{*}\mathbb{C}^{n}$-modeled GKM graph}
\label{sect:4}
The first goal of this paper is to prove the following theorem.

\begin{theorem}
\label{main-theorem1}
Let $\mathcal{G}$ be a $2n$-valent $T^{*}\mathbb{C}^{n}$-modeled GKM graph and $\mathbf{L}=\{L_{1},\ \cdots,\ L_{m}\}$ be the set of all
hyperplanes in $\mathcal{G}$.
Assume that $\mathcal{G}$ satisfies the following two assumptions:
\begin{enumerate}
\item For each $L \in \mathbf{L}$, there exist the unique pair of the halfspace $H$ and its opposite side $\overline{H}$ such that $H \cap \overline{H}=L$;
\item For every subset $\mathbf{L}'\subset \mathbf{L}$, its intersection $\displaystyle\bigcap_{L\in\mathbf{L}'} L$ is empty or connected.
\end{enumerate}
Then the following ring isomorphism holds: 
\begin{align*}
H^{*}(\mathcal{G})\simeq \Z[\mathcal{G}].
\end{align*}
\end{theorem}
Henceforth in this section the $T^{*}\mathbb{C}^{n}$-modeled GKM graph $\mathcal{G}=(\Gamma,\ \alpha,\ \nabla)$ satisfies assumptions $(1)$, $(2)$ of Theorem~\ref{main-theorem1}.
For example, the left GKM graph in Figure~\ref{Figure2} satisfies these assumptions; 
however, the right GKM graph does not satisfy the assumption (1) (also see Figure~\ref{figure of opposite side2}).
We also note that the following example satisfies the assumptions in Theorem~\ref{main-theorem1}.
\begin{example}
\label{example-cotangent-bundle}
The GKM graph in Figure~\ref{fig-non-toricHK} can be obtained from the cotangent bundle of a $4$-dimensional toric manifold with five fixed points. 
Note that this can not be realized as a hyperplane arrangement in $\mathbb{R}^{2}$, because there must be $8$ intersection points (i.e., $8$ vertices) if all straight lines (five lines) extend to infinity but there are only $5$ vertices in Figure~\ref{fig-non-toricHK}.
Therefore, there is no corresponding toric hyperK${\rm\ddot{a}}$hler manifold because of the fundamental theorem of toric hyperK${\rm\ddot{a}}$hler manifolds in \cite{BD}.

\begin{figure}[h]
\begin{tikzpicture}
\begin{scope}[xscale=1.0, yscale=1.0]
\fill (-1,0) coordinate (r) circle (2pt);
\fill (1,-2) coordinate (q) circle (2pt); 
\fill (-1,-2) coordinate (p) circle (2pt); 
\fill (3,0) coordinate (s) circle (2pt);
\fill (3,4) coordinate (t) circle (2pt);

\draw (-2,-1)--(4,5);
\draw (-1,2)--(-1,-3);
\draw (-2,-2)--(3,-2);
\draw (3,5)--(3,-1);
\draw (0,-3)--(5,2);

\draw[->, thick] (p)--(-0.5,-2);
\node[below] at (-0.2,-2) {$e_{2}^{*}$};
\draw[->, thick] (p)--(-1,-1.5);
\node[left] at (-1,-1.2) {$e_{1}^{*}$};

\draw[->, thick] (q)--(1.5,-1.5);
\node[left] at (1.5,-1.5) {$e_{2}^{*}+e_{1}^{*}$};

\draw[->, thick] (r)--(-0.5,0.5);
\node[right] at (-0.5,0.5) {$e_{2}^{*}+e_{1}^{*}$};

\draw[->, thick] (s)--(3,0.5);
\node[left] at (3,0.5) {$e_{1}^{*}$}; 
\end{scope}
\end{tikzpicture}
\caption{The $T^{*}\mathbb{C}^{2}$-modeled GKM graph defined from the cotangent bundle of a toric manifold. 
Note that we assume that the axial functions satisfy $\alpha(e)=-\alpha(\overline{e})$ for all edges. We omit the axial functions on legs because it is automatically determined by the definition.}
\label{fig-non-toricHK}
\end{figure}
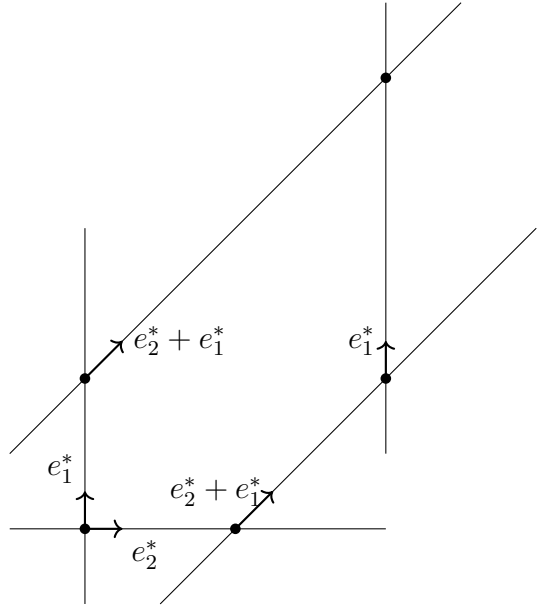
\end{example}

Let $\chi:\mathcal{V}\to H^{2}(BT)\oplus \mathbb{Z} x$ be the function such that $\chi(p)=x$ for all $p\in \mathcal{V}$, and 
$\tau_{H}$ be the Thom class of the halfspace $H$.
In order to prove Theorem~\ref{main-theorem1}, 
we will prove that the following map is an isomorphism:
\begin{eqnarray*}
\Psi:\Z[\mathcal{G}]\to H^{*}(\mathcal{G})
\end{eqnarray*}
where this map is the induced homomorphism from $\Psi(H):=\tau_{H}$ and $\Psi(X):=\chi$.

We first claim that the map $\Psi$ is well-defined (also see the definition of $\mathbb{Z}[\mathcal{G}]$). 
By Lemma~\ref{existence_of_opposite-side}, we have 
\begin{align*}
\tau_{H}+\tau_{\overline{H}}=\chi.
\end{align*}
Let $\mathbf{H}$ be the set of all halfspaces in $\mathcal{G}$.
If a subset $\mathbf{H}'\subset \mathbf{H}$ satisfies that 
$\displaystyle\bigcap_{H\in\mathbf{H}'}H=\emptyset$, then it follows from the definition of the Thom class that 
\begin{align*}
\prod_{H\in\mathbf{H}'}\tau_{H}=0.
\end{align*}
Therefore, the map $\Psi$ is a well-defined homomorphism.

From the next section, we start to prove the bijectivity of $\Psi$.
The proof will be divided into two steps: 
\begin{enumerate}
\item[(I)] To study an equivariant graph cohomology of an \textit{$x$-forgetful graph} $\mathcal{\widetilde{G}}$
and to prove $H^{*}(\widetilde{\mathcal{G}})\simeq \Z[\widetilde{\mathcal{G}}]$;
\item[(II)] To prove $\Psi$ is surjective and injective.
\end{enumerate}
In the first step, we will use the technique of \cite{MMP} (or \cite{MP}) which was used to show the ring structure of the graph equivariant cohomology of a certain GKM graph called a {\it torus graph}.
In the second step, we will use the technique of \cite{HP} which was applied to show the ring structure of the equivariant cohomology of a  toric hyperK${\rm \ddot{a}}$hler variety (also referred to as hypertoric variety).


\section{An $x$-forgetful graph $\mathcal{\widetilde{G}}$}
\label{sect:5}

Let $\mathcal{G}=(\Gamma,\ \alpha,\ \nabla)$ be a $T^{*}\mathbb{C}^{n}$-modeled graph. We assume that $\mathcal{G}$ satisfies the conditions in Theorem~\ref{main-theorem1}.
In this section, as a preparation to prove Theorem~\ref{main-theorem1}, we introduce an $x$-forgetful graph $\mathcal{\widetilde{G}}$ and its graph equivariant cohomology $H^{*}(\widetilde{\mathcal{G}})$, and prove the ring structure of $H^{*}(\widetilde{\mathcal{G}})$.


\subsection{$x$-forgetful graph $\mathcal{\widetilde{G}}$ and its graph equivariant cohomology}
\label{sect:5.1}

For every $\mathcal{G}$, we may define an \textit{$x$-forgetful graph} $\mathcal{\widetilde{G}}=(\Gamma,\ \widetilde{\alpha},\ \nabla)$ as follows:
$\Gamma$ and $\nabla$ is the same graph and connection with $\mathcal{G}$, but the function 
$\widetilde{\alpha}$ is defined as
\begin{eqnarray*}
\widetilde{\alpha}=F\circ \alpha:\mathcal{E}\to H^{2}(BT^{n})
\end{eqnarray*}
where $F:H^{2}(BT^{n})\oplus \Z x\to H^{2}(BT^{n})(\simeq (\algt^{n})^{*}_{\Z})$ is the  the $x$-forgetful map.
We call $\widetilde{\alpha}$ an {\it $x$-forgetful axial function}.

\begin{figure}[h]
\begin{tikzpicture}
\begin{scope}[xscale=1.0, yscale=1.0]
\fill (-1,1) coordinate (r) circle (2pt);
\fill (2,-2) coordinate (q) circle (2pt); 
\fill (-1,-2) coordinate (p) circle (2pt); 

\draw (-2,2)--(3,-3);
\draw (-1,2)--(-1,-3);
\draw (-2,-2)--(3,-2);

\draw[->, thick] (p)--(-0.5,-2);
\node[below] at (-0.2,-2) {$e_{2}^{*}$};
\draw[->, thick] (p)--(-1,-1.5);
\node[left] at (-1,-1.2) {$e_{1}^{*}$};
\draw[->, thick] (p)--(-1.5,-2);
\node[above] at (-2,-2) {$-e_{2}^{*}$};
\draw[->, thick] (p)--(-1,-2.5);
\node[left] at (-1,-2.8) {$-e_{1}^{*}$};

\draw[->, thick] (q)--(1.5,-2);
\node[below] at (1.2,-2) {$-e_{2}^{*}$};
\draw[->, thick] (q)--(2.5,-2);
\node[above] at (2.5,-2) {$e_{2}^{*}$};
\draw[->, thick] (q)--(2.5,-2.5);
\node[below] at (2.5,-2.5) {$e_{2}^{*}-e_{1}^{*}$};
\draw[->, thick] (q)--(1.5,-1.5);
\node[left] at (1.5,-1.5) {$e_{1}^{*}-e_{2}^{*}$};

\draw[->, thick] (r)--(-1,0.5);
\node[left] at (-1,0.3) {$-e_{2}^{*}$};
\draw[->, thick] (r)--(-1,1.5);
\node[right] at (-1,1.5) {$e_{1}^{*}$};
\draw[->, thick] (r)--(-1.5,1.5);
\node[left] at (-1.6,1.5) {$e_{1}^{*}-e_{2}^{*}$};
\draw[->, thick] (r)--(-0.5,0.5);
\node[right] at (-0.5,0.5) {$e_{2}^{*}-e_{1}^{*}$};
\end{scope}
\end{tikzpicture}
\caption{An example of the $x$-forgetful graph for the left GKM graph in Figure~\ref{Figure2}.}
\label{fig-x-forget}
\end{figure}
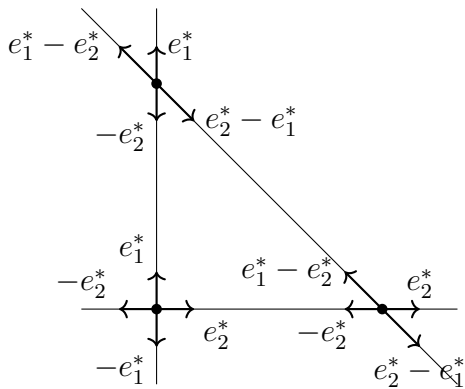

Moreover we define a graph equivariant cohomology of $\mathcal{\widetilde{G}}$ as follows:
\begin{eqnarray*}
H^{*}(\mathcal{\widetilde{G}})=\{f:\mathcal{V}\to H^{*}(BT^{n})\ |\ f(i(\epsilon))-f(t(\epsilon))\equiv 0\ ({\rm mod}\ \widetilde{\alpha}(\epsilon))\}.
\end{eqnarray*}
Let $L\in \mathbf{L}$ be a hyperplane in $\mathcal{G}$.
Fix the halfsapce $H$ such that 
$\partial H=L$.
Define the \textit{Thom class} of $L$ by
\begin{align*}
\tau_{L}=F\circ \tau_{H}:\mathcal{V}\to H^{2}(BT^{n}),
\end{align*} 
where $F$ is the $x$-forgetful map.
Note that for the opposite side $\overline{H}$ of $H$, the following relation:
\begin{align*}
F\circ \tau_{\overline{H}}=-\tau_{L};
\end{align*} 
therefore, the Thom class of $L$ depends on the choice of a halfspace $H$ with $\partial H=L$.
So we 
fix $\{H_{1},\ \cdots,\ H_{m}\}$ in the set of all halfspaces $\mathbf{H}=\{H_{1},\ \cdots,\ H_{m},\ \overline{H}_{1},\ \cdots,\ \overline{H}_{m}\}$.
By the assumption (1) of Theorem~\ref{main-theorem1},
there is a one to one corresponding between $H$ and $L=H\cap\overline{H}$.
Therefore, we may put the set of all hyperplanes by 
$\mathbf{L}=\{L_{1},\ \cdots,\ L_{m}\}$ where $L_{i}=H_{i}\cap \overline{H}_{i}$ for all 
$i=1,\ \cdots,\ m$.
Let $\mathcal{V}^{L}$ be the set of all vertices on $L$.
Then,  we have
\begin{eqnarray*}
\tau_{L}(p)=\left\{
\begin{array}{ll}
0 & p\not\in \mathcal{V}^{L} \\
\widetilde{\alpha}(n_{H}(p)) & p\in \mathcal{V}^{L}
\end{array} \right.
\end{eqnarray*}
by the definitions of $\tau_{H}$ and the $x$-forgetful map $F$,
where $n_{H}(p)$ is a normal edge (or leg) of $H$ on $p$.
Since $\tau_{H}\in H^{*}(\mathcal{G})$ (see Lemma~\ref{thom}), it is easy to check that 
\begin{align*}
\tau_{L}\in H^{*}(\mathcal{\widetilde{G}}).
\end{align*}

\subsection{The ring structure of $H^{*}(\mathcal{\widetilde{G}})$}
\label{sect:5.2}

Next we define the following ring:
\begin{eqnarray*}
\Z[\mathcal{\widetilde{G}}]=\Z[L_{1},\ \cdots,\ L_{m}]\Big/\Big\langle \prod_{L\in \mathbf{L}'}L\ \Big|\ \mathbf{L}'\in \mathbf{I}(\mathbf{L}) \Big\rangle,
\end{eqnarray*}
where
 $\displaystyle\mathbf{I}(\mathbf{L})=\{\mathbf{L}'\subset \mathbf{L}\ |\ \bigcap_{L\in\mathbf{L}'}L=\emptyset\}$
and $\langle \prod_{L\in \mathbf{L}'}L\ |\ \mathbf{L}'\in \mathbf{I}(\mathbf{L}) \rangle$ is an ideal which is generated by 
the product $\prod_{L\in \mathbf{L}'}L$ for all $\mathbf{L}'\in \mathbf{I}(\mathbf{L})$.
 
The goal of this section (the first step (I) of the proof of Theorem~\ref{main-theorem1}) is to prove 
the following theorem:
\begin{theorem}
\label{main-theorem1-1}
Let $\mathcal{G}$ be a $2n$-valent $T^{*}\mathbb{C}^{n}$-modeled GKM graph and $\mathbf{L}=\{L_{1},\ \cdots,\ L_{m}\}$ be the set of all
hyperplanes in $\mathcal{G}$.
Assume that $\mathcal{G}$ satisfies the two assumptions in Theorem~\ref{main-theorem1}. 
If $\widetilde{\mathcal{G}}$ is the $x$-forgetful graph, then 
the following ring isomorphism holds:
\begin{align*}
H^{*}(\mathcal{\widetilde{G}})\simeq \Z[\mathcal{\widetilde{G}}].
\end{align*}
\end{theorem}

Define the induced homomorphism  
\begin{eqnarray*}
\Psi':\Z[\mathcal{\widetilde{G}}]\to H^{*}(\mathcal{\widetilde{G}})
\end{eqnarray*}
by $\Psi'(L)=\tau_{L}$.
Obviously, $\Psi'$ is a well-defined homomorphism.
In order to show Theorem~\ref{main-theorem1-1}, 
it is enough to prove that this homomorphism is bijective.

\subsection{The localization map and the injectivity of $\Psi'$}
\label{sect:5.3}

We first prove the injectivity of $\Psi'$.
In order to prove it, we introduce the map $\rho$ which is the analogue of the localization of the equivariant cohomology of a $T$-manifold to its fixed points.  

Let us define the following ring:
\begin{eqnarray*}
\Z[\mathcal{\widetilde{G}}]_{p}=\Z[L_{1},\ \cdots,\ L_{m}]/\langle L\ |\ p\not\in \mathcal{V}^{L} \rangle,
\end{eqnarray*}
where $\langle L\ |\ p\not\in \mathcal{V}^{L} \rangle$ is an ideal which is generated by $L$ such that $p\not\in\mathcal{V}^{L}$.
As a beginning, we prove the following lemma.
\begin{lemma}
\label{restricted-to-pt}
For the $x$-forgetful graph $\mathcal{\widetilde{G}}=(\Gamma,\ \widetilde{\alpha},\ \nabla)$,
we have
\begin{eqnarray*}
I_{p}:\Z[\mathcal{\widetilde{G}}]_{p}\simeq \Z[L\ |\ p\in \mathcal{V}^{L}]=\Z[L_{p,1},\ \cdots,\ L_{p,n}]\stackrel{\iota_{p}}{\simeq} H^{*}(BT^{n}),
\end{eqnarray*}
where the last isomorphism $\iota_{p}$ is defined by $\iota_{p}:L\mapsto \tau_{L}(p)$.
\end{lemma} 
\begin{proof}
By the definition of $\Z[\mathcal{\widetilde{G}}]_{p}$, the first equivalence $\Z[\mathcal{\widetilde{G}}]_{p}\simeq \Z[L\ |\ p\in \mathcal{V}^{L}]$ is obvious.
We claim $\Z[L\ |\ p\in \mathcal{V}^{L}]=\Z[L_{p,1},\ \cdots,\ L_{p,n}]\stackrel{\iota_{p}}{\simeq} H_{T^{n}}^{*}(pt)$.

Because $\Gamma$ is a $2n$-valent graph,
we may put
\begin{eqnarray*}
\mathcal{E}_{p}=\{\epsilon_{1}^{+}(p),\ \cdots,\ \epsilon_{n}^{+}(p),\ \epsilon_{1}^{-}(p),\ \cdots,\ \epsilon_{n}^{-}(p) \}
\end{eqnarray*}
 for all $p\in \mathcal{V}$.
There is a unique $L_{i}$ such that 
\begin{align*}
\tau_{L_{i}}(p)=\widetilde{\alpha}(\epsilon_{i}^{+}(p))=-\widetilde{\alpha}(\epsilon_{i}^{-}(p))
\end{align*} 
for all $i=1,\ \cdots,\ n$ by Lemma \ref{key-properties-hyperplane}.
Hence, we have 
\begin{align*}
\Z[L\ |\ p\in \mathcal{V}^{L}]=\Z[L_{p,1},\ \cdots,\ L_{p,n}].
\end{align*}
Next, 
by the definition of the axial function of a $T^{*}\mathbb{C}^{n}$-modeled GKM graph, 
\begin{eqnarray*}
\mathbb{Z}\alpha(\epsilon_{1}^{+}(p))\oplus \cdots\oplus \mathbb{Z}\alpha(\epsilon_{n}^{+}(p))\oplus \mathbb{Z}x\simeq H^{2}(BT^{n})\oplus\Z x.
\end{eqnarray*}
Hence, because $\widetilde{\alpha}:=F\circ \alpha$ is defined by the $x$-forgetful map $F:H^{2}(BT^{n})\oplus\Z x\to H^{2}(BT^{n})$, we have that 
\begin{eqnarray*}
\Z[ \widetilde{\alpha}(\epsilon_{1}^{+}(p)),\ \cdots,\ \widetilde{\alpha}(\epsilon_{n}^{+}(p)) ]\simeq H^{*}(BT^{n}).
\end{eqnarray*}
Therefore, 
$\iota_{p}$ is an isomorphism.
\end{proof}

Next we shall define a {\it localization map} $\rho:\Z[\mathcal{\widetilde{G}}]\to \bigoplus_{p\in\mathcal{V}}\Z[\mathcal{\widetilde{G}}]_{p}$ and prove that it is injective in Lemma \ref{rho_inje}.
Since the set $\mathbf{L}'\in \mathbf{I}(\mathbf{L})$ satisfies that  $\cap_{L\in \mathbf{L}'}L=\emptyset$, for every $p\in \mathcal{V}$ 
there is an $L\in \mathbf{L}'$ such that $p\not\in \mathcal{V}^{L}$. Therefore, there exists the following relation for two ideals in $\mathbb{Z}[L_{1},\ldots, L_{m}]$: 
\begin{eqnarray*}
\langle L\ |\ p\not\in \mathcal{V}^{L} \rangle \supset \langle \prod_{L\in \mathbf{L}'}L\ |\ \mathbf{L}'\in \mathbf{I}(\mathbf{L}) \rangle.
\end{eqnarray*}
Hence, the following natural homomorphism is well-defined:
\begin{eqnarray*}
\rho_{p}:\Z[\mathcal{\widetilde{G}}]:=\Z[L_{1},\ \cdots,\ L_{m}]\Big/\Big\langle \prod_{L\in \mathbf{L}'}L\ \Big|\ \mathbf{L}'\in \mathbf{I}(\mathbf{L}) \Big\rangle\longrightarrow \Z[\mathcal{\widetilde{G}}]_{p}:=\Z[L_{1},\ \cdots,\ L_{m}]/\langle L\ |\ p\not\in \mathcal{V}^{L} \rangle.
\end{eqnarray*}
For this projection $\rho_{p}$, we can easily show that its kernel is as follows: 
\begin{eqnarray*}
{\rm Ker}\ \rho_{p}=\langle L\ |\ p\not\in \mathcal{V}^{L} \rangle/\langle \prod_{L\in \mathbf{L}'}L\ |\ \mathbf{L}'\in \mathbf{I}(\mathbf{L}) \rangle.
\end{eqnarray*}
Now we may define the homomorphism $\rho$ as follows:  
\begin{eqnarray*}
\rho=\bigoplus_{p\in \mathcal{V}}\rho_{p}:\Z[\mathcal{\widetilde{G}}]\longrightarrow \bigoplus_{p\in\mathcal{V}}
\Z[\mathcal{\widetilde{G}}]_{p},
\end{eqnarray*}
such that 
\begin{align*}
\rho(Y)=\bigoplus_{p\in \mathcal{V}}\rho_{p}(Y)
\end{align*}
 for $Y\in \Z[\mathcal{\widetilde{G}}]$.
We call $\rho$ a {\it localization map}.
 The following lemma holds. 
\begin{lemma}
\label{rho_inje}
$\rho$ is injective.
\end{lemma} 
\begin{proof}
Obviously we have 
\begin{eqnarray*}
{\rm Ker}\ \rho=\bigcap_{p\in \mathcal{V}} {\rm Ker}\ \rho_{p}=\left(\bigcap_{p\in \mathcal{V}}\ \langle L\ |\ p\not\in \mathcal{V}^{L} \rangle\right)
\Bigg/\left\langle \prod_{L\in \mathbf{L}'}L\ \Bigg|\ \mathbf{L}'\in \mathbf{I}(\mathbf{L}) \right\rangle.
\end{eqnarray*}
Hence, to prove $\rho$ is injective, it is enough to show that ${\rm Ker}\ \rho=\{0 \}$, i.e., we shall prove the following relation: 
\begin{eqnarray}
\label{rel_for_rho_inj}
\bigcap_{p\in \mathcal{V}}\ \langle L\ |\ p\not\in \mathcal{V}^{L} \rangle\subset \left\langle \prod_{L\in \mathbf{L}'}L\ |\ \mathbf{L}'\in \mathbf{I}(\mathbf{L}) \right\rangle
(\subset \Z[L_{1},\ \cdots,\ L_{m}]).
\end{eqnarray} 
Take a non-zero polynomial 
\begin{eqnarray*}
A&=&\sum_{a_{1},\cdots,a_{m}\in \N\cup \{0\}} k(a_{1},\ \cdots,\ a_{m}) L_{1}^{a_{1}}\cdots L_{m}^{a_{m}} \\
&\in& \bigcap_{p\in \mathcal{V}}\ \langle L\ |\ p\not\in \mathcal{V}^{L} \rangle \subset \Z[L_{1},\ \cdots,\ L_{m}],
\end{eqnarray*}
where we only consider the case when $k(a_{1},\ \cdots,\ a_{m})\in \Z-\{0\}$.
Because $A$ is an element of the monomial ideal $\langle L\ |\ p\not\in \mathcal{V}^{L} \rangle$ for all $p\in \mathcal{V}$, we have that
for each term 
\begin{align*}
k(a_{1},\ \cdots,\ a_{m}) L_{1}^{a_{1}}\cdots L_{m}^{a_{m}}\in \langle L\ |\ p\not\in \mathcal{V}^{L} \rangle.
\end{align*}
This shows that for each term $k(a_{1},\ \cdots,\ a_{m}) L_{1}^{a_{1}}\cdots L_{m}^{a_{m}}$ of a non-zero element $A$ there exists 
$r(=r(p))\in\{1,\ \cdots,\ m\}$ such that $p\not\in \mathcal{V}^{L_{r}}$ and $a_{r}\not=0$.
Because this satisfies for all $p\in \mathcal{V}$, we have that each term can be written by 
\begin{eqnarray*}
k(a_{1},\ \cdots,\ a_{m}) L_{1}^{a_{1}}\cdots L_{m}^{a_{m}}
=B\prod_{p\in \mathcal{V}}L_{r(p)}^{a_{r(p)}},
\end{eqnarray*}
where $B$ is some monomial in $\Z[L_{1},\ \cdots,\ L_{m}]$ and $a_{r(p)}\not=0$.
Since $p\not\in\mathcal{V}^{L_{r(p)}}$, we have that 
\begin{align*}
\bigcap_{p\in \mathcal{V}}L_{r(p)}=\emptyset.
\end{align*}
This shows that for each term of $A$
\begin{align*}
k(a_{1},\ \cdots,\ a_{m}) L_{1}^{a_{1}}\cdots L_{m}^{a_{m}}=B\prod_{p\in \mathcal{V}}L_{r(p)}^{a_{r(p)}}\in 
\langle \prod_{L\in \mathbf{L}'}L\ |\ \mathbf{L}'\in \mathbf{I}(\mathbf{L}) \rangle.
\end{align*}
Therefore, $A\in \langle \prod_{L\in \mathbf{L}'}L\ |\ \mathbf{L}'\in\mathbf{I}(\mathbf{L}) \rangle$.
This establishes the relation \eqref{rel_for_rho_inj}.
\end{proof}

By using Lemma \ref{restricted-to-pt} and \ref{rho_inje}, we can prove the following lemma for the homomorphism $\Psi':\mathbb{Z}[\widetilde{\mathcal{G}}]\to H^{*}(\widetilde{\mathcal{G}})$ which is defined from $\Psi'(L):=\tau_{L}$.

\begin{lemma}
\label{Psi'_inj}
$\Psi'$ is injective.
\end{lemma}
\begin{proof}
We first define 
\begin{align*}
\rho':H^{*}(\mathcal{\widetilde{G}})\to \bigoplus_{p\in \mathcal{V}}H^{*}(BT^{n})
\end{align*}
by the homomorphism 
\begin{align*}
\rho'(f)=\bigoplus_{p\in \mathcal{V}} f(p).
\end{align*}
Then it is easy to check that the following diagram is commutative:
\begin{eqnarray*}
\xymatrix{
\Z[\mathcal{\widetilde{G}}]\ar[r]^(.4){\rho} \ar[d]^{\Psi'} & \bigoplus_{p\in\mathcal{V}}\Z[\mathcal{\widetilde{G}}]_{p} \ar[d]^{\oplus_{p}I_{p}} \\
H^{*}(\mathcal{\widetilde{G}})\ar[r]^(.4){\rho'} & \bigoplus_{p\in \mathcal{V}}H^{*}(BT^{n})
}
\end{eqnarray*}
where $I_{p}:\mathbb{Z}[\widetilde{\mathcal{G}}]_{p}\to H^{*}(BT^{n})$ is the isomorphism defined by $I_{p}(L):=\tau_{L}(p)$ in Lemma~\ref{restricted-to-pt}.
Because of Lemma~\ref{rho_inje}, $\rho$ is injective. 
Therefore, the composition map $\oplus_{p}I_{p}\circ \rho$ is injective.
Because of the commutativity of the diagram, 
$\rho'\circ \Psi'=\oplus_{p}I_{p}\circ \rho$ is also injective.
Consequently, $\Psi'$ is injective. 
\end{proof}

\subsection{The surjectivity of $\Psi'$}
\label{sect:5.4}

We next prove the surjectivity of $\Psi'$. 
In order to prove it, we will define an ideal $I(K)$ of $H^{*}(BT^{n})$,
where $K$ is the non-empty intersection of some hyperplanes, say $K=L_{1}\cap\cdots \cap L_{k}(\not=\emptyset)$.
Note that the graph $K$ is connected because of the assumption (2) of Theorem~\ref{main-theorem1}.
Because $L_{1},\ldots, L_{k}$ defines hyperplanes $\mathcal{L}_{1},\ldots, \mathcal{L}_{k}$ (respectively) of $\mathcal{G}=(\Gamma,\ \alpha,\ \nabla)$, 
the subgraph $K$ is also defines a $(2n-2k)$-valent ($T^{*}\mathbb{C}^{n-k}$-modeled) GKM subgraph of $\mathcal{G}$, say 
$\mathcal{K}:=(K,\ \alpha^{K},\nabla^{K})$.
Now we may define its $x$-forgetful graph, i.e., 
for $\widetilde{\alpha}^{K}:=F\circ \alpha^{K}$, the labeled graph  
\begin{align*}
\widetilde{\mathcal{K}}:=(K,\ \widetilde{\alpha}^{K},\nabla^{K}).
\end{align*}
We define an ideal $I(K)$ (in $H^{*}(BT^{n})$) on $K$ as follows:
\begin{align*}
I(K)=\langle \widetilde{\alpha}^{K}(\epsilon)(=\widetilde{\alpha}(\epsilon))\ |\ \epsilon\in \mathcal{E}^{K}\rangle,
\end{align*}
that is, this ideal is generated by all $x$-forgetful axial functions of edges and legs in $K$.
The following lemma, which will be used to prove the surjectivity of $\Psi'$, holds for $I(K)$.
\begin{lemma}
\label{key-lemma-surj-Psi'}
Let $f$ be an element in $H^{*}(\mathcal{\widetilde{G}})$.
If $f(p)\not\in I(K)$ for some $p\in \mathcal{V}^{K}$, then $f(q)\not\in I(K)$ for all $q\in \mathcal{V}^{K}$.
\end{lemma}
\begin{proof}
Let $K:=(\mathcal{V}^{K},\mathcal{E}^{K})$. 
For $f\in H^{*}(\widetilde{\mathcal{G}})$, 
we assume that $f(p)\not\in I(K)$ for some $p\in \mathcal{V}^{K}$.
We also assume that there exists a vertex $q\in \mathcal{V}^{K}$ such that  $f(q)\in I(K)$.
Since $K$ is connected, there is a path in $K$ from $q$ to $p$, which consists of edges 
\begin{align*}
qr_{1},\ r_{1}r_{2},\ \cdots,\ r_{s-1}r_{s},\ r_{s}p\in E^{K}\subset \mathcal{E}^{K}.
\end{align*}
Because of
the congruence relations in $H^{*}(\widetilde{\mathcal{G}})$, 
there are $A_{1},\ldots, A_{s+1}\in H^{*}(BT^{n})$ such that 
\begin{align*}
 &f(q)-f(p) \\
=&(f(q)-f(r_{1}))+(f(r_{1})-f(r_{2}))+\cdots +(f(r_{s-1})-f(r_{s}))+(f(r_{s})-f(p)) \\
=&A_{1}\widetilde{\alpha}(qr_{1})+A_{2}\widetilde{\alpha}(r_{1}r_{2})\cdots +A_{s}\widetilde{\alpha}(r_{s-1}r_{s})+A_{s+1}\widetilde{\alpha}(r_{s}p).
\end{align*}
Therefore, by the definition of $I(K)$, we have 
\begin{align*}
 f(q)-f(p) \in  I(K).
\end{align*}
However, 
since $f(q),\ A_{1}\widetilde{\alpha}(qr_{1}), \cdots,\  A_{s+1}\widetilde{\alpha}(r_{s}p)\in I(K)$, we have $f(p)\in I(K)$.
This gives a contradiction.
This established that if $f(p)\not\in I(K)$ then $f(q)\not\in I(K)$ for all $q\in \mathcal{V}^{K}$.
\end{proof}

By using this lemma, we can prove the surjectivity of $\Psi':\mathbb{Z}[\widetilde{\mathcal{G}}]\to H^{*}(\widetilde{\mathcal{G}})$.
\begin{lemma}
\label{Psi'_surj}
$\Psi'$ is surjective.
\end{lemma}
\begin{proof}
Let $f\in H^{*}(\widetilde{\mathcal{G}})$. 
For some $p\in \mathcal{V}$, 
we assume that $f(p)\in H^{*}(BT^{n})$ has a non-zero constant term $k\in \Z-\{0\}$, i.e.,
\begin{align*}
f(p)=k+g(p) 
\end{align*}
where $g(p)\in H^{>0}(BT^{n})\cup \{0\}$.
Note that $H^{*}(BT^{n})\simeq \mathbb{Z}[x_{1},\ldots, x_{n}]$, where $\deg x_{i}=2$ for all $i=1,\ldots n$.
Because $f\in H^{*}(\mathcal{\widetilde{G}})$ satisfies the congruence relation, there exists $g\in H^{>0}(\mathcal{\widetilde{G}})\cup \{0\}$ such that for all $q\in \mathcal{V}$ we may write 
\begin{align*}
f(q)=k+g(q), 
\end{align*}
where $H^{>0}(\mathcal{\widetilde{G}})\cup \{0\}$ is the set of 
$g\in H^{*}(\mathcal{\widetilde{G}})$ whose constant term is $0$, i.e., for all $p\in \mathcal{V}$ the constant term of the polynomial $g(p)\in \mathbb{Z}[x_{1},\ldots, x_{n}]$ is $0$. 
This shows that for all $f\in H^{*}(\widetilde{\mathcal{G}})$
there exists the constant term $k$ and $g\in H^{>0}(\mathcal{\widetilde{G}})\cup \{0\}$ such that 
\begin{align*}
f=k+g.
\end{align*}
Therefore,   
we can take $k\in \Z\subset \Z[\mathcal{\widetilde{G}}]$ such that 
\begin{align*}
f=\Psi'(k)+g.
\end{align*}

Take $g=f-\Psi'(k)$.
Then $g(p)\in H^{>0}(BT^{n})\cup \{0\}$ for all $p\in \mathcal{V}$.
Now we may put 
\begin{align*}
Z(g)=\{p\in \mathcal{V}\ |\ g(p)=0\}.
\end{align*} 
We first assume that $Z(g)=\emptyset$. 
Then $g(p)\not= 0$ for all $p\in \mathcal{V}$.
Note that by Lemma \ref{restricted-to-pt} we have   
\begin{align*}
g(p)(\not=0)\in H^{*}(BT^{n})=\Z[\tau_{L_{p,1}}(p),\ \cdots,\ \tau_{L_{p,n}}(p)], 
\end{align*}
where $L_{p,i}$, $i=1,\ldots,n$, are the hyperplanes such that
$p\in \mathcal{V}^{L_{p,i}}$.  This also shows that for the fixed vertex
$p\in \mathcal{V}$, we may take an element
\begin{align*}
A\in \Z[\mathcal{\widetilde{G}}]
\end{align*} 
such that 
\begin{align*}
\Psi'(A)(p)=g(p).
\end{align*}
Because $g-\Psi'(A)\in H^{*}(\widetilde{\mathcal{G}})$ and $g(p)-\Psi'(A)(p)=0$, we have that 
\begin{align*}
p\in Z(g-\Psi'(A)).
\end{align*}
Next, 
by taking $h=g-\Psi'(A)=f-\Psi'(k+A)$, 
we may assume that $Z(h)\not=\emptyset$.
Take $p\in \mathcal{V}\backslash Z(h)$, i.e., $h(p)\not=0$.
Let $a\tau_{L_{1}}^{a_{p,1}}\cdots\tau_{L_{n}}^{a_{p,n}}(p)$ be a monomial appearing in $h(p)$, where $a$ is a non-zero integer,
$p\in \mathcal{V}^{L_{p,i}}$ and $a_{i}\ge 0$ ($i=1,\ \cdots,\ n$).
Since $h(p)\in H^{>0}(BT^{n})$, we may assume that 
\begin{align*}
a_{1},\ \cdots,\ a_{b}\not=0,\quad a_{b+1}=\cdots =a_{n}=0.
\end{align*}
Put $K=\cap_{i=1}^{b}L_{p,i}$.
Then 
we have
\begin{align*}
h(p)\not\in I(K)=\langle \widetilde{\alpha}^{K}(\epsilon)\ |\ \epsilon\in \mathcal{E}^{K}\rangle\subset H^{*}(BT^{n})
\end{align*} 
because $h(p)$ contains the non-zero monomial $a\tau_{L_{p,1}}^{a_{1}}\cdots\tau_{L_{p,b}}^{a_{b}}(p)$ such that 
$\tau_{L_{p,i}}(p)$ ($i=1,\ \cdots,\ b$) is defined by the axial function of the normal edge or leg of $K$ on $p$ (which are not the edges or legs in $\mathcal{E}^{K}$).
Therefore, by  Lemma \ref{key-lemma-surj-Psi'}, 
we have that for all $q\in \mathcal{V}^{K}$,
\begin{align*}
h(q)\not\in I(K).
\end{align*}
In particular, $h(q)\not=0$ for all $q\in \mathcal{V}^{K}$.
Let $r\not\in \mathcal{V}^{K}$.
Because $K=L_{p,1}\cap \cdots \cap L_{p,b}$, 
 we see that 
\begin{align*}
a\tau_{L_{p,1}}^{a_{1}}\cdots\tau_{L_{p,b}}^{a_{b}}(r)=0.
\end{align*}
Therefore, if we 
put 
\begin{align*}
h'=h-a\tau_{L_{p,1}}^{a_{1}}\cdots\tau_{L_{p,b}}^{a_{b}}=h-\Psi'(aL_{p,1}^{a_{1}}\cdots L_{p,b}^{a_{b}})=f-\Psi'(k+A+aL_{p,1}^{a_{1}}\cdots L_{p,b}^{a_{b}}),
\end{align*}
 then $h'(r)=h(r)$ for all $r\not\in \mathcal{V}^{K}$.
Namely, $h(q)\not=0$ for all $q\in \mathcal{V}^{K}$ and $h'(r)=h(r)$ for all $r\not\in \mathcal{V}^{K}$.
This shows that 
\begin{align*}
Z(h')\supset Z(h).
\end{align*}
Note that by the definition of $h'$, the number of monomials in $h'(p)$ is strictly smaller than that in $h(p)$.
If $h'(p)=0$, then we have $Z(h')\supsetneq Z(h).$
If $h'(p)\not=0$, then  we may apply the same argument as above for $h'\in H^{*}(\widetilde{\mathcal{G}})$ and the vertex $p\in \mathcal{V}$ again because $Z(h')\not=\emptyset$.
Then we have that there exists hyperplanes $L_{p,i_{1}},\ldots, L_{p,i_{c}}$ in $\{L_{p,1},\ldots, L_{p,n}\}$ and a non-zero integer $a'$ such that 
\begin{align*}
h''=h'-\Psi'(a'L_{p,i_{1}}^{a'_{1}}\cdots L_{p,i_{c}}^{a'_{c}})
\end{align*}
which satisfies that 
\begin{align*}
Z(h'')\supset Z(h') 
\end{align*}
and 
the number of monomials in $h''(p)$ is strictly smaller than that in $h'(p)$,
where $a'_{1},\ldots, a'_{c}$ are positive integers.
If $h''(p)\not=0$, then we repeat the same argument again.
Because the number of monomials in $h(p)$ is strictly smaller than smaller in each step, finally we have an element 
\begin{align*}
B\in \Z[\mathcal{\widetilde{G}}]
\end{align*}
 such that 
\begin{align*}
Z(h-\Psi'(B))\supsetneq Z(h).
\end{align*}
Moreover repeating this procedure, we can find an element $C\in \Z[\mathcal{\widetilde{G}}]$ such that $Z(h-\Psi'(C))=\mathcal{V}$.
This shows that 
\begin{align*}
h-\Psi'(C)=f-\Psi'(k+A+C)=0.
\end{align*} 
Therefore,
for all $f\in H^{*}(\widetilde{\mathcal{G}})$ there exists an element 
$k+A+C\in \mathbb{Z}[\widetilde{\mathcal{G}}]$ such that 
$f=\Psi'(k+A+C)$.
This establishes that $\Psi'$ is surjective.
\end{proof}

Consequently $\Psi'$ is an isomorphic map by Lemma \ref{Psi'_inj} and \ref{Psi'_surj}, and we have 
\begin{align*}
H^{*}(\mathcal{\widetilde{G}})\simeq \Z[\mathcal{\widetilde{G}}].
\end{align*}
This establishes Theorem~\ref{main-theorem1-1}.


\begin{remark}
\label{rem_after_surj}
From the above argument, we know that the assumption (2) of
Theorem~\ref{main-theorem1} is not needed to prove the ``injectivity''
of $\Psi'$; however, it is needed to prove the ``surjectivity'' of
$\Psi'$.  Hence, the assumption (2) of Theorem~\ref{main-theorem1}
means that $H^{*}(\mathcal{\widetilde{G}})$
(resp. $H^{*}(\mathcal{G})$) is generated by elements of
$H^{2}(\mathcal{\widetilde{G}})$ (resp. $H^{2}(\mathcal{G})$), that
is, $\tau_{L}\in H^{2}(\mathcal{\widetilde{G}})$ (resp.
$\tau_{H},\ \chi\in H^{2}(\mathcal{G})$).  
For example, the Figure~\ref{sphere} shows the $T^{*}\mathbb{C}^{2}$-modeled GKM graph which does not satisfy the assumption (2) of
Theorem~\ref{main-theorem1} and its $x$-forgetful graph.
In this case, we need a generator which is not in $H^{2}(\widetilde{\mathcal{G}})$.
\end{remark}

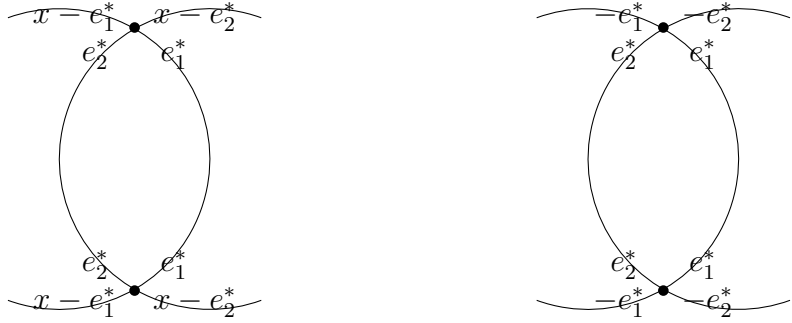
\begin{figure}[h]
\begin{tikzpicture}
\begin{scope}[xscale=1.0, yscale=1.0]
\draw ([shift={(1,0)}]70:2) arc[radius=2, start angle=70, end angle= 290];
\draw ([shift={(-1,0)}]-110:2) arc[radius=2, start angle=-110, end angle= 110];

\node[left] at (-0.2,1.4) {$e_{2}^{*}$};
\node[right] at (0.2,1.4) {$e_{1}^{*}$};
\node[left] at (-0.1,1.9) {$x-e_{1}^{*}$};
\node[right] at (0.1,1.9) {$x-e_{2}^{*}$};
\fill (0,1.75) circle (2pt);

\node[left] at (-0.2,-1.4) {$e_{2}^{*}$};
\node[right] at (0.2,-1.4) {$e_{1}^{*}$};
\node[left] at (-0.1,-1.9) {$x-e_{1}^{*}$};
\node[right] at (0.1,-1.9) {$x-e_{2}^{*}$};
\fill (0,-1.75) circle (2pt);

\end{scope}

\begin{scope}[xshift=200, xscale=1.0, yscale=1.0]
\draw ([shift={(1,0)}]70:2) arc[radius=2, start angle=70, end angle= 290];
\draw ([shift={(-1,0)}]-110:2) arc[radius=2, start angle=-110, end angle= 110];

\node[left] at (-0.2,1.4) {$e_{2}^{*}$};
\node[right] at (0.2,1.4) {$e_{1}^{*}$};
\node[left] at (-0.1,1.9) {$-e_{1}^{*}$};
\node[right] at (0.1,1.9) {$-e_{2}^{*}$};
\fill (0,1.75) circle (2pt);

\node[left] at (-0.2,-1.4) {$e_{2}^{*}$};
\node[right] at (0.2,-1.4) {$e_{1}^{*}$};
\node[left] at (-0.1,-1.9) {$-e_{1}^{*}$};
\node[right] at (0.1,-1.9) {$-e_{2}^{*}$};
\fill (0,-1.75) circle (2pt);

\end{scope}
\end{tikzpicture}
\caption{An example of the $T^{*}\mathbb{C}^{2}$-modeled graph and its $x$-forgetful graph which does not satisfy the assumption (2) in Theorem~\ref{main-theorem1}. Geometrically, this graph can be defined from $T^{*}S^{4}$ with the $T^{2}\times S^{1}$-action}.
\label{sphere}
\end{figure}


\section{Proof of Theorem~\ref{main-theorem1}}
\label{sect:6}

In this section, we prove Theorem~\ref{main-theorem1}.
We first recall the statement of Theorem~\ref{main-theorem1}.
Let $\mathcal{G}$ be a $2n$-valent $T^{*}\mathbb{C}^{n}$-modeled GKM graph and $\mathbf{L}=\{L_{1},\ \cdots,\ L_{m}\}$ be the set of all hyperplanes in $\mathcal{G}$.
Assume the following two assumptions for $\mathcal{G}$:
\begin{enumerate}
\item For each $L \in \mathbf{L}$, there exist the unique pair of the halfspace $H$ and its opposite side $\overline{H}$ such that $H \cap \overline{H}=L$;
\item For every subset $\mathbf{L}'\subset \mathbf{L}$, its intersection $\displaystyle\bigcap_{L\in\mathbf{L}'} L$ is empty or connected.
\end{enumerate}
Then, we will prove the following isomorphism:
\begin{align*}
\Z[\mathcal{G}]\simeq H^{*}(\mathcal{G}).
\end{align*}
Recall the ring homomorphism in Section~\ref{sect:4}
\begin{align*}
\Psi:\Z[\mathcal{G}]\to H^{*}(\mathcal{G})
\end{align*} 
is defined by 
\begin{align}
\label{def_of_Psi}
\Psi(X)=\chi,\quad \Psi(H)=\tau_{H}.
\end{align}
To prove Theorem~\ref{main-theorem1}, 
we claim that $\Psi$ is an isomorphism.

By the assumption (1) of Theorem~\ref{main-theorem1}, we can put the set of all halfspaces in $\mathcal{G}$ by 
\begin{eqnarray*}
\mathbf{H}=\{H_{1},\ \cdots,\ H_{m},\ \overline{H_{1}},\ \cdots,\ \overline{H_{m}}\},
\end{eqnarray*} 
where $L_{i}=H_{i}\cap \overline{H_{i}}$.
We prepare the following diagram:
\begin{align}
\label{diagram}
\xymatrix{
\Z[X,H_{1},\cdots,H_{m},\overline{H}_{1},\cdots,\overline{H_{m}}]\ar[r]^(.7){\hat{\pi}}\ar[d]^{\phi'} & \Z[\mathcal{G}]\ar[d]^{\Psi} \\
\Z[X,H_{1},\cdots,H_{m}]\ar[r]^(.7){\pi}\ar[d]^{\phi} & H^{*}(\mathcal{G})\ar[d]^{\widetilde{F}} \\
\Z[L_{1},\cdots,L_{m}]\ar[r]^(.7){\pi'} & H^{*}(\mathcal{\widetilde{G}})
}
\end{align}
where the maps in the diagram is defined as follows:
\begin{itemize}
\item $\hat{\pi}$ is the natural projection; 
\item $\Psi$ is defined by \eqref{def_of_Psi} as before;  
\item $\phi'$ is the surjective homomorphism induced from 
\begin{align*}
\phi'(X)=X,\quad \phi'(H_{i})=H_{i},\quad \phi'(\overline{H_{i}})=X-H_{i},\ i=1,\ldots, m;
\end{align*}
\item $\pi$ is the homomorphism induced from 
\begin{align*}
\pi(X)=\chi,\quad  \pi(H_{i})=\tau_{H_{i}},\ i=1,\ldots, m;
\end{align*}
\item $\widetilde{F}$ is the homomorphism defined by 
\begin{align*}
\widetilde{F}(f)(p):=F\circ f(p)
\end{align*}
 for $f\in H^{*}(\mathcal{G})$ and $p\in \mathcal{V}$, 
where $F:H^{*}(BT^{n}\times BS^{1})\to H^{*}(BT^{n})$ is the $x$-forgetful map for the fixed generator $x$ of $H^{2}(BS^{1})\simeq \Z x$;
\item $\phi$ is the surjective homomorphism induced from 
\begin{align*}
\phi(X)=0,\quad \phi(H_{i})=L_{i},\ i=1, \ldots, m;
\end{align*}
\item $\pi'$ is the homomorphism induced from 
\begin{align*}
\pi'(L_{i})=\tau_{L_{i}},\ i=1,\ldots, m.
\end{align*}
\end{itemize}
It easily follows from the definitions of homomorphisms as above and 
Lemma~\ref{existence_of_opposite-side} that  the top diagram is commutative.
By Section~\ref{sect:5.1}, we may choose $H_{1},\ldots, H_{m}$ as $\tau_{L_{i}}=F\circ \tau_{H_{i}}$ for $i=1,\ldots,m$.
Therefore, we may assume that the bottom diagram is also commutative.
Therefore, this diagram is commutative.

By the proof of Theorem~\ref{main-theorem1-1}, i.e., $\mathbb{Z}[\widetilde{\mathcal{G}}]\simeq H^{*}(\widetilde{\mathcal{G}})$, we have that 
$\pi'$ is surjective. 
This shows that $\pi'\circ \phi=\widetilde{F}\circ \pi$ is surjective; therefore, $\widetilde{F}$ is also surjective. 

\subsection{Surjectivity of $\Psi$}
\label{sect:6.1}

We first prove the surjectivity of $\Psi$.
By the commutativity of the top diagram, 
it is enough to prove that 
the homomorphism 
\begin{align*}
\pi:\mathbb{Z}[X,H_{1},\ldots, H_{m}]\to H^{*}(\mathcal{G}) 
\end{align*}
is surjective
(see Lemma~\ref{surj_pi}).
To do that, 
we will prove the following three lemmas.

The following first lemma is about the kernel of $\widetilde{F}:H^{*}(\mathcal{G})\to H^{*}(\widetilde{\mathcal{G}})$.
\begin{lemma}
\label{kernel_of_Ftilde}
Let $\chi$ be the element in $H^{*}(\mathcal{G})$ such that $\chi(p)=x$ for all $p\in \mathcal{V}$, where $x$ is the residual basis.
Then we have ${\rm Ker}\ \widetilde{F}=\langle \chi\rangle$,
i.e., the ideal generated by $\chi$. 
\end{lemma}
\begin{proof}
Let $f\in {\rm Ker}\ \widetilde{F}$.
By the definition of $\widetilde{F}$, we have $\widetilde{F}(f)(p)=F\circ f(p)=0$ for all $p\in \mathcal{V}$.
Since $F:H^{*}(BT^{n}\times BS^{1})=\Z[\alpha_{1},\ \cdots,\ \alpha_{n},\ x]\to \Z[\beta_{1},\ \cdots,\ \beta_{n}]=H^{*}(BT^{n})$ is defined by
$F(x)=0$ and $F(\alpha_{i})=\beta_{i}$ for all $i=1,\ \cdots,\ n$,
we have 
\begin{align*}
f(p)\in {\rm Ker}\ F=\langle x\rangle\subset H^{*}(BT^{n}\times BS^{1}).
\end{align*}
Therefore, for every $p\in \mathcal{V}$, there exists a polynomial $g(p)\in H^{*}(BT^{n}\times BS^{1})$ such that  
\begin{align*}
f(p)=g(p)x.
\end{align*}
Because $f\in H^{*}(\mathcal{G})$, it satisfies the congruence relation
\begin{align*}
f(i(\epsilon))-f(t(\epsilon))=g(i(\epsilon))x-g(t(\epsilon))x=(g(i(\epsilon))-g(t(\epsilon)))x\equiv 0\ ({\rm mod}\ \alpha(\epsilon))
\end{align*}
for every edge $\epsilon$.
Because $x$ is the residual basis, by definition of
$T^{*}\mathbb{C}^{n}$-modeled GKM graph (see Definition \ref{def_modeled_graph}), we see that $\alpha(\epsilon)\not=x$ for every edge $\epsilon\in E$.
Hence, because the polynomial ring is an integral domain, 
we have 
\begin{align*}
g(i(\epsilon))-g(t(\epsilon))\equiv 0\ ({\rm mod}\ \alpha(\epsilon))
\end{align*}
for every edge $\epsilon$. 
This implies that  
$g\in H^{*}(\mathcal{G})$.
Therefore for all $f\in {\rm Ker}\ \widetilde{F}$, there exists an element $g\in H^{*}(\mathcal{G})$ such that $f=g\chi$.
Hence, ${\rm Ker}\ \widetilde{F}\subset \langle \chi \rangle$.
On the other hand, we can easily check that ${\rm Ker}\ \widetilde{F}\supset \langle \chi \rangle$.
This establishes that ${\rm Ker}\ \widetilde{F}=\langle \chi \rangle$.
\end{proof}

The following second lemma is about the degree-wise decomposition of an element in $H^{*}(\mathcal{G})$. 

\begin{lemma}
\label{deg_decomp_graph_cohom}
For every $f\in H^{*}(\mathcal{G})$, 
there exists a non-negative integer $l$ and an element $f_{2i}\in H^{2i}(\mathcal{G})$ for each $0\le i\le l$ which satisfy 
\begin{align*}
f=f_{0}+f_{2}+\cdots +f_{2l},
\end{align*}
where $H^{2i}(\mathcal{G})$ consists of the element, say $h_{2i}$, which satisfies $h_{2i}(p)\in H^{2i}(BT^{n}\times BS^{1})$ for all $p\in \mathcal{V}$. 
\end{lemma}
\begin{proof}
Since $f(p)\in H^{*}(BT^{n}\times BS^{1})$, for every $p\in \mathcal{V}$ 
there exists a non-negative integer $l(p)$ and an element $f_{2i}(p)\in H^{2i}(BT^{n}\times BS^{1})$ such that 
\begin{align*}
f(p)=f_{0}(p)+\cdots +f_{2l(p)}(p).
\end{align*} 
If we take the maximal integer $l=\max \{l(p)\ |\ p\in \mathcal{V}\}$, then we may write 
\begin{align*}
f(p)=f_{0}(p)+\cdots +f_{2l}(p).
\end{align*} 
for all $p\in \mathcal{V}$.
Therefore, we can define the map 
$f_{2i}:\mathcal{V}\to H^{2i}(BT^{n}\times BS^{1})$ by $p\mapsto f_{2i}(p)$ for all $0\le i \le l$.
We claim that $f_{2i}\in H^{*}(\mathcal{G})$.
Because $f$ satisfies the congruence relation for all edges $\epsilon$, we see that
\begin{align}
\label{cong-rel-lem6.2}
f(i(\epsilon))-f(t(\epsilon))=(f_{0}(i(\epsilon))-f_{0}(t(\epsilon)))+\cdots +(f_{2l}(i(\epsilon))-f_{2l}(t(\epsilon)))=A\alpha(\epsilon)
\end{align}
for some $A\in H^{*}(BT^{n}\times BS^{1})$.
Moreover,  
there is a monomial $A_{2i}\in H^{2i}(BT^{n}\times BS^{1})$ for each $0\le i\le l-1$ such that 
\begin{align*}
A=A_{0}+\cdots +A_{2l-2}.
\end{align*}
Comparing the same degree monomials of both sides in \eqref{cong-rel-lem6.2}, we have 
\begin{align*}
f_{2i}(i(\epsilon))-f_{2i}(t(\epsilon))=A_{2i-2}\alpha(\epsilon)\equiv 0 \mod \alpha(\epsilon).
\end{align*}
Because this relation satisfies for all $\epsilon\in E$, 
we have $f_{2i}\in H^{*}(\mathcal{G})$ for all $i=0,\ \cdots,\ l$.
This establishes the statement. 
\end{proof}

We call each $f_{2i}$ in Lemma~\ref{deg_decomp_graph_cohom} a \textit{$2i$ degree homogeneous term} of $f$ for $i=0,\ \cdots,\ l$.
We denote ${\rm deg}\ f_{2i}=2i$.
Of course, $f_{2i}\in H^{2i}(\mathcal{G})$.

The following third lemma is about the map $\pi:\mathbb{Z}[X,H_{1},\ldots, H_{m}]\to H^{*}(\mathcal{G})$. This will be a technical part to show that $\pi$ is surjective (Lemma~\ref{surj_pi}).
\begin{lemma}
\label{technical_for_pi}
Assume that there exists an element $f\in H^{*}(\mathcal{G})$ such that $f\not\in {\rm Im}\ \pi$. 
Then there are $A\in \Z[X,H_{1},\cdots,H_{m}]$ and 
some integer $j_{k}$ such that 
\begin{eqnarray*}
\pi(A)-f=\chi\sum_{k}g_{2j_{k}},
\end{eqnarray*}
where 
$g_{2j_{k}}\in H^{2j_{k}}(\mathcal{G})$ but $g_{2j_{k}}\not\in{\rm Im}\ \pi$ with $j_{0}<j_{1}<\cdots <j_{k}<\cdots $.
\end{lemma}
\begin{proof}
Assume $f\not\in {\rm Im}\ \pi$.
Recall that the following two homomorphisms in \eqref{diagram} are surjective by the assumption (1) of Theorem~\ref{main-theorem1} and Theorem~\ref{main-theorem1-1}:
\begin{align*}
& \phi:\mathbb{Z}[X,H_{1},\ldots, H_{m}]\longrightarrow \mathbb{Z}[L_{1},\ldots, L_{m}]; \\
& \pi':\mathbb{Z}[L_{1},\ldots, L_{m}]\longrightarrow H^{*}(\widetilde{\mathcal{G}}). 
\end{align*}
Therefore, there exists a non-zero polynomial  
\begin{align*}
B\in \Z[X,H_{1},\cdots,H_{m}] 
\end{align*}
such that for $\widetilde{F}:H^{*}(\mathcal{G})\to H^{*}(\widetilde{\mathcal{G}})$, 
\begin{align*}
\widetilde{F}(f)=\pi'\circ\phi(B).
\end{align*}
Because $\pi' \circ\phi=\widetilde{F} \circ\pi$ in the diagram \eqref{diagram}, we have 
\begin{eqnarray*}
\pi' \circ\phi(B)=\widetilde{F} \circ\pi(B)=\widetilde{F}(f).
\end{eqnarray*}
Hence $\pi(B)-f\in {\rm Ker}\ \widetilde{F}$.
Because of Lemma~\ref{kernel_of_Ftilde}, i.e., ${\rm Ker}\widetilde{F}=\langle \chi\rangle$, there is a $g'\in H^{*}(\mathcal{G})$ such that 
\begin{align}
\label{eq-Lem6.3}
\pi(B)-f=g'\chi.
\end{align}
Since $f\not\in {\rm Im}\ \pi$ and $\pi(X)=\chi$,
we have 
\begin{eqnarray*}
g'\not\in {\rm Im}\ \pi.
\end{eqnarray*}
Because of Lemma~\ref{deg_decomp_graph_cohom}, this element $g'$ can be divided into 
\begin{align*}
g'=g_{0}+\cdots +g_{2l}, 
\end{align*}
where $g_{2i}$ is a $2i$ degree homogeneous term, for $0\le i\le l$.
If $g_{2i}\in {\rm Im}\ \pi$, then $g'-g_{2i}\not\in {\rm Im}\ \pi$.
Therefore, $g'$ can be divided into two terms $(0\not=)g=\sum_{k}g_{2j_{k}}$ for all $g_{2j_{k}}\not\in{\rm Im}\ \pi$ and $h=\sum_{k'}g_{2i_{k'}}$ for all $g_{2i_{k'}}\in {\rm Im}\ \pi$ such that
\begin{eqnarray*}
g'=g+h.
\end{eqnarray*}
Since 
\begin{eqnarray*}
g'\chi=g\chi+h\chi=g\chi+\pi(CX)
\end{eqnarray*}
 for some $C\in \Z[X,H_{1},\cdots,H_{m}]$, together with \eqref{eq-Lem6.3}, 
we see that there is an element $A=B-CX\in \Z[X,H_{1},\cdots,H_{m}]$ such that
$\pi(A)-f=g\chi$.
\end{proof}

Now we may prove Lemma~\ref{surj_pi}.
\begin{lemma}
\label{surj_pi}
The homomorphism $\pi:\Z[X,H_{1},\cdots,H_{m}] \to H^{*}(\mathcal{G})$ is surjective. 
\end{lemma}
\begin{proof}
By Lemma~\ref{technical_for_pi}, it is enough to show that 
every homogeneous term of $f\in H^{*}(\mathcal{G})$ is an element of ${\rm Im}\ \pi$. 

Assume that $H^{*}(\mathcal{G})\backslash {\rm Im}\ \pi\not=\emptyset$. 
Let $f$ be a minimal degree homogeneous element in $H^{*}(\mathcal{G})\backslash {\rm Im}\ \pi$.
Because of Lemma~\ref{technical_for_pi}, 
there exists a polynomial $A\in \Z[X,H_{1},\cdots,H_{m}]$ and an element $g\in H^{*}(\mathcal{G})\backslash {\rm Im}\ \pi$ such that 
\begin{align*}
f=\pi(A)-g\chi.
\end{align*}
By using Lemma~\ref{technical_for_pi} again, we also have that $g$  is a sum of homogeneous elements in $H^{*}(\mathcal{G})\backslash {\rm Im}\ \pi$. 

We claim that $\pi(A)(\in {\rm Im}\ \pi)$ and $g \chi(\in H^{*}(\mathcal{G})\backslash {\rm Im}\ \pi)$ are also homogeneous elements in $H^{*}(\mathcal{G})$ whose degrees are the same with the degree of $f$. Assume that $\pi(A)=\sum_{k}h_{2i_{k}}$ and $g\chi=\sum_{k}g_{2j_{k}}\chi$, where $h_{2i_{k}}\in H^{2i_{k}}(\mathcal{G})\cap {\rm Im}\ \pi$ for all $i_{0}<i_{1}<\cdots$ and $g_{2j_{k}}\in H^{2j_{k}}(\mathcal{G})\setminus {\rm Im}\ \pi$ for all $j_{0}<j_{1}<\cdots$. Because $f$ is a minimal homogeneous element in $H^{*}(\mathcal{G})\backslash {\rm Im}\ \pi$, we see that $\deg g_{2j_{0}}\chi=\deg f$ and $\deg h_{2i_{0}}\ge \deg f$; moreover, the higher terms of $\pi(A)\in {\rm Im}\ \pi$ and $g\chi\not\in{\rm Im}\ \pi$ are the same, i.e., they must be  $0$. Hence, both of $\pi(A)$ and $g$ are also homogeneous elements.

However, in this case, we have  
\begin{align*}
\deg g=\deg g \chi-\deg \chi=\deg f -2<\deg f.
\end{align*}
This gives a contradiction to that $f$ is a minimal homogeneous element in $H^{*}(\mathcal{G})\backslash {\rm Im}\ \pi$. 
Hence, there does not exist any homogeneous elements in 
$H^{*}(\mathcal{G})\backslash {\rm Im}\ \pi$.
Consequently, by Lemma~\ref{technical_for_pi}, we have that 
$H^{*}(\mathcal{G})\backslash {\rm Im}\ \pi=\emptyset$, i.e.,
$\pi$ is surjective. 
\end{proof}

Therefore, by the commutativity of the top diagram in \eqref{diagram}, the following lemma holds:
\begin{lemma}
\label{surj2}
$\Psi$ is surjective.
\end{lemma}


\subsection{Injectivity of $\Psi$}
\label{sect:6.2}

Finally, in this section, we will prove the injectivity of $\Psi$.
In this section we use the following notation: 
\begin{eqnarray*}
I_{j}=\{1,\ \cdots,\ l\}-\{j\}
\end{eqnarray*} 
for $j=1,\ \cdots,\ l$.
We first prove Lemma \ref{key2}.
In order to prove Lemma \ref{key2}, we prepare the following lemma.

\begin{lemma}
\label{lem_for_key2}
Assume that $\cap_{k=1}^{l}L_{k}=\emptyset$ and $L_{k}=H_{k}\cap\overline{H_{k}}$ $(k=1,\ \cdots,\ l)$. 
Then for all $j=1, \ldots, l$, one of the following holds: 
\begin{itemize}
\item $H_{j}\cap (\cap_{k\in I_{j}}L_{k})=\emptyset$;
\item $\overline{H_{j}}\cap (\cap_{k\in I_{j}}L_{k})=\emptyset$.
\end{itemize}
\end{lemma}
\begin{proof}
Assume $\cap_{k=1}^{l}L_{k}=\emptyset$.
For $j\in \{1,\ \cdots,\ l\}$, 
if the following relation holds: 
\begin{align*}
L_{j}\cap (\cap_{k\in I_{j}}L_{k})=\cap_{k\in I_{j}}L_{k}=\emptyset,
\end{align*}
then it follows from $\cap_{k\in I_{j}}L_{k}=\emptyset$ that for each  
$H_{j}$ and $\overline{H_{j}}$ we have 
\begin{align*}
H_{j}\cap (\cap_{k\in I_{j}}L_{k})=\overline{H_{j}}\cap (\cap_{k\in I_{j}}L_{k})=\emptyset.
\end{align*}
So we may take $j\in \{1,\ \cdots,\ l\}$ such that 
\begin{eqnarray*}
\bigcap_{k\in I_{j}}L_{k}\not=\emptyset.
\end{eqnarray*} 
In this case, 
there exists a vertex $p\in \mathcal{V}^{\cap_{k\in I_{j}}L_{k}}$. 
Since $L_{j}\cap(\cap_{k\in I_{j}}L_{k})=\cap_{k=1}^{l}L_{k}=\emptyset$,
we have that $p\not\in \mathcal{V}^{L_{i}}$.
Therefore, for all vertices $p\in \mathcal{V}^{\cap_{k\in I_{j}}L_{k}}$, the following equation holds: 
\begin{eqnarray*}
\tau_{H_{j}}(p)=
\left\{
\begin{array}{cc}
0 & ({\rm if}\  p\not\in \mathcal{V}^{H_{j}}) \\
x & ({\rm if}\  p\in \mathcal{V}^{H_{j}})
\end{array}
\right.
\end{eqnarray*} 
where $\tau_{H_{j}}$ is the Thom class of $H_{j}$.

If there are two vertices $p,q\in \mathcal{V}^{\cap_{k\in I_{j}}L_{k}}$ such that 
\begin{align*}
\tau_{H_{j}}(p)=0;\quad \tau_{H_{j}}(q)=x.
\end{align*}
By the assumption (2) in Theorem~\ref{main-theorem1},
there exists a path from $p$ to $q$ in $\cap_{k\in I_{j}}L_{k}$, i.e., 
we may take the following sequence in $E^{\cap_{k\in I_{j}}L_{k}}$:
\begin{eqnarray*}
\epsilon_{1},\cdots, \epsilon_{s}\in E^{\cap_{k\in I_{j}}L_{k}}
\end{eqnarray*}
such that $i(\epsilon_{1})=p$ and $t(\epsilon_{s})=q$.
By the definition of the $T^{*}\mathbb{C}^{n}$-modeled GKM graph, the axial function satisfies that $\alpha(\epsilon)\not=x$ for all $\epsilon\in E$.
Moreover, $\tau_{H_{j}}$ satisfies the congruence relation.
Therefore, there exists an edge $\epsilon\in \{\epsilon_{1},\ldots, \epsilon_{s}\}$ such that  $r=i(\epsilon_{t})$ satisfies that 
\begin{align*}
\tau_{H_{j}}(r)\not=0,\ x.
\end{align*}
By the definition of the Thom class of the halfspace $H_{j}$,
the vertex $r\in \partial H_{j}=L_{j}$.
However, this gives that $r\in L_{j}\cap (\cap_{k\in I_{j}}L_{k})=\cap_{k=1}^{l}L_{k}$.
This gives a contradiction to that $\cap_{k=1}^{l}L_{k}=\emptyset$.

Therefore, we may assume $\tau_{H_{j}}(p)=0$ (resp.~$x$)
for all $p\in\mathcal{V}^{\cap_{k\in I_{j}}L_{k}}$.
Then, by definition of the halfspace, we have $H_{j}\cap \mathcal{V}^{\cap_{k\in I_{j}}L_{k}}=\emptyset$ (resp.~$\overline{H_{j}}\cap \mathcal{V}^{\cap_{k\in I_{j}}L_{k}}=\emptyset$).
This establishes the statement of this lemma.
\end{proof}

From Lemma \ref{lem_for_key2}, we have the following key fact.
\begin{lemma}
\label{key2}
Assume the $T^{*}\mathbb{C}^{n}$-modeled GKM graph $\mathcal{G}$ satisfies two assumptions (1), (2) of Theorem~\ref{main-theorem1}.
If $\cap_{k=1}^{l}L_{k}=\emptyset$ and $L_{k}=H_{k}\cap\overline{H_{k}}$ $(k=1,\ \cdots,\ l)$, then we can take a halfspace $H_{k}$
such that $\cap_{k=1}^{l}H_{k}=\emptyset$. 
\end{lemma}
\begin{proof}
If $\cap_{k=1}^{l}L_{k}=\emptyset$ and $L_{k}=H_{k}\cap\overline{H_{k}}$ $(k=1,\ \cdots,\ l)$,
 we can take $H_{j}$ as 
$H_{j}\cap (\cap_{k\in I_{j}}L_{k})=\emptyset$
for all $j=1,\ \cdots,\ l$ from Lemma \ref{lem_for_key2}.
Now we may 
set 
\begin{eqnarray*}
\mathbf{H}'=\left\{H_{1},\ \cdots,\ H_{l}\ |\ H_{j}\cap (\cap_{k\in I_{j}}L_{k})=\emptyset,\ j=1, \ldots, l\right\}.
\end{eqnarray*}
We claim that $\cap_{H\in\mathbf{H}'}H=\cap_{j=1}^{l}H_{j}=\emptyset$.
If there exists a vertex $p\in \mathcal{V}^{\cap_{j=1}^{l}H_{j}}$, it follows from the assumption $\cap_{k=1}^{l}L_{k}=\emptyset$ that we have $\tau_{H_{j}}(p)=x$ for all $j=1, \ldots, l$; therefore, 
\begin{align*}
\prod_{j=1}^{l}\tau_{H_{j}}(p)=x^{l}.
\end{align*}
Because $\prod_{j=1}^{l}\tau_{H_{j}}\in H^{*}(\mathcal{G})$, 
$\prod_{j=1}^{l}\tau_{H_{j}}$ satisfies the congruence relations for all edges $\epsilon\in E$.
By definition of $T^{*}\mathbb{C}^{n}$-modeled GKM graph,
the axial function satisfies $\alpha(\epsilon)\not=x$ for all edges $\epsilon\in E$.
This shows that for all edge $\epsilon\in E_{p}$ the following equation holds:
\begin{align*}
\prod_{j=1}^{l}\tau_{H_{j}}(t(\epsilon))=x^{l}.
\end{align*}
Because the graph $\Gamma$ is connected, we can apply the same argument for all vertices; therefore, we have 
\begin{align*}
\prod_{j=1}^{l}\tau_{H_{j}}(q)=x^{l}
\end{align*}
for all $q\in \mathcal{V}$.
This shows that $\mathcal{V}=\mathcal{V}^{\cap_{j=1}^{l}H_{j}}$.
However, by definition of the halfspace, it is obvious that $\mathcal{V}\not=\mathcal{V}^{\cap_{j=1}^{l} H_{j}}$ and this gives a contradiction.
Hence, we have $\cap_{j=1}^{l} H_{j}=\cap_{H\in\mathbf{H}'}H=\emptyset$.
\end{proof}

We next will prove Lemma \ref{key_for_inj}.
In order to prove it, we prepare some notations and two lemmas: Lemma \ref{1st_for_key_for_inj} and \ref{2nd_for_key_for_inj}.

Let $\widetilde{\pi}:\Z[X,\ H_{1},\ \cdots,\ H_{m}]\to \Z[\mathcal{G}]$ be the natural homomorphism such that $\widetilde{\pi}(X)=X$, 
$\widetilde{\pi}(H_{i})=H_{i}$ for $i=1,\ \cdots,\ m$.
Because $\overline{H_{i}}=X-H_{i}$ in $\Z [\mathcal{G}]$, we have 
\begin{eqnarray*}
\widetilde{\pi}\circ\phi'=\hat{\pi}:\Z[X,\ H_{1},\ \cdots,H_{m},\overline{H}_{1},\cdots, \overline{H_{m}}]\to \Z[\mathcal{G}].
\end{eqnarray*}
Since $\hat{\pi}$ is surjective, $\widetilde{\pi}$ is also surjective.
Moreover we have 
\begin{eqnarray*}
\Psi\circ\widetilde{\pi}=\pi:\Z[X,\ H_{1},\ \cdots,\ H_{m}]\to H^{*}(\mathcal{G})
\end{eqnarray*} 
by definitions of $\Psi$ and $\pi$.
Hence we have the following commutative diagram:
\begin{eqnarray*}
\xymatrix{
\Z[X,H_{1},\cdots,H_{m},\overline{H}_{1},\cdots,\overline{H_{m}}]\ar[r]^(.7){\hat{\pi}}\ar[d]^{\phi'} & \Z[\mathcal{G}]\ar[d]^{\Psi} \\
\Z[X,H_{1},\cdots,H_{m}]\ar[r]^(.7){\pi}\ar[d]^{\phi}\ar[ru]^{\widetilde{\pi}} & H^{*}(\mathcal{G})\ar[d]^{\widetilde{F}} \\
\Z[L_{1},\cdots,L_{m}]\ar[r]^(.7){\pi'} & H^{*}(\mathcal{\widetilde{G}})
}
\end{eqnarray*}
Define the following ideal in $\Z[X,H_{1},\cdots,H_{m},\overline{H}_{1},\cdots,\overline{H_{m}}]$:
\begin{align*}
\mathcal{I}=\Big\langle H_{i}+\overline{H_{i}}-X,\ \prod_{H\in \mathbf{H}'} H\ \Big|\ i=1,\ \cdots,\ m,\ \mathbf{H}'\in \mathbf{I}(\mathbf{H}) \Big\rangle,
\end{align*}
where $\mathbf{I}(\mathbf{H})=\{\mathbf{H}'\subset \mathbf{H}\ |\ \cap_{H\in \mathbf{H}'}H=\emptyset\}$.
For this ideal, the following property holds.
\begin{lemma}
\label{1st_for_key_for_inj}
For the ideal $\mathcal{I}\subset \Z[X,H_{1},\cdots,H_{m},\overline{H}_{1},\cdots,\overline{H_{m}}]$, 
the following two properties hold:
\begin{enumerate}
\item[(i)] ${\rm Ker}\ \widetilde{\pi}=\phi'(\mathcal{I})$;
\item[(ii)] ${\rm Ker}\ \pi'=\phi\circ\phi'(\mathcal{I})$.
\end{enumerate}
\end{lemma}
\begin{proof}
Since $\hat{\pi}$ is the natural projection, it follows from the definition of $\mathbb{Z}[\mathcal{G}]$ that 
\begin{align*}
\mathcal{I}={\rm Ker}\ \hat{\pi}.
\end{align*}
So, by the commutativity of the diagram, we have that 
\begin{align*}
\widetilde{\pi}(\phi'(\mathcal{I}))=\hat{\pi}(\mathcal{I})=\hat{\pi}({\rm Ker}\ \hat{\pi})=\{0\}.
\end{align*}
Hence $\phi'(\mathcal{I})\subset {\rm Ker}\ \widetilde{\pi}$. 
Let $A$ be an element in ${\rm Ker}\ \widetilde{\pi}$.
Because $\phi'$ is surjective, there is an element $B\in \Z[X,\ H_{1},\ \cdots,H_{m},\overline{H}_{1},\cdots, \overline{H_{m}}]$ such that $\phi'(B)=A$.
By the commutativity of the diagram, we also have 
\begin{align*}
\hat{\pi}(B)=\widetilde{\pi}\circ\phi'(B)=\widetilde{\pi}(A)=0.
\end{align*}
So $B\in {\rm Ker}\ \hat{\pi}=\mathcal{I}$.
Hence $A=\phi'(B)\in \phi'(\mathcal{I})$, that is, ${\rm Ker}\ \widetilde{\pi}\subset \phi'(\mathcal{I})$. 
Therefore, we establish the first property: ${\rm Ker}\ \widetilde{\pi}=\phi'(\mathcal{I})$.

By Theorem~\ref{main-theorem1-1}, we know 
\begin{align*}
{\rm Ker}\ \pi'=\langle\prod_{L\in \mathbf{L}'}L\ |\ \mathbf{L}'\in \mathbf{I}(\mathbf{L})\rangle,
\end{align*}
where $\mathbf{I}(\mathbf{L})=\{\mathbf{L}'\subset \mathbf{L}\ |\ \cap_{L\in \mathbf{L}'}L=\emptyset\}$.
Take a generator $\prod_{L\in \mathbf{L}'}L\in {\rm Ker}\ \pi'$.
From Lemma~\ref{key2}, for $\mathbf{L}'=\{L_{1},\ \cdots,\ L_{l}\}\in \mathbf{I}(\mathbf{L})$, 
there exists a set of halfspaces $\mathbf{H}'=\{H_{1},\ \cdots,\ H_{l}\}\in \mathbf{I}(\mathbf{H})$ such that $H_{k}\cap \overline{H_{k}}=L_{k}$.
By the definition of the ideal $\mathcal{I}$, a product $\prod_{k=1}^{l}H_{k}$ is one of the generators of $\mathcal{I}$.
Moreover, by the definitions of $\phi'$ and $\phi$, we see that 
\begin{align*}
\phi\circ\phi'(\mathcal{I})\ni \phi\circ\phi'(\prod_{k=1}^{l}H_{k})=\pm\prod_{k=1}^{l}L_{k}.
\end{align*}
Because this satisfies for all generators $\prod_{L\in \mathbf{L}'}L$ in ${\rm Ker}\ \pi'$, 
we have that 
\begin{align*}
{\rm Ker}\ \pi'\subset\phi\circ\phi'(\mathcal{I}).
\end{align*}
On the other hand,  
because $\phi'(H+\overline{H}-X)=0$ and $\phi\circ\phi'(\prod_{H\in \mathbf{H}'}H)=\pm\prod_{L\in \mathbf{L}'}L\in {\rm Ker}\ \pi'$, 
for all $A\in \mathcal{I}$
we have 
\begin{align*}
\pi'\circ\phi\circ\phi'(A)=\{0\}.
\end{align*} 
So we have ${\rm Ker}\ \pi'\supset\phi\circ\phi'(\mathcal{I})$.
Therefore we conclude the second property: ${\rm Ker}\ \pi'=\phi\circ\phi'(\mathcal{I})$.
\end{proof}

In order to prove Lemma~\ref{key_for_inj}, we also prepare the following technical lemma for general polynomial rings.
\begin{lemma}
\label{2nd_for_key_for_inj}
Let $\mathcal{I}\subset \Z[x_{1},\ \cdots,\ x_{l}]$ be an ideal generated by homogeneous polynomials, that is,
$\mathcal{I}=\langle p_{1},\ \cdots,\ p_{m}\rangle$
where $p_{i}$ is a homogeneous polynomial of $\Z[x_{1},\ \cdots,\ x_{l}]$ such that ${\rm deg}\ p_{i}\le {\rm deg}\ p_{j}$ for $i< j$.
For every element $A\in \mathcal{I}$, if we denote 
$A=A_{1}+\cdots +A_{n}$,
where $A_{i}$ is a homogeneous term ($i=1,\ \cdots,\ n$) and ${\rm deg}\ A_{i}<{\rm deg}\ A_{j}$ for $i<j$, 
then $A_{i}\in \mathcal{I}$ for all $i=1,\ \cdots,\ n$.
\end{lemma}
\begin{proof}
Because $A\in \mathcal{I}=\langle p_{1},\ \cdots,\ p_{m}\rangle$, 
there exists $X_{k}\in \Z[x_{1},\ \cdots,\ x_{l}]$, $k=1,\ldots, m$, such that 
\begin{align*}
A=X_{1}p_{1}+\cdots +X_{m}p_{m}.
\end{align*}
Then we can put $X_{k}=X_{k1}+X_{k2}+\cdots +X_{ks_{k}}$ where $X_{ki}$ is a homogeneous term ($i=1,\ \cdots,\ s_{k}$) 
and ${\rm deg}\ X_{ki}<{\rm deg}\ X_{kj}$ for $i<j$.
Hence, by changing the order of monomials, we may rewrite
\begin{align*}
A&=(X_{11}+\cdots +X_{1s_{1}})p_{1}+\cdots +(X_{m1}+\cdots +X_{ms_{m}})p_{m} \\
&=X_{11}p_{1}+\cdots X_{ms_{m}}p_{m}
\end{align*}
as
\begin{align*}
A=A_{1}+\cdots +A_{n},
\end{align*}
where $\deg A_{i}<\deg A_{j}$ if $i<j$.
Because $A_{i}$ is a homogeneous term, we have 
\begin{align*}
A_{i}=\sum_{j\in \mathbf{D}_{i}}X_{jh_{j}}p_{j} 
\end{align*}
where 
$\mathbf{D}_{i}=\{j\ |\ \deg X_{jh_{j}}+\deg p_{j}=\deg A_{i}\}$.
Therefore, $A_{i}\in \mathcal{I}$ for all $i=1,\ \cdots,\ n$.
\end{proof}

Using two lemmas as above, we have the following lemma.
\begin{lemma}
\label{key_for_inj}
${\rm Ker}\ \widetilde{\pi}={\rm Ker}\ \pi$.
\end{lemma}
\begin{proof}
By Lemma~\ref{1st_for_key_for_inj} (i), 
${\rm Ker}\ \widetilde{\pi}=\phi'(\mathcal{I})$.
Therefore, by using 
the commutativity of the diagram and ${\rm Ker}\ \hat{\pi}=\mathcal{I}$, we have 
\begin{align*}
\pi({\rm Ker}\ \widetilde{\pi})=\pi\circ\phi'(\mathcal{I})=\Psi\circ\hat{\pi}(\mathcal{I})=0.
\end{align*}
Hence, 
\begin{align*}
{\rm Ker}\ \widetilde{\pi}=\phi'(\mathcal{I})\subset {\rm Ker}\ \pi.
\end{align*}

We claim that ${\rm Ker}\ \pi\subset {\rm Ker}\ \widetilde{\pi}(\subset \Z[X,\ H_{1},\ \cdots,\ H_{m}])$.
Assume that ${\rm Ker}\ \pi\backslash\phi'(\mathcal{I})\not=\emptyset$.
Let $A\in {\rm Ker}\ \pi\backslash\phi'(\mathcal{I})\subset \Z[X,\ H_{1},\ \cdots,\ H_{m}]$ be a minimal degree homogeneous polynomial.
By the previous diagram, 
\begin{align*}
\pi'\circ\phi(A)=\widetilde{F}\circ\pi(A)=0.
\end{align*}
Hence, by Lemma~\ref{1st_for_key_for_inj} (ii), 
\begin{align*}
\phi(A)\in {\rm Ker}\ \pi'=\phi\circ\phi'(\mathcal{I}).
\end{align*}
Therefore, 
we can take $B\in \phi'(\mathcal{I})(\subset {\rm Ker}\ \pi)$ such that $\phi(A)=\phi(B)$.
Because $\phi'(\mathcal{I})$ is an ideal in $\mathbb{Z}[X,H_{1},\ldots, H_{m}]$ and $A$ is a homogeneous polynomial, 
it follows from Lemma~\ref{2nd_for_key_for_inj} that 
we may also take $B$ as the homogeneous polynomial in $\phi'(\mathcal{I})$ such that 
\begin{align*}
\deg A=\deg B.
\end{align*}
By the definition of $\phi$, 
it is easy to check that ${\rm Ker}\ \phi=\langle X\rangle \subset \mathbb{Z}[X,H_{1},\ldots, H_{m}]$; therefore, 
we have
\begin{eqnarray*}
A-B\in {\rm Ker}\ \phi=\langle X\rangle.
\end{eqnarray*}
This means that there exists a polynomial $C\in \Z[X,\ H_{1},\ \cdots,\ H_{m}]$ such that 
\begin{eqnarray*}
A-B=CX.
\end{eqnarray*}
Because $A=B+CX\not\in \phi'(\mathcal{I})$ and $B\in \phi'(\mathcal{I})$, we have $CX\not\in \phi'(\mathcal{I})$.
Moreover, because we take $\deg A=\deg B$, 
$CX$ is also a homogeneous polynomial with $\deg A=\deg B=\deg CX$. 
Therefore, because $A,\ B\in {\rm Ker}\ \pi$, we have $CX$ is a homogeneous polynomial in ${\rm Ker}\ \pi\setminus \phi'(\mathcal{I})$.
Then, we have $\deg A=\deg CX=\deg C+\deg X=\deg C+2$.
Moreover, because $A$ is a minimal homogeneous polynomial in ${\rm Ker}\ \pi\setminus \phi'(\mathcal{I})$, 
we have that 
\begin{align*}
C\in \phi'(\mathcal{I})(\subset {\rm Ker}\ \pi).
\end{align*}
However, because $\phi'(\mathcal{I})$ is an ideal in $\mathbb{Z}[X,H_{1},\ldots, H_{m}]$, we see that 
\begin{align*}
CX\in \phi'(\mathcal{I}).
\end{align*}
This shows that $A=B+CX\in \phi'(\mathcal{I})$ and 
this gives the contradiction to that there is an element in ${\rm Ker}\ \pi\backslash\phi'(\mathcal{I})$.
Hence, we have ${\rm Ker}\ \pi\backslash \phi'(\mathcal{I})=\emptyset$, that is, ${\rm Ker}\ \pi=\phi'(\mathcal{I})={\rm Ker}\ \widetilde{\pi}$
by Lemma~\ref{1st_for_key_for_inj} (i).
\end{proof}

So we can prove the injectivity of $\Psi$.
\begin{lemma}
\label{inj2}
$\Psi$ is injective.
\end{lemma}
\begin{proof}
Let $A$ be in ${\rm Ker}\ \Psi$.
Since $\widetilde{\pi}$ is surjective, 
there is an element $B\in \Z[X,\ H_{1},\ \cdots,\ H_{m}]$ 
such that $\widetilde{\pi}(B)=A$.
So we have $\pi(B)=\Psi\circ\widetilde{\pi}(B)=\Psi(A)=0$.
Hence $B\in {\rm Ker}\ \pi={\rm Ker}\ \widetilde{\pi}$ by Lemma~\ref{key_for_inj}.
Therefore, we have $A=\widetilde{\pi}(B)=0$.
This concludes that $\Psi$ is injective.
\end{proof}

Because of Lemma~\ref{surj2} and \ref{inj2}, we have that $\Psi$ is the isomorphism.
Consequently the proof of Theorem~\ref{main-theorem1} is complete, that is, we get
\begin{align*} 
H^{*}(\mathcal{G})\simeq \Z[\mathcal{G}].
\end{align*}

\section{Generators of $\mathbb{Z}[\widetilde{\mathcal{G}}]$ as $H^{*}(BT^{n})$-module}
\label{sect:7}

Let $\mathcal{G}=(\Gamma, \alpha, \nabla)$ be a $2n$-valent
$T^{*}\mathbb{C}^n$-modeled GKM graph and
$\mathbf{L}=\{L_1,\ldots, L_m\}$ be the set of all hyperplanes in
$\mathcal{G}$. 
Assume that $\mathcal{G}$ satisfies the two assumptions of
Theorem \ref{main-theorem1} so that
$H^*(\mathcal{G})\simeq \mathbb{Z}[\mathcal{G}]$.

\subsection{Simplicial complex associated to $\mathbf{L}$}.  
\label{sect:7.1}

Let
$\mathbf{L}=\{L_1,\ldots, L_m\}$.  Let $\Delta_{\mathbf{L}}$ denote
the simplicial complex associated to $\mathbf{L}$ defined as
follows. There is a vertex $v_i$ in $\Delta_{\mathbf{L}}$
corresponding to the hyperplane $L_i$ such that whenever
$L_{i_1}\cap\cdots \cap L_{i_k}\neq \emptyset$ in $\mathcal{G}$, the
vertices $\{v_{i_1},\ldots, v_{i_k}\}$ span a simplex in
$\Delta_{\mathbf{L}}$. In particular, for $1\leq i\leq d$, let
$\sigma_i=\langle v_{i_1},\ldots, v_{i_n}\rangle$ be the
$(n-1)$-dimensional simplex of $\Delta_{\mathbf{L}}$ corresponding to a
vertex ${\bf p}_{\sigma_i} :=L_{i_1}\cap\cdots\cap L_{i_n}$ of
$\mathcal{G}$.

Note that $\Delta_{\mathbf{L}}$ is pure i.e., all maximal faces (also called facets) are of
the same dimension $n-1$. Let $\Delta_{\mathbf{L}}(n-1)$ denote the set of
facets of $\Delta_{\mathbf{L}}$. Then
$d=|\Delta_{\mathbf{L}}(n-1)|$ which is also equal to
$|\mathcal{V}^{\Gamma}|$ the number of vertices of $\Gamma$.
For simplices $\tau$ and $\sigma$ in $\Delta_{\mathbf{L}}$,
by $\tau\preceq \sigma$ we mean that $\tau$ is a face of $\sigma$.

We say that $\Delta_{\mathbf{L}}$ is a {\it shellable
  simplicial complex} if the following holds: There is an ordering
$\sigma_1,\sigma_2,\ldots, \sigma_d$ of $\Delta_{\mathbf{L}}(n-1)$ such
that if $\Delta_j$ denotes the subcomplex generated by
$\sigma_1,\ldots, \sigma_j$ for each $1\leq j\leq d$, then
$\Delta_{i}\setminus \Delta_{i-1}$ has a unique minimal face $\mu_i$
for each $2\leq i\leq d$. We further let $\mu_1:=\emptyset$ to be the
unique minimal face of $\Delta_{1}\setminus \Delta_{0}$ where
$\Delta_{0}=\emptyset$ (see \cite[Section 2.1 p.79]{stanley}).
(Also see \cite[Definition 5.1.11]{BH} for other
equivalent definitions of shellability of a pure simplicial complex.)

\begin{example}
In the case when $n=1$, $\mathcal{G}$ is a $2$-valent
$T^*\mathbb{C}$-modelled GKM graph and the set of hyperplanes
$\mathbf{L}=\{L_1,\ldots, L_m\}$ of $\mathcal{G}$ coincides with the
set of vertices of $\mathcal{G}$ (see Figure~\ref{4dim-toricHK}). Here $\mathcal{G}$ corresponds to
the hyperplane arrangement of a $4$-dimensional toric
hyperK{$\ddot{\mathrm{a}}$}hler manifold. Furthermore, since
$L_i\cap L_j=\emptyset$ for every $i\neq j$, $1\leq i,j\leq m$, the
associated simplicial complex $\Delta_{\mathbf{L}}$ is a pure
$0$-dimensional simplicial complex consisting of a vertex (or
$0$-dimensional simplex) $v_i$ corresponding to $L_i$ for every
$1\leq i\leq m$. Then $\Delta_{\mathbf{L}}$ is seen to be trivially
shellable for any ordering of the set of vertices $\{v_1,\ldots, v_m\}$ which
are also the facets of $\Delta_{\mathbf{L}}$ in this case.


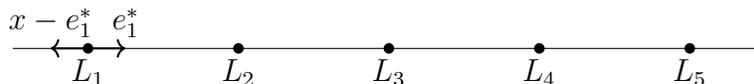
\begin{figure}[h]
\begin{tikzpicture}
\begin{scope}[xscale=1.0, yscale=1.0]
\draw (-5,0)--(5,0);

\fill (-4,0) coordinate (p) circle (2pt);
\node[below] at (-4,0) {$L_{1}$};
\fill (-2,0) coordinate (q) circle (2pt); 
\node[below] at (-2,0) {$L_{2}$};
\fill (0,0) coordinate (r) circle (2pt); 
\node[below] at (0,0) {$L_{3}$};
\fill (2,0) coordinate (s) circle (2pt);
\node[below] at (2,0) {$L_{4}$};
\fill (4,0) coordinate (t) circle (2pt);
\node[below] at (4,0) {$L_{5}$};

\draw[->, thick] (p)--(-3.5,0);
\node[above] at (-3.5,0) {$e_{1}^{*}$};
\draw[->, thick] (p)--(-4.5,0);
\node[above] at (-4.5,0) {$x-e_{1}^{*}$};
\end{scope}
\end{tikzpicture}
\caption{The $T^{*}\mathbb{C}$-modeled GKM graph $\mathcal{G}$ defined from a five vertices (hyperplanes) arrangement on the line. In this case, the hyperplane $L_{i}$ in $\mathcal{G}$ corresponds to the $0$-dimensional simplex $v_{i}$ in $\Delta_{\mathbf{L}}$ for $i=1,\ldots, 5$. This is defined from a $4$-dimensional toric hyperK{$\ddot{\mathrm{a}}$}hler manifold. Note that we omit the axial functions which are automatically determined.}
\label{4dim-toricHK}
\end{figure}
\end{example}

\begin{example}
Consider the simplicial complex $\Delta_{\mathbf{L}}$ corresponding to
the $T^*\mathbb{C}^2$-modelled GKM graph given in Figure
\ref{fig-non-toricHK}. We label the set of $5$ hyperplanes as
$L_1,\ldots, L_5$ where $L_i\cap L_{i+1}\neq \emptyset$ for
$1\leq i\leq 4$ and $L_5\cap L_1\neq \emptyset$ (see Figure~\ref{fig-non-toricHK2}). Then
$\Delta_{\mathbf{L}}$ consists respectively of the corresponding
vertices $v_1,\ldots, v_5$ and the $1$-simplices
$[v_1,v_2], [v_2,v_3], [v_3,v_4],[v_4,v_5],[ v_5,v_1]$ which are its
facets. For the ordering
\[\sigma_1=[v_1,v_2], \sigma_2=[v_2,v_3],
  \sigma_3=[v_3,v_4],\sigma_4=[v_4,v_5],\sigma_5=[ v_5,v_1]\] of
$\Delta_{\mathbf{L}}(1)$, we see that
$\Delta_1\setminus \Delta_0=\{\emptyset, \{v_1\}, \{v_2\},
[v_1,v_2]\}$ has the minimal element $\mu_1=\emptyset$,
$\Delta_2\setminus \Delta_1=\{\{v_3\}, [v_2,v_3]\}$ has the minimal
element $\mu_2=\{v_3\}$,
$\Delta_3\setminus \Delta_2=\{\{v_4\}, [v_3,v_4]\}$ has the minimal
element $\mu_3=\{v_4\}$,
$\Delta_4\setminus \Delta_3=\{\{v_5\}, [v_4,v_5]\}$ has the minimal
element $\mu_4=\{v_5\}$ and $\Delta_5\setminus \Delta_4=\{[v_5,v_1]\}$
has the minimal element $\mu_5=[v_5,v_1]$. Thus $\Delta_{\mathbf{L}}$
is a shellable simplicial complex.


\begin{figure}[h]
\begin{tikzpicture}
\begin{scope}[xscale=1.0, yscale=1.0]
\fill (-1,0) coordinate (q) circle (2pt);
\node[left] at (-1,0) {$p_{2}$};
\fill (1,-2) coordinate (t) circle (2pt); 
\node[below] at (1,-2) {$p_{5}$};
\fill (-1,-2) coordinate (p) circle (2pt); 
\node[below] at (-1.2,-2) {$p_{1}$};
\fill (3,0) coordinate (s) circle (2pt);
\node[right] at (3,0) {$p_{4}$};
\fill (3,4) coordinate (r) circle (2pt);
\node[right] at (3,4) {$p_{3}$};

\draw (-2,-1)--(4,5);
\draw (-1,2)--(-1,-3);
\draw (-2,-2)--(3,-2);
\draw (3,5)--(3,-1);
\draw (0,-3)--(5,2);

\node[above] at (0,-2) {$L_{1}$};
\node[right] at (-1,-1) {$L_{2}$};

\node[left] at (2,-1) {$L_{5}$};
\node[right] at (0.5,1.5) {$L_{3}$};

\node[left] at (3,1.5) {$L_{4}$}; 
\end{scope}
\end{tikzpicture}
\caption{The GKM graph $\mathcal{G}$ in Figure~\ref{fig-non-toricHK}. In this case, the hyperplane $L_{i}$ (resp.~ vertex $p_{i}$) in $\mathcal{G}$ corresponds to the $0$ (resp.~$1$)-dimensional simplex $v_{i}$ (resp.~$\sigma_{i}$) in $\Delta_{\mathbf{L}}$ for $i=1,\ldots, 5$.}
\label{fig-non-toricHK2}
\end{figure}
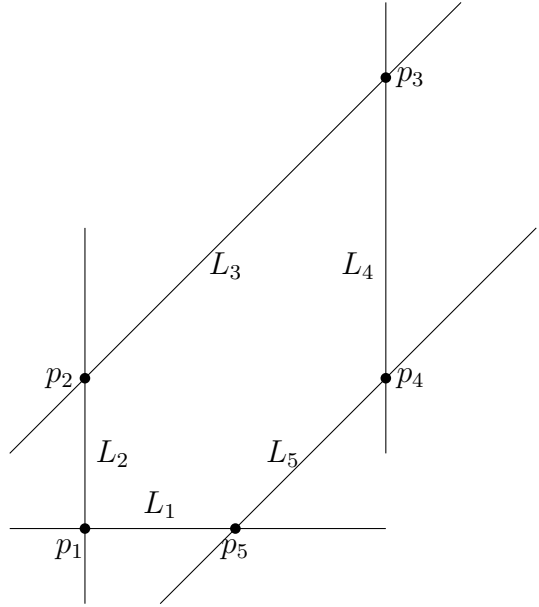

In subsection \ref{sect:7.5}, we shall show the shellability of
$\Delta_{\mathbf{L}}$ for a $T^*\mathbb{C}^2$-modelled GKM graph
$\mathcal{G}$ that is induced from the $8$-dimensional toric
hyperK{$\ddot{\mathrm{a}}$}hler manifold.
\end{example}

For $\gamma\in \Delta_{\mathbf{L}}$, let $j$ be the smallest $1\le j \le d$
such that $\gamma\preceq\sigma_j$. Then
$\gamma\in \Delta_j\setminus \Delta_{j-1}$. Thus it follows that
$\mu_j\preceq \gamma\preceq \sigma_j$. Hence there exists a
  unique $j$, $1\leq j\leq d$ such that $\gamma\in
  [\mu_j,\sigma_j]$ where
  $[\mu_j,\sigma_j]:=\{\gamma~~\mid~~\mu_j\preceq \gamma\preceq \sigma_j\}$. In other words we can write
  \begin{equation}\label{partition}\Delta_{\mathbf{L}}=[\mu_1,\sigma_1]\sqcup\cdots\sqcup
  [\mu_d,\sigma_d].
\end{equation}

If a simplicial complex $\Delta_{\mathbf{L}}$ satisfies
(\ref{partition}) then it is called partitionable (see \cite[p.80,
Section 2.1]{stanley}). In particular, shellable simplicial complexes
are partitionable.

Moreover, \begin{equation}\label{star}\mu_i\preceq \sigma_j
  \Rightarrow j\geq i. 
  \end{equation}

Let
\[\mathcal{E}_{{\bf p}_{\sigma_i}}=\{\epsilon^{+}_{i_1},\ldots, \epsilon^{+}_{i_n},
  \epsilon^{-}_{i_1},\ldots, \epsilon^{-}_{i_n}\}\] for
$1\leq i\leq d$. Recall from Definition \ref{def_modeled_graph} that the set
$\{\alpha(\epsilon_{i_j}^+), x\mid j=1,\ldots, n\}$ spans
$\mathfrak{t}_{\mathbb{Z}}^*\oplus\mathbb{Z} x$ i.e.,
\begin{equation}\label{smoothness}\langle
  \alpha(\epsilon_{i_1}^{+}),\ldots, \alpha(\epsilon_{i_n}^{+}),
  x\rangle=\mathfrak{t}_{\mathbb{Z}}^*\oplus \mathbb{Z} x.
\end{equation}

Let $\widetilde{\mathcal{G}}$ denote the $x$-forgetful graph
associated to $\mathcal{G}$ where
$\widetilde{\mathcal{G}}=(\Gamma, \widetilde{\alpha},\nabla)$ having
$\Gamma$ and $\nabla$ same as $\mathcal{G}$ and $\widetilde{\alpha}$
is the $x$-forgetful axial function defined as
$\widetilde{\alpha}=F\circ \alpha:\mathcal{E}\longrightarrow
H^2(BT^n)$ where
$F:H^2(BT^n)\bigoplus \mathbb{Z} x\longrightarrow H^2(BT^n)$ is the
$x$-forgetful map (see Section \ref{sect:5.2}). Recall that
\[\displaystyle \mathbb{Z}[\widetilde{\mathcal{G}}]:=\frac{\mathbb{Z}[L_1,\ldots,L_m]}{\langle\prod_{L\in \mathbf{L}'}
    L \mid \mathbf{L}'\in \mathbf{I}(\mathbf{L})\rangle}\]
where $\displaystyle\mathbf{I}(\mathbf{L})=\{\mathbf{L}'\subseteq \mathbf{L}
\mid \bigcap_{L\in\mathbf{L}'}L=\emptyset\}$. 

Let $x_{\gamma}$ denote the monomial
\[x_{\gamma}:=\prod_{j=1}^{p} L_{i_j}\] where
$\gamma=\langle v_{i_1},\ldots, v_{i_p}\rangle\in
  \Delta_{\mathbf{L}}$.

\subsection{The characteristic function associated to the hyperplane
  $L$}
\label{sect:7.2}

\begin{definition}
  Let $L$ be a connected $(2n-2)$-valent hyperplane in
  $\mathcal{G}$. Let $H$ and $\overline{H}$ be the unique halfspace
  such that $L=H\cap \overline{H}$.

For $p\in \mathcal{V}^{L}$,~~
\[\mathcal{E}^{L}_p=\{\epsilon_1^+,\ldots,\epsilon_{n-1}^+,\epsilon_1^{-},\ldots,\epsilon_{n-1}^{-}\}\]
is the $(n-1)$-pairs and 
\[\mathcal{E}^{\Gamma}_p=\{\epsilon_1^+,\ldots,\epsilon_{n}^+,\epsilon_1^{-},\ldots,\epsilon_{n}^{-}\}\]
is the $n$-pairs so that $n_{H}(p)=\epsilon_n^{+}$ and
$n_{\overline{H}}(p)=\epsilon_n^{-}$.

By (\ref{smoothness}), the axial functions
$\widetilde{\alpha}(\epsilon^{+}_{1}),
\widetilde{\alpha}(\epsilon^{+}_{2}), \ldots,
\widetilde{\alpha}(\epsilon^{+}_{n})$ form a basis for
$(\mathfrak{t}^n_{\mathbb{Z}})^*$.  The {\it characteristic function}
associated to $L$ is defined as the unique element
$\lambda(L)\in \mathfrak{t}^n_{\mathbb{Z}}$ such that
\begin{align*}
\langle 
\widetilde{\alpha}(\epsilon^+_i), \lambda(L)\rangle=\delta_{i,n}.
\end{align*}
\end{definition}

\begin{lemma}
  The definition of $\lambda(L)$ is independent of the choice of a vertex 
  $p\in \mathcal{V}^L$.
\end{lemma}

\begin{proof}
  Let $\epsilon=pq\in \mathcal{E}^L$, in particular let $\epsilon=\epsilon_j^+$ for
  some $1\leq j\leq n-1$. Here $i(\epsilon)=p$ and $t(\epsilon)=q$. Then under the
  connection ${\nabla}_{\epsilon}$, $\mathcal{E}_p^{\Gamma}$, the set of edges
  around $p$, maps bijectively onto $\mathcal{E}_q^{\Gamma}$, the set of
  edges around $q$.  Since the hyperplane $L$ is closed under the
  connection $\nabla$, $\mathcal{E}_p^{L}$ maps bijectively onto
  $\mathcal{E}_q^{L}$.  Moreover, since a halfspace $H$ is closed
  under $\nabla$, it follows that ${\nabla_{\epsilon}} (n_{H}(p))=n_{H}(q)$ so
  that
  $\widetilde{\alpha}({\nabla_{\epsilon}}
  (n_{H}(p)))\equiv \widetilde{\alpha}(n_{H}(q)) \mod \widetilde{\alpha}(\epsilon)$.  Moreover, since
  $\epsilon\in \mathcal{E}^{L}$ and $\nabla_{\epsilon}(\epsilon)=\overline{\epsilon}$ by definition of $\nabla_{\epsilon}$ it follows that the elements
  $\widetilde{\alpha}(\epsilon^{+}_1),\ldots,\widetilde{\alpha}(\epsilon_{n-1}^+)$
  and
  $\widetilde{\alpha}({\nabla}_{\epsilon}(\epsilon^{+}_1)),\ldots,\widetilde{\alpha}({\nabla}_{\epsilon}(\epsilon_{n-1}^+))$
  span the same subspace of
  $(\mathfrak{t}^n_{\mathbb{Z}})^*$. Further, since
  $\epsilon\in \mathcal{E}^L$ and
  $\langle\widetilde{\alpha}(\epsilon),\lambda(L)\rangle=0$, by the
  congruence relation we have
\[\langle 
  \widetilde{\alpha}({\nabla}_{\epsilon}(\epsilon_{n}^+)),\lambda(L)\rangle= \langle
  \widetilde{\alpha}(\epsilon_n^{+}),\lambda(L)\rangle.\] Thus for
the $n$-pairs
$\mathcal{E}^{\Gamma}_q=\{\nabla_{\epsilon}(\epsilon_1^+),\ldots,\nabla_{\epsilon}(\epsilon_{n}^+),
\nabla_{\epsilon}(\epsilon_1^{-}),\ldots,\nabla_{\epsilon}(\epsilon_{n}^{-})\}$ we have
\[\langle \lambda(L), \widetilde{\alpha}(\nabla_{\epsilon}(\epsilon_i^{+})) \rangle=\delta_{i,n}.\] Hence without loss of generality we could
have started with the vertex $q\in \mathcal{V}^{L}$ to define
$\lambda(L)$.  Moreover, since $L$ is connected, by repeating the
above procedure for an edge $\epsilon'$ such that $i(\epsilon')=q$, it follows that
the definition of $\lambda(L)$ is independent of the choice of
$p\in \mathcal{V}^{L}$.
\end{proof}

\subsection{The $H^*(BT^{n})$-algebra structure of
  $H^*(\widetilde{\mathcal {G}})$} 
\label{sect:7.3}

Since
$\displaystyle H^*(\widetilde{\mathcal G})\subset \bigoplus_{p\in
  \mathcal{V}}H^{*}_{T}(p)\simeq \bigoplus_{p\in \mathcal{V}}H^*(BT^n)$, the ring
$H^*(\widetilde{\mathcal G})$ may be regarded as the
$H^*(BT^{n})$-submodule of
$\displaystyle \bigoplus_{p\in \mathcal{V}}H^*(BT^{n})$.  In Theorem
\ref{module-generators} of this section, which is the second main
theorem of this paper, we determine module generators of
$H^*(\widetilde{\mathcal G})$ as a $H^*(BT^{n})$-module.  For this
purpose, we begin with the following lemma (also see \cite{MMP} for
the corresponding statement on torus graphs).

\begin{lemma}\label{$H^*(BT^n)$-algebra}
(i) The $H^*(BT^{n})$-module structure on $H^*(\widetilde{\mathcal G})$ is obtained from the following map from  $H^{2}(BT^{n})$ to $H^{*}(\widetilde{\mathcal G})$:  
\begin{align*}
H^{2}(BT^{n})\ni u\mapsto \displaystyle\sum_{i=1}^m\langle u, \lambda(L_{i})\rangle \cdot \tau_{L_i}\in
  H^{*}(\widetilde{\mathcal{G}}).
\end{align*}
Moreover, $\Psi':H^{*}(\widetilde{\mathcal{G}})\to \mathbb{Z}[\widetilde{\mathcal{G}}]$ is an isomorphism
  of $H^*(BT^{n})$-algebras where the algebra structure on
  $\mathbb{Z}[\widetilde{\mathcal{G}}]$ is obtained by sending
  $u\in H^2(BT^n)$ to the element
  $\displaystyle\sum_{i=1}^m\langle u, \lambda(L_{i})\rangle \cdot L_i\in
  \mathbb{Z}[\widetilde{\mathcal{G}}]$.

  (ii) We therefore have the following presentation for
  $H^*(\widetilde{\mathcal{G}})$ as an $H^*(BT^{n})$-algebra:
  \[H^*(\widetilde{\mathcal{G}})\simeq \frac{H^*(BT^n)[L_1,\ldots,
      L_m]}{\langle\prod_{L\in \mathbf{L}'}L\mid \mathbf{L}'\in
      \mathbf{I}(\mathbf{L}) ~; ~\sum_{i=1}^m\langle u,
      \lambda(L_{i})\rangle \cdot L_i-u ,~~\forall ~u\in H^2(BT^n)\rangle}\]

\end{lemma}  
\begin{proof}
  Let $p\in \mathcal{V}^{L}$ where
  $p=L_{i_1}\cap \cdots \cap L_{i_n}$. Then by Section \ref{sect:5.1}
  we have \begin{equation}\label{algebra}\sum_{i=1}^m\langle u,
    \lambda(L_{i})\rangle \cdot \tau_{L_i}(p)=\sum_{j=1}^n\langle u,
    \lambda(L_{i_j})\rangle \cdot
    \widetilde{\alpha}(n_{H_{i_j}}(p)).\end{equation} Note that
  $n_{H_{i_j}}(p)=\epsilon_j^+(p)$ for $1\leq j\leq n$ so that
  $\widetilde{\alpha}(n_{H_{i_j}}(p))$ for $1\leq j\leq n$ form a
  basis of $\mathfrak{t}^*_{\mathbb{Z}}$. Since
  $\lambda(L_{i_j})\in \mathfrak{t}_{\mathbb{Z}}$ for $1\leq j\leq n$
  is the corresponding dual basis, the right hand side of
  (\ref{algebra}) is nothing but $u$.  Thus
  \begin{equation}\label{eu}(\sum_{i=1}^m\langle u, \lambda(L_{i})\rangle \cdot
    \tau_{L_i})(p)=u~~\mbox{for every}~~ p\in \mathcal{V}.
    \end{equation} 
    Since $p\in \mathcal{V}$ was arbitrary from
  Section \ref{sect:5.3} it follows that the $H^*(BT^n)$-algebra structure
  defined above is canonical corresponding to the diagonal inclusion
  of $H^*(BT^n)$ in $\bigoplus_{i=1}^{d}H^{*}(BT^{n})=(H^*(BT^n))^d$.

Finally, by definition of $\Psi'$ in Section~\ref{sect:5.2}, we also have that the $H^{*}(BT^{n})$-algebra structure on $\mathbb{Z}[\widetilde{\mathcal{G}}]$ is obtained as in the statement.
\end{proof} 

The following theorem is the second main theorem in this paper.

\begin{theorem}
\label{module-generators}
Let $\Delta_{\mathbf{L}}$ be the simplicial complex defined by the
hyperplanes $\{L_{1},\ldots, L_{m}\}$ of $\widetilde{\mathcal{G}}$.
Suppose that $\Delta_{\mathbf{L}}$ is a shellable simplicial complex
with respect to ordering $\sigma_1,\ldots, \sigma_d$ of
$\Delta_{\mathbf{L}}(n)$. In particular,
$\Delta_{\mathbf{L}}$ is partitionable with partition
(\ref{partition}) and (\ref{star}) holds. Then the following
statements hold:

  (i) For $\gamma\in \Delta_{\mathbf{L}}$, 
let
  $x_{\gamma}={L_{j_1}}\cdots {L_{j_p}}\in
  \mathbb{Z}[\widetilde{\mathcal{G}}]$ where
  $\gamma=\langle v_{j_1},\ldots, v_{j_p}\rangle$. 
Then there exists an element $u\in H^2(BT^n)$ such that
\begin{align*}
L_{j_1}\cdot x_{\gamma}=-\sum_{k}\langle
    u,\lambda(L_{j_k}) \rangle \cdot
    x_{\gamma_k}+u\cdot x_{\gamma}
\end{align*}
where $k$ runs through $1\le k\le m$ such that $k\not\in \{j_{1},\ldots, j_{n}\}$ and $\gamma_{k}=\langle v_{k},v_{j_1},\ldots, v_{j_p}\rangle$.

(ii) Let $\eta\preceq \gamma\preceq \theta$ be simplices in
$\Delta_{\mathbf{L}}$. Then we can write
\[x_{\gamma}=\sum_{k} c_k\cdot x_{\eta_{k}}+c\cdot x_{\eta}\]
for $c_k,c\in H^{*}(BT^{n})$ and $\eta_{k}\npreceq \theta$.

(iii) The monomials $x_{\mu_i}$ for $1\leq i\leq d$ form a basis of
$\mathbb{Z}[\widetilde{\mathcal{G}}]$ as $H^{*}(BT^{n})$-module.

(iv) Let $f\in \mathbb{Z}[\widetilde{\mathcal{G}}]$ and
\begin{equation}\label{coefficients} f=\sum_{j=1}^da_j\cdot x_{\mu_j}\end{equation} for unique
$a_j\in H^*(BT^n)$. Let $i=i(f)$ be the smallest $1\le i\le d$ such that
$a_i\neq 0$. Then we can determine the coefficients $a_j$, $j\geq i$
iteratively as follows: We have
$a_i=\frac{\rho_{\mathbf{p}_{\sigma_{i}}(f)}}{\rho_{\mathbf{p}_{\sigma_i}}(x_{\mu_{i}})}$. Suppose
$a_i, a_{i+1},\ldots, a_{j-1}$ are determined by induction then
\[\displaystyle a_{j}=\frac{\rho_{\mathbf{p}_{\sigma_{j}}}(f-\sum_{k=i}^{j-1}a_k\cdot
  x_{\mu_k})} {\rho_{\mathbf{p}_{\sigma_j}}(x_{\mu_{j}})}.\] 
\end{theorem}  

\begin{proof} (i) Let $\sigma=\langle v_{j_1},\ldots, v_{j_n}\rangle$
be an $(n-1)$-simplex containing $\gamma$ and let
${\bf p}_{\sigma}=L_{j_1}\cap\cdots\cap L_{j_n}$ be the corresponding
vertex in $\widetilde{\mathcal{G}}$. Consider
$\lambda(L_{j_1})=e_{j_1}\in \mathfrak{t}^n_{\mathbb{Z}}$ which is
dual to $u:=\widetilde{\alpha}(\epsilon^{+}_{j_1})$.
In this case, by Lemma~\ref{$H^*(BT^n)$-algebra}, 
we have that
\begin{align*}
u=\sum_{i=1}^{m}\langle u, \lambda(L_{i})\rangle \cdot L_i.
\end{align*}
Since $\langle u,e_{j_i}\rangle=\delta_{i,1}$ for $i=1,\ldots, n$,
we have the relation
\begin{equation}
\label{rel1} 
L_{j_1}=-\sum_{k}\langle
  u, \lambda(L_{k}) \rangle
  \cdot L_k+u
\end{equation}
as an $H^{*}(BT^{n})$-module, where the sum on the
right hand side is over those $1\leq k\leq m$ such that
$k\notin\{j_1,\ldots, j_n\}$.  Multiplying (\ref{rel1}) by
$x_{\gamma}$ we get
\begin{equation}\label{rel2} L_{j_1}\cdot x_{\gamma}= -\sum\langle
  u,\lambda(L_{k}) \rangle
  \cdot x_{\gamma_{k}}+u\cdot x_{\gamma} \end{equation}
where $\gamma_{k}\in \Delta_{\mathbf{L}}$ is spanned by the
vertices $v_{k},v_{j_1},\ldots, v_{j_p}$ and
$x_{\gamma_k}=L_k\cdot L_{j_1}\cdots L_{j_p}$.
This proves (i)

(ii) Let $\eta\prec \gamma\preceq \theta$ be simplices in
$\Delta_{\mathbf{L}}$.  Let
$\gamma=\langle v_{j_1},\ldots, v_{j_p}\rangle$,
$\eta=\langle v_{j_1},\ldots, v_{j_r}\rangle\in \Delta_{\mathbf{L}}$
and $\theta=\langle v_{j_1},\ldots, v_{j_l}\rangle$ and $r<p\leq l$.
Because the following argument can also apply for any such $l$, we can
assume that $l=n$ i.e., $\theta$ is $(n-1)$-dimensional.  Thus
$\lambda(L_{j_1}) ,\ldots, \lambda(L_{j_n})$ is the basis of
$\mathfrak{t}_{\mathbb{Z}}$ dual to the basis
$\widetilde{\alpha}(\epsilon_{j_1}),\ldots,
\widetilde{\alpha}(\epsilon_{j_n})$ of $\mathfrak{t}^*_{\mathbb{Z}}$.
Let $u:=\widetilde{\alpha}(\epsilon_{j_p})\in H^{*}(BT^{n})$.  Thus,
with the similar reason to obtain \eqref{rel1} as in the proof of (i),
we have the following relation in
$\mathbb{Z}[\widetilde{\mathcal{G}}]$:
\begin{equation} 
\label{e1} L_{j_p}+\sum_{k}\langle u,
    \lambda(L_{k})\rangle \cdot L_k-u=0
\end{equation}
where $k$ runs through $\{1,\ldots, m\}\setminus \{j_1,\ldots, j_n\}$ in the equation.  Multiplying  (\ref{e1}) by ${L_{j_1}}\cdots {L_{j_{p-1}}}$, we get
\begin{equation} \label{e'2} x_{\gamma}+\sum_{k}\langle u,
    \lambda(L_{k})\rangle \cdot x_{\eta_{k}}- u\cdot x_{\gamma'}=0
\end{equation} 
where $\eta_{k}=\langle
  v_{k}, v_{j_1},\ldots, v_{j_{p-1}}\rangle$,
  $\gamma'=\langle v_{j_1},\ldots, v_{j_{p-1}} \rangle\in
  \Delta_{\mathbf{L}}$.
Note that $\eta_{k}\npreceq \theta$ since $k\notin \{j_1,\ldots,
  j_n\}$. Also $\eta\preceq \eta_{k}$ and $\eta\preceq\gamma'$, since
  $r\leq p-1$. Therefore, $\eta\preceq\gamma'\prec\gamma\preceq \theta$.
Now, proceeding by downward induction on $p$ and
  repeating the above arguments for $\gamma'$ we arrive at (ii).

  (iii) By the ring structure of $\mathbb{Z}[\widetilde{\mathcal{G}}]$
  defined in Section~\ref{sect:5.2}, for every element in
  $\mathbb{Z}[\widetilde{\mathcal{G}}]$ can be written by the sum of
  $x_{\gamma}$'s for $\gamma\in \Delta_{\mathbf{L}}$ with
  $H^{*}(BT^{n})$-coefficients.  Therefore, for every
  $\gamma\in \Delta_{\mathbf{L}}$, it suffices to show that
  $x_{\gamma}$ lies in the $H^{*}(BT^{n})$-submodule of
  $\mathbb{Z}[\widetilde{\mathcal{G}}]$ spanned by $x_{\mu_i}$ for
  $1\leq i\leq d$, where $\mu_{i}$ is the minimal face which appears
  in
  $\Delta_{\mathbf{L}}=[\mu_1,\sigma_1]\sqcup\cdots\sqcup
  [\mu_d,\sigma_d]$. Since $\Delta_{\mathbf{L}}$ is a shellable
  simplicial complex, for every $\gamma\in \Delta_{\mathbf{L}}$ there
  exists the unique $1\leq i\leq d$ such that
  $\mu_i\preceq \gamma\preceq \sigma_i$ (see (\ref{star})).  We prove
  (iii) by downward induction on $i$.

  If $\gamma\in [\mu_{d},\sigma_{d}]$, we are done since
  $\mu_d=\sigma_d=\gamma$, i.e., $x_{\gamma}=x_{\mu{d}}$ and hence
  lies in the $H^{*}(BT^{n})$-span of $x_{\mu_{d}}$.  
  Assume that for
  every
  $\gamma\in [\mu_{i+1},\sigma_{i+1}]\sqcup\cdots\sqcup
  [\mu_d,\sigma_d]$,
  $x_{\gamma}\in H^{*}(BT^{n}) x_{\mu_{i+1}}\oplus\cdots\oplus
  H^{*}(BT^{n}) x_{\mu_{d}}$.  If
  $\gamma\in [\mu_{i},\sigma_{i}]$, then by (ii) we can write
\begin{equation}\label{e2} x_{\gamma}=\sum_{\mu_i\prec\gamma_j\nprec
    \sigma_i}c_j\cdot x_{\gamma_j}+c\cdot x_{\mu_i}
\end{equation} for
$c_j,c\in H^{*}(BT^{n})$.
Now there is the unique $r$ such that $\mu_r\preceq
\gamma_j\preceq\sigma_r$. This implies by \eqref{star} that $r>i$.

Thus by induction assumption $x_{\gamma_j}$ lies in the
$H^{*}(BT^{n})$-span of $x_{\mu_q}$ for $q\geq r$.
This together with (\ref{e2}) implies that $x_{\gamma}$ lies in the
$H^{*}(BT^{n})$-span of $x_{\mu_q}$ for $q\geq i$. 

It remains now to show that $x_{\mu_i}$ for $1\leq i\leq d$ are
linearly independent.
Suppose that there exist $a_i\in H^{*}(BT^{n})$ for $1\leq i\leq d$
such that
\begin{equation}\label{linearrelation}\sum_{i=1}^d a_i\cdot
  x_{\mu_i}=0\end{equation} in
$\mathbb{Z}[\widetilde{\mathcal{G}}]$. Let
$i\in \{1,\ldots, n\}$ be the smallest integer such that $a_i\neq 0$.

Recall that
$\sigma_i=\langle v_{i_1},\ldots, v_{i_n}\rangle$ where
$\mathbf{p}_{\sigma_i}=L_{i_1}\cap\cdots \cap L_{i_n}$ in
$\mathcal{G}$. 
Consider the localization map
$\rho=(\rho_{\mathbf{p}_{\sigma_j}})_{j=1}^d$ defined in Section
\ref{sect:5.3}. By (\ref{star}) and the definition of
$\rho_{\mathbf{p}_{\sigma_i}}$ it follows that
$\rho_{\mathbf{p}_{\sigma_i}}(x_{\mu_j})=0$ for $j>i$ (since
$\mu_j\npreceq \sigma_i$ there exists $L_k$ in $\mathbf{L}$ such that
the corresponding vertex $v_k\in \mu_j$ and $v_k\notin \sigma_i$ in
$\Delta_{\mathbf{L}}$. Thus $\rho_{\mathbf{p}_{\sigma_i}}(L_k)=0$ in
$\mathbb{Z}[\widetilde{\mathcal{G}}]_{\mathbf{p}_{\sigma_i}}$). Thus applying
$\rho_{\mathbf{p}_{\sigma_i}}$ on (\ref{linearrelation}) we get
\[\rho_{\mathbf{p}_{\sigma_i}}(\sum_{j=1}^n a_j\cdot
  x_{\mu_j})=\rho_{\mathbf{p}_{\sigma_i}}(\sum_{j\geq i} a_j\cdot
  x_{\mu_j})=\rho_{\mathbf{p}_{\sigma_i}} (a_i \cdot
  x_{\mu_i})=\rho_{\mathbf{p}_{\sigma_i}} (a_i)\cdot
  \rho_{\mathbf{p}_{\sigma_i}} (x_{\mu_i})=0\] in the integral domain
$\mathbb{Z}[\widetilde{\mathcal{G}}]_{\mathbf{p}_{\sigma_i}}\simeq
\mathbb{Z}[L_{i_1},\ldots, L_{i_n}]$. Since
${\rho_{\mathbf{p}_{\sigma_i}} (x_{\mu_i})}$ is the monomial
$L_{i_{j_1}},\ldots, L_{i_{j_p}}$, where
$\mu_i=\langle v_{i_{j_1}},\ldots, v_{i_{j_p}} \rangle$, and hence a
non-zero element of $\mathbb{Z}[L_{i_1},\ldots, L_{i_n}]$, we get that
$\rho_{\mathbf{p}_{\sigma_i}} (a_i)=0$. Moreover, $\rho$ can be seen
to be the diagonal embedding when restricted to the subalgebra
$H^*(BT^n)$ of $\mathbb{Z}[\widetilde{\mathcal{G}}]$ ($u\in H^*(BT^n)$
is equal to
$\displaystyle\sum_{j=1}^m \langle u,\lambda({L_j})\rangle\cdot L_j\in
\mathbb{Z}[\widetilde{\mathcal{G}}]$ maps to
$\displaystyle\sum_{j=1}^n \langle u,\lambda({L_{i_j}}) \rangle\cdot
L_{i_j}\in \mathbb{Z}[L_{i_1},\ldots, L_{i_n}]$ which is identified
with
$\displaystyle\sum_{j=1}^n \langle u,\lambda({L_{i_j}}) \rangle\cdot
\tau_{L_{i_j}} $ in $H^*(BT^n)$ (see Lemma \ref{restricted-to-pt})
which is equal to $u$ see (\ref{eu})) and $\rho$ is injective by
Lemma \ref{rho_inje}, which implies that $a_i=0$. This contradicts our
original assumption that $a_i\neq 0$. Thus we cannot have a relation
of the type (\ref{linearrelation}) in
$\mathbb{Z}[\widetilde{\mathcal{G}}]$ unless $a_i=0$ in $H^*(BT^n)$
for each $1\leq i\leq d$. Hence we conclude that $x_{\mu_i}$ for
$1\leq i\leq d$ are linearly independent in
$\mathbb{Z}[\widetilde{\mathcal{G}}]$. This proves (iii).

(iv) As for the proof of linear independence of $x_{\mu_i}$
$1\leq i\leq d$ our idea is to again use the localization map
$\displaystyle\rho=(\rho_{\mathbf{p}_{\sigma_i}}):
\mathbb{Z}[\widetilde{\mathcal{G}}]\longrightarrow
\bigoplus_{i=1}^{d}\mathbb{Z}[\widetilde{\mathcal{G}}]_{\mathbf{p}_{\sigma_i}}
(\simeq \bigoplus_{i=1}^{d} H^*(BT^n))$ defined in Section
\ref{sect:5.3}. We know that $\rho$ is injective and by (\ref{star})
and the definition of $\rho_{\mathbf{p}_{\sigma_i}}$ we have
$\rho_{\mathbf{p}_{\sigma_i}}(x_{\mu_j})=0$ for $j>i$. Also $\rho$ is the
diagonal map when restricted to the subalgebra $H^*(BT^n)$ of
$\mathbb{Z}[\widetilde{\mathcal{G}}]$. In particular, this implies that
$\rho_{\mathbf{p}_{\sigma_j}}(a_k)=a_k$ for $1\leq j,k\leq d$. 

Applying $\rho_{\mathbf{p}_{\sigma_i}}$ on (\ref{coefficients}) we get
\[\rho_{\mathbf{p}_{\sigma_i}}
  (f)=\rho_{\mathbf{p}_{\sigma_i}}(\sum_{j=1}^da_j\cdot
  x_{\mu_j})=\rho_{\mathbf{p}_{\sigma_i}}(\sum_{j\geq i}a_j\cdot x_{\mu_j}) = a_i\cdot
  \rho_{\mathbf{p}_{\sigma_i}}(x_{\mu_i})\] in the unique
factorization domain $\mathbb{Z}[L_{i_1},\ldots, L_{i_n}]$ where
$\mathbf{p}_{\sigma_i}=L_{i_1}\cap\cdots \cap L_{i_n}$. Thus
$\rho_{\mathbf{p}_{\sigma_i}}(f)$ is divisible by the irreducible
elements $L_{i_{j_1}},\ldots, L_{i_{j_p}}$ and hence by the monomial
$\rho_{\mathbf{p}_{\sigma_i}}(x_{\mu_i})=L_{i_{j_1}}\cdots L_{i_{j_p}}$ in
$\mathbb{Z}[L_{i_1},\ldots, L_{i_n}]$. Thus
$a_i=\frac{\rho_{\mathbf{p}_{\sigma_i}}(f)}{\rho_{\mathbf{p}_{\sigma_i}}(x_{\mu_i})}\in
\mathbb{Z}[L_{i_1},\ldots, L_{i_n}]$. Now, let
\[f_1:=f-a_{i(f)}\cdot x_{\mu_{i(f)}}\in
\mathbb{Z}[\widetilde{\mathcal{G}}].\] Then
$\displaystyle f_1=\sum_{j>i(f)}a_j \cdot x_{\mu_j} .$ Moreover, now
putting $i=i(f_1)$ and repeating the above argument given for
determining $a_{i(f)}$ we get
\[a_{i(f_1)}=\frac{\rho_{\mathbf {p}_{\sigma_{i(f_1)}}}(
    f_1)}{\rho_{\mathbf {p}_{\sigma_{i(f_1)} }}(x_{\mu_{i(f_1)}})}\] in
$\mathbb{Z}[\widetilde{\mathcal{G}}]_{\mathbf {p}_{\sigma_{i(f_1)}}}\simeq H^*(BT^n)$.
Proceeding similarly after $k$ steps we get
\[\displaystyle f_{k}=f-\sum_{i(f_{1})\le j<i(f_{k})} a_j\cdot
x_{\mu_j}=\sum_{j>i(f_{k-1})}a_j\cdot x_{\mu_j}.\] Putting $i=i(f_k)$
to be the smallest index in $\{i(f_{k-1}), i(f_{k-1})+1,\ldots, d\}$ such that $a_{i(f_k)}\neq
0$ and following similar arguments as above we get
\[a_{i(f_k)}=\frac{\rho_{\mathbf{p}_{\sigma_{i(f_k)}}}(
    f_k)}{\rho_{\mathbf {p}_{\sigma_{i(f_k)} }}(x_{\mu_{i(f_k)}})}\]  in
$\mathbb{Z}[\widetilde{\mathcal{G}}]_{\mathbf {p}_{\sigma_{i(f_k)}}}\simeq H^*(BT^n)$. This
proves (iv).

\end{proof}

\subsection{The ordinary cohomology ring of $\widetilde{\mathcal{G}}$ and an example}
\label{sect:7.4}

In geometry, if the equivariant cohomology $H_{T}^{*}(M;\mathbb{Z})$
has the structure of a free $H^{*}(BT)$-algebra, then we can compute
its ordinary cohomology $H^{*}(M;\mathbb{Z})$ by
$H_{T}^{*}(M;\mathbb{Z})\otimes_{H^*({BT})}\mathbb{Z}$.  By
Theorem~\ref{module-generators}, we know that
$H^{*}(\widetilde{\mathcal{G}})$ is a free $H^{*}(BT)$-algebra
of rank $d$.  So we may define the ``ordinary'' cohomology of
$\widetilde{\mathcal{G}}$ by
$H^*(\widetilde{\mathcal{G}})\otimes_{H^*({BT})}\mathbb{Z}$; we
denote it by $H^*_{{ord}}(\widetilde{\mathcal{G}})$.  The precise
computation of $H^*_{{ord}}(\widetilde{\mathcal{G}})$ is given as the
following corollary.
\begin{corollary}
\label{ord-cohomology}
(i) The following is the presentation
\[H_{ord}^*(\widetilde{\mathcal{G}})\simeq
  \frac{\mathbb{Z}[L_1,\ldots, L_m]}{\langle\prod_{L\in
      \mathbf{L}'}L\mid \mathbf{L}'\in \mathbf{I}(\mathbf{L}) ~;
    ~\sum_{i=1}^m\langle u, \lambda(L_{i})\rangle \cdot L_i ,~~\forall
    ~u\in H^2(BT^n)\rangle}\] for the ordinary cohomology ring
$H^*_{{ord}}(\widetilde{\mathcal{G}})$ as a $\mathbb{Z}$-algebra.

(ii) The monomials $x_{\mu_i}$, $1\leq i\leq d$ form a
$\mathbb{Z}$-basis for $H^*_{{ord}}(\widetilde{\mathcal{G}})$.

\end{corollary}  
\begin{proof} 
  (i) The $H^{*}(BT^{n})$-algebra structure on
  $H^*(\widetilde{\mathcal{G}})$ is given by Lemma
  \ref{$H^*(BT^n)$-algebra} and $\mathbb{Z}$ has $H^*({BT})$-algebra
  structure given by augmentation which sends each $u\in H^2(BT^n)$ to
  $0$.  Since $H^{*}(\widetilde{\mathcal{G}})$ is
    free, the corollary now follows from Lemma
  \ref{$H^*(BT^n)$-algebra} due to the $H^*({BT})$-algebra
  isomorphism $\Psi'$ of $\mathbb{Z}[\widetilde{\mathcal{G}}]$ with
  $H^*(\widetilde{\mathcal{G}})$.

  (ii) This follows by Theorem \ref{module-generators} (iii) and by
  the isomorphism
  \[H_{ord}^*(\widetilde{\mathcal{G}})\simeq
  \mathbb{Z}[\widetilde{\mathcal{G}}]\otimes_{H^*(BT^n)} \mathbb{Z}.\]
\end{proof}


\subsection{$H^*(\widetilde{\mathcal{G}})$ for
  $\mathcal{G}$ induced from the $8$-dimensional toric
  hyperK{$\ddot{\mathrm{a}}$}hler manifold.}
\label{sect:7.5}

By using the fundamental theorem of toric
hyperK{$\ddot{\mathrm{a}}$}hler manifolds in \cite{BD}, the $8$
dimensional toric hyperK{$\ddot{\mathrm{a}}$}hler manifold $M$ is
completely classified up to equivariant diffeomorphism by the
hyperplane arrangement $\mathcal{L}_{k,l,m}$ in $\mathbb{R}^{2}$
consisting of $k$ horizontal lines $\{Hor_1,\ldots, Hor_k\}$ which is
ordered from the bottom, $l$ virtical lines $\{Vir_1,\ldots, Vir_l\}$
which is ordered from the left and $m$ diagonal lines
$\{Dia_1,\ldots, Dia_m\}$ which is ordered from the left in $\mathbb{R}^2$
(also see Figure~\ref{Figure_shelling}).

It is easy to check that every set of hyperplanes in
$\mathcal{L}_{k,l,m}$ have the non-empty intersections except the
following cases:
\begin{align*}
& Hor_r\cap Vir_s\cap Dia_t=\emptyset \quad \text{for}\ 1\leq r\leq k, 1\leq s\leq l\ \text{and}\ 1\leq t\leq m;
\end{align*}
and 
\begin{align*}
& Hor_i\cap Hor_j=\emptyset\quad \text{for}\ 1\leq i,j\leq k; \\
& Vir_r\cap Vir_s=\emptyset\quad \text{for}\ 1\leq r,s\leq l; \\
& Dia_p\cap Dia_q=\emptyset\quad \text{for}\ 1\leq p,q\leq m.
\end{align*}
This hyperplane arrangement $\mathcal{L}_{k,l,m}$ induces the
$T^{*}\mathbb{C}^{2}$-modeled GKM graph $\mathcal{G}$.  We can see
that the characteristic functions associated to the hyperplanes are
given by $\lambda({Hor_r})=e_1$ for all $1\leq r\leq k$,
$\lambda(Vir_s)=e_2$ for all $1\leq s\leq l$ and $\lambda(Dia_p)=-e_1-e_2$
for all $1\leq p\leq m$ where
$H_2(BT^2)=\mathbb{Z}\cdot e_1\bigoplus \mathbb{Z}\cdot e_2$.
Therefore, the $x$-forgetful graph
  $\widetilde{\mathcal{G}}$ is
  given by Figure~\ref{Figure_example}.
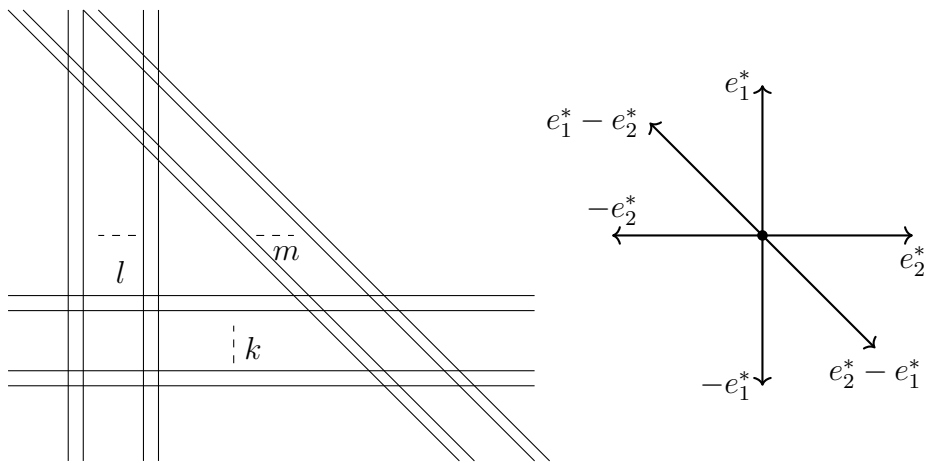
\begin{figure}[h]
\begin{tikzpicture}
\begin{scope}[xscale=1.0, yscale=1.0]


\draw (-3,3)--(3,-3);
\draw (-2.8,3)--(3.2,-3);
\draw[dashed] (0.3,0)--(0.8,0);
\node[below] at (0.7,0) {$m$};
\draw (-2,3)--(4,-3);
\draw (-1.8,3)--(4.2,-3);

\draw (-1,3)--(-1,-3);
\draw (-1.2,3)--(-1.2,-3);
\draw[dashed] (-1.3,0)--(-1.8,0);
\node[below] at (-1.5,-0.2) {$l$};
\draw (-2,3)--(-2,-3);
\draw (-2.2,3)--(-2.2,-3);

\draw (-3,-1.8)--(4,-1.8);
\draw (-3,-2)--(4,-2);
\draw[dashed] (0,-1.7)--(0,-1.2);
\node[right] at (0,-1.5) {$k$};
\draw (-3,-1)--(4,-1);
\draw (-3,-0.8)--(4,-0.8);

\end{scope}

\begin{scope}[xshift=200, xscale=1.0, yscale=1.0]
\fill (0,0) coordinate (o) circle (2pt);
\draw[->, thick] (o)--(2,0);
\node[below] at (2,0) {$e_{2}^{*}$};
\draw[->, thick] (o)--(-2,0);
\node[above] at (-2,0) {$-e_{2}^{*}$};
\draw[->, thick] (o)--(1.5,-1.5);
\node[below] at (1.5,-1.5) {$e_{2}^{*}-e_{1}^{*}$};
\draw[->, thick] (o)--(-1.5,1.5);
\node[left] at (-1.5,1.5) {$e_{1}^{*}-e_{2}^{*}$};
\draw[->, thick] (o)--(0,2);
\node[left] at (0,2) {$e_{1}^{*}$};
\draw[->, thick] (o)--(0,-2);
\node[left] at (0,-2) {$-e_{1}^{*}$};
\end{scope}
\end{tikzpicture}
\caption{The $x$-forgetful graph induced from $\mathcal{L}_{k,l,m}$. The axial functions on four edges around each vertex (each intersection of two lines) are defined by choosing the labels in the right figure for each direction, where $\{ e_{1}^{*},e_{2}^{*}\}\subset \algt^{*}_{\mathbb{Z}}$ is the dual basis of $\{e_{1},e_{2}\}\subset \algt_{\mathbb{Z}}$. 
For example, Figure~\ref{fig-x-forget} is the $x$-forgetfull graph of $\mathcal{L}_{1,1,1}$.}
\label{Figure_example}
\end{figure}

Consider the polynomial ring
$$R:=\mathbb{Z}[X_1,\ldots, X_k, Y_1,\ldots, Y_l, Z_1,\ldots, Z_m]$$
  in $k+l+m$ variables. 
Let $I$ be the ideal in $R$ generated by the
following monomials:
\begin{align*}
  & X_{r}Y_{s}Z_{t}\quad \text{for}\ 1\leq r\leq k,\ 1\leq s\leq l\
    \text{and}\ 1\leq t\leq m; \\
  & X_iX_j\quad \text{for}\ 1\leq i\neq j\leq k; \\
  & Y_rY_s\quad \text{for}\  1\leq r\neq s\leq l; \\
  & Z_pZ_q\quad \text{for}\ 1\leq p\neq q\leq m.
\end{align*}

It follows from Theorem \ref{main-theorem1-1} that $R/I$ is isomorphic to
  $H^*(\widetilde{\mathcal{G}})$ under the map
  which is defined by the following correspondences:
\begin{align*}
& X_r\mapsto \tau_{Hor_r}\quad \text{for}\ 1\leq r\leq k; \\
& Y_s\mapsto \tau_{Vir_s}\quad \text{for}\  1\leq s\leq l; \\
& Z_t\mapsto \tau_{Dia_t}\quad \text{for}\ 1\leq t\leq m.
\end{align*}

Now we determine the structure of $H^*(\widetilde{\mathcal{G}})$ as an
$H^*(BT^2)$-algebra.

Let
$u\in H^2(BT^2)=\mathbb{Z}\cdot e_1^*\bigoplus \mathbb{Z}\cdot e_2^*$
and $u=a\cdot e_1^*+b\cdot e_2^*$.  Then under the $H^*(BT^2)$-algebra
structure on $\mathbb{Z}[\widetilde{\mathcal{G}}]$, $u$ corresponds to
the element
\begin{align*}&\sum_{r=1}^ka\cdot X_r+\sum_{s=1}^lb\cdot
  Y_s-\sum_{t=1}^m(a+b) \cdot Z_t\\ &=a\cdot
  (X_1+\cdots+X_k-Z_1-\cdots-Z_m)+b\cdot
  (Y_1+\cdots+Y_l-Z_1-\cdots-Z_m).\end{align*} Let
$\mathcal{R}:=H^*(BT^2)[X_1,\ldots, X_k, Y_1,\ldots, Y_l, Z_1,\ldots,
Z_m]$ and $\mathcal{I}$ be the ideal in $\mathcal{R}$ generated by the
monomials generating the ideal $I$ in $R$, together with the following
two linear polynomials:
\begin{align*}
X_1+\cdots+X_k-Z_1-\cdots-Z_m-e_1^*; \quad
Y_1+\cdots+Y_l-Z_1-\cdots-Z_m-e_2^*.
\end{align*}
Then it follows from Lemma \ref{$H^*(BT^n)$-algebra} that the ring
$\mathcal{R}/\mathcal{I}$ is isomorphic to $H^{*}(\widetilde{\mathcal{G}})$
as an $H^*(BT^2)$-algebra.

We now note that the simplicial complex $\Delta_{\mathcal{L}}$ dual to
the hyperplane arrangement $\mathcal{L}$ has vertices
$u_1,\ldots, u_k$ corresponding to the hyperplanes $Hor_1,\ldots, Hor_k$,
$v_1,\ldots, v_l$ corresponding to the hyperplanes $Vir_1,\ldots, Vir_l$
and $w_1,\ldots, w_m$ corresponding to the hyperplanes
$Dia_1,\ldots Dia_m$.

  Moreover, $\Delta_{\mathcal{L}}$ is a $1$-dimensional simplicial
  complex where the number of $1$-simplices, i.e., the vertices of the $x$-forgetful graph, in $\Delta_{\mathcal{L}}$ is
  $kl+km+lm$. We can see that $\Delta_{\mathcal{L}}$ is shellable with the
  following shelling order of the $1$-dimensional simplices:
\begin{align*}
 &\hspace{0.5cm}\sigma_1=[u_1,v_1]< \sigma_2=[u_1,v_2]<\cdots<\sigma_l=[u_1,v_l]\\
   & < \sigma_{l+1}=[u_1,w_1]<\sigma_{l+2}=[u_1,w_2]<\cdots< \sigma_{l+m}=[u_1,w_m]\\
   & <\sigma_{l+m+1}=[u_2,v_1]<\sigma_{l+m+2}=[u_2,v_2]<\cdots<\sigma_{2l+m}=[u_2,v_l]\\
   &<\sigma_{2l+m+1}=[u_2,w_1]<\sigma_{2l+m+2}=[u_2,w_2]<\cdots<\sigma_{2l+2m}=[u_2,w_m]\\
   &\hspace{6cm}\vdots\\
   &<\sigma_{(k-1)\cdot l+(k-1)\cdot m+1}=[u_k,v_1]<
     \sigma_{(k-1)\cdot l+(k-1)\cdot
     m+2}=[u_k,v_2]<\cdots<\sigma_{kl+(k-1)\cdot m}=[u_k,v_l]\\ 
   &<\sigma_{kl+(k-1)\cdot m+1}=[u_k,w_1]<\sigma_{kl+(k-1)\cdot
     m+2}=[u_k,w_2]<\cdots<\sigma_{kl+km}=[u_k,w_m]\\ & < \sigma_{kl+km+1}=
                                                                                               [v_1,w_1]<\sigma_{kl+km+2}=[v_1,w_2]<\cdots<\sigma_{kl+km+m}=[v_1,w_m]\\ &\hspace{6cm}\vdots\\
   & <\sigma_{kl+km+(l-1)\cdot m
     +1}=[v_l,w_1]<\sigma_{kl+km+(l-1)\cdot m+2}=[v_l,w_2]<\cdots<\sigma_{kl+km+lm}=[v_l,w_m].
  \end{align*}

For example, we give the order on vertices in $\widetilde{\mathcal{G}}$ induced from $\mathcal{L}_{2,1,2}$, i.e., $1$-simplicies in $\Delta_{\mathcal{L}}$ as in Figure~\ref{Figure_shelling}.

  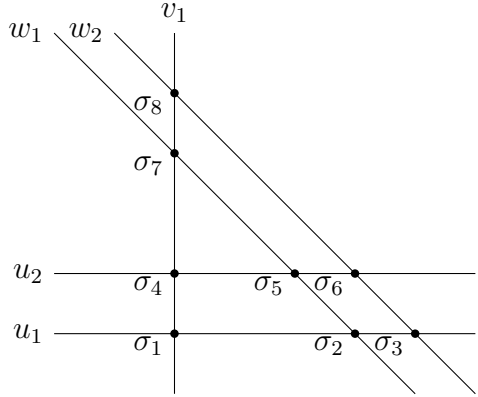
\begin{figure}[h]
\begin{tikzpicture}
\begin{scope}[xscale=0.8, yscale=0.8]

\draw (-3,3)--(3,-3);
\node[left] at (-3,3) {$w_{1}$};
\draw (-2,3)--(4,-3);
\node[left] at (-2,3) {$w_{2}$};

\draw (-1,3)--(-1,-3);
\node[above] at (-1,3) {$v_{1}$};

\draw (-3,-2)--(4,-2);
\node[left] at (-3,-2) {$u_{1}$};
\draw (-3,-1)--(4,-1);
\node[left] at (-3,-1) {$u_{2}$};

\fill (-1,-2) circle (2pt);
\node[left] at (-1,-2.2) {$\sigma_{1}$};
\fill (2,-2) circle (2pt); 
\node[left] at (2,-2.2) {$\sigma_{2}$};
\fill (3,-2) circle (2pt); 
\node[left] at (3,-2.2) {$\sigma_{3}$};

\fill (-1,-1) circle (2pt);
\node[left] at (-1,-1.2) {$\sigma_{4}$};
\fill (1,-1) circle (2pt);
\node[left] at (1,-1.2) {$\sigma_{5}$};
\fill (2,-1) circle (2pt);
\node[left] at (2,-1.2) {$\sigma_{6}$};

\fill (-1,1) circle (2pt);
\node[left] at (-1,0.8) {$\sigma_{7}$};

\fill (-1,2) circle (2pt);
\node[left] at (-1,1.8) {$\sigma_{8}$};
\end{scope}
\end{tikzpicture}
\caption{Ordering the vertices of the $x$-forgetful graph induced from $\mathcal{L}_{2,1,2}$.
This is equivalent to choose shelling of $\Delta_{\mathcal{L}}$.}
\label{Figure_shelling}
\end{figure}

In order to find the module generators of $H^{*}(\widetilde{\mathcal{G}})$, it is enough to find the minimal vertices of $\Delta_{i}\setminus \Delta_{i-1}$, where $\Delta_{i}$ is the subcomplex generated by $\sigma_{1},\ldots, \sigma_{i}$ in $\Delta_{\mathcal{L}}$.
For example, as a set $\Delta_{1}=[u_{1},v_{1}]=\{u_{1},v_{1},\sigma_{1}\}$ 
and $\Delta_{2}=[u_{1},v_{1}]\cup [u_{1},v_{2}]=\{u_{1},v_{1},v_{2},\sigma_{1},\sigma_{2}\}$;
therefore, $\Delta_{2}\setminus \Delta_{1}=\{v_{2}, \sigma_{2}\}$ such that $v_{2}\prec \sigma_{2}$ and the minimal face is $\mu_{2}:=v_{2}$.
Similarly, we obtain the following shelling:
\begin{align*}&\mu_1=\emptyset,
    \mu_2=\{v_2\}, \mu_3=\{v_3\},\ldots, \mu_l=\{ v_l\}, \\ &
    \mu_{l+1}=\{w_1\}, \mu_{l+2}=\{ w_2\},\ldots,
                                                                \mu_{l+m}=\{w_m\},\\
                                               & \mu_{l+m+1}=\{u_2\},
                                                 \mu_{l+m+2}=[u_2,v_2]\ldots,
                                                 \mu_{2l+m}=[u_2,v_l],\\
                                               &
                                                 \mu_{2l+m+1}=[u_2,w_1],\mu_{2l+m+2}=[u_2,w_2]
                                                 \ldots,
                                                 \mu_{2l+2m}=[u_2,w_m],\\
                                               &
                                                 \mu_{2l+2m+1}=\{u_3\},\mu_{2l+2m+2}=[u_3,v_2],\ldots,
                                                 \mu_{3l+2m}=[u_3,v_l],\\
                                               &\mu_{3l+2m+1}=[u_3,
                                                 w_1],
                                                 \mu_{3l+2m+2}=[u_3,
                                                 w_2],\ldots,
                                                 \mu_{3l+3m}=[u_3,
                                                 w_m],\\ &\hspace{6cm}
                                                          \vdots\\ &
                                                                     \mu_{(k-1)\cdot
                                                                     l+(k-1)\cdot
                                                                     m+1
                                                                     }=\{u_k\},
                                                                     \mu_{(k-1)\cdot
                                                                     l+(k-1)\cdot
                                                                     m+2}=[u_k,v_2],\ldots,
                                                                     \mu_{kl+(k-1)\cdot
                                                                     m}=[u_k,v_l],\\
                                               & \mu_{kl+(k-1)\cdot m
                                                 +1}=[u_{k},w_1],\mu_{kl+(k-1)\cdot
                                                 m+2}=[u_{k},w_2]\ldots,\mu_{kl+km}=[u_{k},w_{m}]\\
                                               &
                                                 \mu_{kl+km+1}=[v_1,w_1],
                                                 \mu_{kl+km+2}=[v_1,w_2],
                                                 \ldots,\mu_{kl+km+m
                                                 }=[v_1,w_m],\\ &
                                                                 \hspace{6cm}
                                                                 \vdots\\
                                               &
                                                 \mu_{kl+km+(l-1)\cdot
                                                 m+1 }=[v_l,
                                                 w_1],\mu_{kl+km+(l-1)\cdot
                                                 m+2}=[v_l,w_2],\ldots,
                                                 \mu_{kl+km+lm}=[v_l,w_m].
                                 \end{align*}

                                 By Theorem \ref{module-generators} the
                                 monomial basis for
                                 $\mathcal{R}/\mathcal{I}$ as a
                                 $H^*(BT^2)$-module is as follows
\begin{align*}
&1, Y_2,\ldots, Y_l, Z_1,\ldots, Z_m, X_2,\ldots, X_k, \\ &
                                                            X_2Y_2,\ldots,
                                                                  X_2Y_l,
                                                            X_2Z_1,\ldots,
  X_2Z_m,\\ &\hspace{3cm}\vdots\\ & X_kY_2,\ldots, X_kY_l,
                                    X_kZ_1,\ldots, X_kZ_m, \\ &
                                                               Y_1Z_1,\ldots,
                                                               Y_1Z_m,\ldots, Y_lZ_1,\ldots, Y_lZ_m
\end{align*}

For example, in the case when $k=2$, $l=1$ and $m=2$, the equivariant
cohomology ring $H^{*}(\widetilde{\mathcal{G}})$ of
$\widetilde{\mathcal{G}}$ in Figure~\ref{Figure_shelling} is
isomorphic as $H^{*}(BT^{2})$-algebra to
\[
\frac{H^{*}(BT^{2})[X_{1},X_{2},Y_{1},Z_{1},Z_{2}]}{\langle X_{1}X_{2},
  Z_{1}Z_{2}, X_{1}Y_{1}Z_{1}, X_{1}Y_{1}Z_{2}, X_{2}Y_{1}Z_{1},
  X_{2}Y_{1}Z_{2} ; X_1+X_{2}-Z_1-Z_2-e_1^*,
 Y_1-Z_1-Z_2-e_2^*.\rangle}\]

The shelling order of
$\Delta_{\mathcal{L}_{2,1,2}}$ is given by

\begin{align*}
 &\sigma_1=[u_1,v_1]<\sigma_2=[u_1,w_1]<\sigma_{3}=[u_1,
                                             w_2]<
                                                      \sigma_{4}=[u_2,v_1]\\
  & <\sigma_{5}=[u_2,w_1]< \sigma_{6}=[u_2,w_2] <\sigma_{7}=[v_1,w_1]<\sigma_{8}=[v_1,w_2].\end{align*}

Here we have 

\begin{align*}
  &\mu_{1}=\emptyset,\ 
                         \mu_{2}=\{w_1\},\ \mu_{3}=\{
                                             w_2\},\  
                                                      \mu_{4}=\{u_2\},\\
  & \mu_{5}=[u_2,w_1],\ \mu_{6}=[u_2,w_2], \mu_{7}=[v_1,w_1],\ 
                                                \mu_{8}=[v_1,w_2].\end{align*}

                                              From Theorem
                                              \ref{module-generators}(iii)
                                              we have the following
                                              basis of
                                              $H^*(\widetilde{\mathcal{G}})\simeq
                                              \mathbb{Z}[\widetilde{\mathcal{G}}]\simeq
                                              \mathcal{R}/\mathcal{I}$
                                              as a free
                                              $H^{*}(BT^{2})$-module:
\begin{equation}\label{basisexample}
x_{\mu_1}=1, x_{\mu_2}=Z_{1}, x_{\mu_3}=Z_{2}, x_{\mu_4}=X_{2}, x_{\mu_5}=X_{2}Z_{1}, x_{\mu_6}=X_{2}Z_{2}, x_{\mu_7}=Y_{1}Z_{1}, x_{\mu_8}=Y_{1}Z_{2}.\end{equation}

We shall now apply Theorem \ref{module-generators} (iv) to determine
some of the multiplicative structure constants of the basis
(\ref{basisexample}). We first consider $Z_1^2$. Let
$\displaystyle Z_1^2=\sum_{i=1}^8 a_i\cdot x_{\mu_i}$. Note first
that $\rho_{\mathbf{p}_{\sigma_i}}(Z_1^2)=0$ for $i=1,3,4,6,8$ so that
$a_1=a_3=a_4=a_6=a_8=0$.  We further see that
$\rho_{\mathbf{p}_{\sigma_2}}(Z_1^2)=Z_1^2$ in
$H_{T^2}^*(x_{\sigma_2})\simeq \mathbb{Z}[X_1,Z_1]$. Also
$x_{\mu_2}=Z_1$ and $\rho_{\mathbf{p}_{\sigma_2}}(Z_1)=Z_1$. Thus
$a_{2}=\frac{\rho_{\mathbf{p}_{\sigma_2}}(Z_1^2)}{\rho_{\mathbf{p}_{\sigma_1}}(Z_1)}=Z_1$
which corresponds to the element $-e_2^*$ under the isomorphism
$H_{T^2}^*(x_{\sigma_1})\simeq H^*(BT^2)$. Thus
$a_2=-e_2^*$. Proceeding as in Theorem \ref{module-generators}(iv) we
next consider $Z_1^2+e_2^*\cdot Z_1$. Using the relation
$e_2^*=Y_1-Z_1-Z_2$ and $Z_1\cdot Z_2=0$ in $\mathcal{R}/\mathcal{I}$
we get
$Z_1^2+e_2^*\cdot Z_1= Z_1^2+(Y_1-Z_1-Z_2)\cdot
Z_1=Y_1Z_1-Z_1Z_2=Y_1Z_1=x_{\mu_7}$. Thus
$Z_1^2=-e_2^*\cdot Z_1+Y_1Z_1=-e_2^*\cdot x_{\mu_2}+x_{\mu_7}$.

Next we consider $X_2^2$. If
$\displaystyle X_2^2=\sum_{i=1}^8 a_i\cdot x_{\mu_i}$ then
$\rho_{\mathbf{p}_{\sigma_i}}(X_2^2)=0$ for $i=1,2,3,7,8$. Thus by
Theorem \ref{module-generators}(iv) we get $a_i=0$ for $i=1,2,3,7,8$. To
find $a_4, a_5, a_6$ we first apply
$\rho_{\mathbf{p}_{\sigma_4}}(X_2^2)=X_2^2$ in
$H^*_{T^2}(x_{\sigma_4})=\mathbb{Z}[X_2, Y_1]$. Since $x_{\mu_4}=X_2$
we get
$a_4=\frac{\rho_{\mathbf{p}_{\sigma_4}}(X_2^2)}{\rho_{\mathbf{p}_{\sigma_4}}(X_2)}=X_2=e_1^*$
under the isomorphism $H^*_{T^2}(x_{\sigma_4})\simeq H^*(BT^2)$. We
then consider
$X_2^2-e_1^*\cdot X_2=X_2^2-(X_1+X_2-Z_1-Z_2)\cdot X_2=Z_1X_2+Z_2X_2$
using the relations $e_1^*=X_1+X_2-Z_1-Z_2$ and $X_1X_2=0$ in
$\mathcal{R}/\mathcal{I}$. This implies from Theorem
\ref{module-generators} that $a_5=a_6=1$ so that
$X_2^2=e_1^*\cdot x_{\mu_4}+x_{\mu_5}+x_{\mu_6}$.

Using similar arguments we have the following in the
$H^{*}(BT^{2})$-algebra, $\mathcal{R}/\mathcal{I}$:
\begin{align*}
& X_{2}^{2}=e_1^*\cdot X_2+1\cdot X_{2}Z_{1}+ 1\cdot X_{2}Z_{2}; \\
& X_{2}Z_{1}=1\cdot X_{2}Z_{1}; \\
& X_{2}Z_{2}=1\cdot X_{2}Z_{2}; \\
& Z_{1}^{2}=-e_2^*\cdot Z_{1}+1\cdot Y_{1}Z_{1}; \\
& Z_{1}Z_{2}=0; \\
& Z_{2}^{2}=-e_2^*\cdot Z_2+1\cdot Y_{1}Z_{2}.
\end{align*}

By Corollary \ref{ord-cohomology} the ordinary cohomology ring
$H^{*}_{ord}(\widetilde{\mathcal{G}})$ is isomorphic to
\[R'/I'\simeq \mathcal{R}/\mathcal{I}\otimes_{H^*(BT^2)}\mathbb{Z}\]
where $\mathbb{Z}$ is viewed as a
$H^*(BT^2)=\mathbb{Z}[e_1^*,e_2^*]$-module via the augmentation map
which sends $e_i^*$ to $0$ for $i=1,2$.  Hence
$R'=\mathbb{Z}[X_{1},X_{2},Y_{1},Z_{1},Z_{2}]$ and
\[I'=\langle X_{1}X_{2}, Z_{1}Z_{2}, X_{1}Y_{1}Z_{1}, X_{1}Y_{1}Z_{2},
  X_{2}Y_{1}Z_{1}, X_{2}Y_{1}Z_{2},
X_{1}+X_{2}-Z_{1}-Z_{2}, Y_{1}-Z_{1}-Z_{2} \rangle.\] By Corollary
\ref{ord-cohomology}(ii) and (\ref{basisexample}) we see that
$H^{*}_{ord}(\widetilde{\mathcal{G}})$ is isomorphic as a graded
$\mathbb{Z}$-module to
\begin{align*}
\mathbb{Z}\oplus \mathbb{Z}Z_{1}\oplus \mathbb{Z}Z_{2}\oplus \mathbb{Z}X_{2}\oplus \mathbb{Z}X_{2}Z_{1}\oplus \mathbb{Z}X_{2}Z_{2}\oplus\mathbb{Z}Y_{1}Z_{1}\oplus \mathbb{Z}Y_{1}Z_{2}.
\end{align*}
It therefore follows that the Euler characteristic of $M$ is
$kl+km+lm=8$ which is the number of elements in the monomial
basis. Also the rank of $H _{ord}^2(M;\mathbb{Z})$ is $k+l+m-2=3$
which are the number of monomials of degree $1$ in the basis. By Theorem
\ref{module-generators}, using the fact that every $u\in H^*(BT^2)$ is
equated to zero in the graded ring $R'/I'$ it follows that every
monomial of degree $r$ is a linear combination of the monomials
$x_{\mu_i}$ of degree greater than or equal to $r$ (see
\ref{e'2}). Since there are no monomials $x_{\mu_i}$ of degree
greater than or equal to $3$ we get that
$H^{2n}_{ord} (\mathcal{G})=0$ for all $n\geq 3$. This also implies in
particular that $H _{ord}^6(M;\mathbb{Z})=0$ and
$H _{ord}^8(M;\mathbb{Z})=0$. Using this fact or directly observing
that there are $4$ monomial basis elements of degree $2$ it follows
that $\mbox{rank}((H _{ord}^4(M;\mathbb{Z}))=kl+km+lm-(k+l+m-2)-1=4$.

Further, the multiplicative structure constants for the basis
\[1, Z_{1}, Z_{2}, X_{2}, X_{2}Z_{1}, X_{2}Z_{2}, Y_{1}Z_{1},
  Y_{1}Z_{2}\] of $R'/I'$ can be derived to be as follows:
\begin{align*}
& X_{2}^{2}=1\cdot X_{2}Z_{1}+ 1\cdot X_{2}Z_{2}; \\
& X_{2}Z_{1}=1\cdot X_{2}Z_{1}; \\
& X_{2}Z_{2}=1\cdot X_{2}Z_{2}; \\
& Z_{1}^{2}=1\cdot Y_{1}Z_{1}; \\
& Z_{1}Z_{2}=0; \\
& Z_{2}^{2}=1\cdot Y_{1}Z_{2}.
\end{align*}

More generally, we can compute the multiplicative structure constants
of the $\mathbb{Z}$-algebra
$H_{ord}^*(\widetilde{\mathcal{G}})\simeq R'/I'$ for the $x$-forgetful
graph $\widetilde{\mathcal{G}}$ induced from $\mathcal{L}_{k,l,m}$
with repect to the basis
\begin{align}
\label{8-dim_basis}
&1, Y_2,\ldots, Y_l, Z_1,\ldots, Z_m, X_2,\ldots, X_k, \nonumber \\ &
                                                            X_2Y_2,\ldots,
                                                                  X_2Y_l,
                                                            X_2Z_1,\ldots,
  X_2Z_m, \nonumber\\ &\hspace{3cm}\vdots\\ & X_kY_2,\ldots, X_kY_l,
                                    X_kZ_1,\ldots, X_kZ_m, \nonumber \\ &
                                                               Y_1Z_1,\ldots,
                                                               Y_1Z_m,\ldots, Y_lZ_1,\ldots, Y_lZ_m. \nonumber
\end{align}
Here again we note that the Euler characteristic of $M$ is $kl+km+lm$
which is the number of elements in the monomial basis. Also the rank
of $H _{ord}^2(M;\mathbb{Z})$ is $k+l+m-2$ which are the number of
monomials of degree $1$ in the basis. By Theorem
\ref{module-generators}, using the fact that every $u\in H^*(BT^2)$ is
equated to zero in the graded ring $R'/I'$ it follows that every
monomial of degree $r$ is a linear combination of the monomials
$x_{\mu_i}$ of degree greater than or equal to $r$ (see
\eqref{e'2}). Since there are no monomials $x_{\mu_i}$ of degree
greater than or equal to $3$ we get that
$H^{2n}_{ord} (\mathcal{G})=0$ for all $n\geq 3$. This implies in
particular that $H _{ord}^6(M;\mathbb{Z})=0$ and
$H _{ord}^8(M;\mathbb{Z})=0$. Thus it suffices to compute the
structure constants when we multiply two monomial basis elements of
degree $1$ which gives us a degree $2$ monomial. This can be done as
follows.

To compute the structure constants of $H_{ord}^{*}(\mathcal{L}_{k,l,m})$ with respect to the basis \eqref{8-dim_basis}, 
firstly we observe by the following steps given in Theorem
\ref{module-generators}(iv) (as explained in detail above for the case
when $k=2$ $l=1$ and $m=2$) that in $\mathcal{R}/\mathcal{I}$ we have
the following relations for $2\leq r\leq k$, $2\leq s\leq l$ and
$1\leq t\leq m$:
\begin{align*}
& X_{r}^{2}=e_1^*\cdot X_r+\sum_{t=1}^m1\cdot X_rZ_t; \\
& Y_{s}^2= e_2^*\cdot Y_s+\sum_{t=1}^m1\cdot Y_sZ_t;\\
& Z_{t}^{2}=-e_2^*\cdot Z_{t}+\sum_{s=1}^l 1\cdot Y_{s}Z_{t}. \\
\end{align*}
Other products of degree $1$ monomials in $\mathcal{R}/\mathcal{I}$
multiply to give square free monomials of degree $2$, $X_rY_s$ or
$X_rZ_t$ or $Y_sZ_t$ which are already part of the basis. Note also
that $X_rX_{r'}=0$, $Y_{s}Y_{s'}=0$ and $Z_{t}Z_{t'}=0$ for
$r\neq r'$, $s\neq s'$ and $t\neq t'$.

We therefore arrive at the following relations in
$R'/I' \simeq H_{ord}^*(\widetilde{\mathcal{G}})$:
\begin{align*}
& X_{r}^{2}=\sum_{t=1}^m1\cdot X_rZ_t; \\
& Y_{s}^2= \sum_{t=1}^m1\cdot Y_sZ_t;\\
  & Z_{t}^{2}=\sum_{s=1}^l 1\cdot Y_{s}Z_{t}; \\
  &  X_rY_s=1\cdot X_r Y_s;\\
  & X_rZ_t=1\cdot X_rZ_t; \\
  & Y_sZ_t=1\cdot Y_sZ_t,
\end{align*}
where $2\leq r\leq k$, $2\leq s\leq l$ and $1\leq t\leq m$.

We have the following corollary of Theorem \ref{module-generators} for
the hyperplane arrangements $\mathcal{L}_{k,l,m}$ classifying the
$8$-dimensional toric hyperK{$\ddot{\mathrm{a}}$}hler manifolds.
\begin{corollary}
The ordinary cohomology $H_{ord}^{*}(\mathcal{L}_{k,l,m})$ is isomorphic to a free $\mathbb{Z}$-module generated by the elements \eqref{8-dim_basis}.
Furthermore, all structure constants of their multiplications are $1$
except for the case when they are equal to $0$.
\end{corollary}

\section*{Acknowledgement}
The first author is grateful to Professors Megumi Harada and Mikiya Masuda for their invaluable advice and comments for the previous version of this paper \cite{Ku10}.

\bibliographystyle{amsalpha}

\end{document}